\documentclass[preprint,3p]{elsarticle}

\usepackage{graphicx}
\usepackage{amssymb}
\usepackage{tikz}
\usetikzlibrary{positioning}
\usepackage{geometry}

\usepackage{comment}

\usepackage{lineno}

\usepackage{amsmath}
\usepackage{amsthm}
\usepackage{epsfig}
\usepackage{circuitikz}
\ctikzset{
    voltage=american, 
    current=american
}
\usepackage{mathrsfs}

\usepackage{psfrag}
\usepackage{color}

\newtheorem{theorem}{Theorem}

\newtheorem{lemma}[theorem]{Lemma}
\newtheorem{example}[theorem]{Example}
\newtheorem{definition}[theorem]{Definition}

\newtheorem{corollary}[theorem]{Corollary}
\newtheorem{remark}{Remark}

\newtheoremstyle{algstyle}%
  {10mm}       % measure of space to leave above the theorem. E.g.: 3pt
  {10mm}       % measure of space to leave below the theorem. E.g.: 3pt
  {\tt}   % name of font to use in the body of the theorem
  {0pt}        % measure of space to indent
  {\bfseries}  % name of head font
  {\newline}   % punctuation between head and body
  {10mm}       % space after theorem head
  {\thmname{#1}\thmnumber{ #2}\thmnote{ (#3)}}          
\theoremstyle{algstyle}

\newtheoremstyle{algdashstyle}%
  {10mm}       % measure of space to leave above the theorem. E.g.: 3pt
  {10mm}       % measure of space to leave below the theorem. E.g.: 3pt
  {\tt}   % name of font to use in the body of the theorem
  {0pt}        % measure of space to indent
  {\bfseries}  % name of head font
  {\newline}   % punctuation between head and body
  {10mm}       % space after theorem head
  {\thmname{#1}\thmnumber{ #2}$'$\thmnote{ (#3)}}          % Manually specify head
\theoremstyle{algdashstyle}

\begin{document}
\begin{frontmatter}

\title{Free products, principal sums, bipartition complete matroids and the linking operation $(\mathcal{M}_{SP} \vee \mathcal{M}_{PQ}) \times (S \uplus Q)$}

\author{H. Narayanan}
\ead{hn@ee.iitb.ac.in}
\address{Department of Electrical Engineering, Indian Institute of Technology Bombay, Mumbai, India}
\begin{abstract}
In this paper we examine the special class of bipartition complete matroids, which includes free products and principal sums of matroids (Crapo and Schmitt, 2005; Bonin and Kung, 2015). A matroid $\mathcal{M}_{SQ}$ on $S \uplus Q$ is said to be $\{S,Q\}$--complete iff whenever $b^1_S \uplus b^1_Q, \, b^1_S \uplus b^2_Q,$ and $b^2_S \uplus b^1_Q$ are bases of $\mathcal{M}_{SQ},$ we also have $b^2_S \uplus b^2_Q$ as a base of $\mathcal{M}_{SQ}.$
We use the linking operation $\mathcal{M}_{SP} \leftrightarrow_{\mathcal{M}} \mathcal{M}_{PQ} :\equiv (\mathcal{M}_{SP} \vee \mathcal{M}_{PQ}) \times (S \uplus Q)$ (Theja and Narayanan, 2014) for this study. The linking operation ``$\leftrightarrow_{\mathcal{M}}$'' for matroids is a generalization of the composition of electrical multiports along a common set of ports (Narayanan, 1997).
We construct the $\{S, Q\}$--augmentation of an arbitrary matroid and show that this augmentation is again a matroid,  that the dual of the augmentation is the augmentation of the dual matroid and that the $\{S, Q\}$--augmentation of an $\{S, Q\}$--complete  matroid leaves it unchanged. We also show that bipartition completeness is stable under minors in the appropriate bipartitioned sense. We then prove that an $\{S, Q\}$--complete matroid $ \mathcal{M}_{SQ} :\equiv \mathcal{M}_{SP} \leftrightarrow_{\mathcal{M}} \mathcal{M}_{PQ}$ that satisfies a simple condition of ``strong compatibility'' between its factors, admits a minimal linking decomposition $\mathcal{M}_{SQ} = \mathcal{M}_{S\hat{P}} \leftrightarrow_{\mathcal{M}} \mathcal{M}_{\hat{P}Q}$ with $|\hat{P}| = r(\mathcal{M}_{SQ} \circ S) - r(\mathcal{M}_{SQ} \times S),$ and with $\mathcal{M}_{S\hat{P}}, \mathcal{M}_{\hat{P}Q}$ respectively $\{S, \hat{P}\}-$ complete,  $\{\hat{P}, Q\}-$ complete. Finally, we show that free products and principal sums of matroids $\mathcal{M}_{S},\mathcal{M}_{Q}$ are $\{S, Q\}$--complete and give explicit minimal decompositions for them.
\end{abstract}

\end{frontmatter}
\section{Introduction}
\label{sec:intro}
The operations of composition and decomposition are commonly used to study
larger systems in terms of their ``building blocks.''
In electrical network theory, multiport composition/decomposition is a standard method for analysis and synthesis
(see, for instance, \cite{belevitch68}, \cite{HNarayanan1997}).
This can be generalized to vector spaces as follows.
We may regard a vector $f_S$ as a function from an index set $S$ to a field $F.$ Equivalently, vectors can be regarded as ``row vectors'' with entries from a field, with the columns being elements of the index set.
A vector space $\mathcal{V}_{SP}$ is a collection of vectors $f_{SP}: S \uplus P \rightarrow F$ (where we denote disjoint union by $\uplus$), closed under addition and scalar multiplication.
We define $\mathcal{V}_{SP} \leftrightarrow_{\mathcal{V}} \mathcal{V}_{PQ}$ (where $S \cap Q = \emptyset$) to be the collection of all vectors $(f_S, f_Q)$
such that $(f_S, f_P) \in \mathcal{V}_{SP}$ and $(f_P, f_Q) \in \mathcal{V}_{PQ}$ (\cite{HNarayanan1997}).
For matroids, there is the analogous notion of ``linking.'' We define $\mathcal{M}_{SP} \leftrightarrow_{\mathcal{M}} \mathcal{M}_{PQ}
:\equiv (\mathcal{M}_{SP} \vee \mathcal{M}_{PQ}) \times (S \uplus Q)$ (\cite{STHN2014}), denoting
the contraction operation by $\times$.

In this paper, we use the notion of linking of matroids to study the class of matroids which we call bipartition complete matroids.
A matroid $\mathcal{M}_{SQ}$ on $S \uplus Q$ is said to be $\{S,Q\}$--complete iff
whenever $b^1_S \uplus b^1_Q$, $b^1_S \uplus b^2_Q$, and $b^2_S \uplus b^1_Q$ are bases of $\mathcal{M}_{SQ}$, we also have $b^2_S \uplus b^2_Q$ as a base of $\mathcal{M}_{SQ}.$
We show that the free product and principal sum of matroids (\cite{crapofree}, \cite{boninsemidirect})
are bipartition complete in the appropriate sense.

There is a simple expression for the minimum size of $P'$ such that $\mathcal{V}_{SP} \leftrightarrow_{\mathcal{V}} \mathcal{V}_{PQ} = \mathcal{V}_{SP'} \leftrightarrow_{\mathcal{V}} \mathcal{V}_{P'Q}$: 
$|P'| \geq r((\mathcal{V}_{SP} \leftrightarrow_{\mathcal{V}} \mathcal{V}_{PQ}) \circ S) - r((\mathcal{V}_{SP} \leftrightarrow_{\mathcal{V}} \mathcal{V}_{PQ}) \times S),$ 
where $\circ$ and $\times$ denote respectively the operations of restriction and contraction (\cite{HNarayanan1997}).
For matroids, there is an analogous expression for the lower bound on the size of $|P'|$: 
$|P'| \geq r((\mathcal{M}_{SP} \leftrightarrow_{\mathcal{M}} \mathcal{M}_{PQ}) \circ S) - r((\mathcal{M}_{SP} \leftrightarrow_{\mathcal{M}} \mathcal{M}_{PQ}) \times S).$
We give sufficient conditions under which the minimum size $|P'|$ can be achieved.
These sufficient conditions are satisfied by all bipartition complete matroids and therefore also by free products and principal sums of matroids.
We give a simple construction for the minimal decomposition of a matroid $\mathcal{M}_{SQ}$ when the sufficient conditions are satisfied.

The outline of the paper follows.

Section~\ref{sec:Preliminaries} fixes the notation used in the paper for vector spaces and matroids.

Section~\ref{subsec:complete} provides a brief description of bipartition complete matroids with examples.

Section~\ref{subsec:linking} details the linking operation for matroids. The general minimization problem for $\mathcal{M}_{SP} \leftrightarrow_{\mathcal{M}} \mathcal{M}_{PQ}$ is discussed using only contraction and restriction of the interacting matroids. Strongly compatible matroids are defined. Given strongly compatible matroids $\mathcal{M}_{SP}$ and $\mathcal{M}_{PQ},$ an algorithm for the solution to the minimization problem is constructed.

Section~\ref{subsec:bipartitioncompletion} gives an algebraic characterization of bipartition complete matroids. An algorithm for minimally decomposing a bipartition complete matroid $\mathcal{M}_{SQ}$ into $\mathcal{M}_{SP} \leftrightarrow_{\mathcal{M}} \mathcal{M}_{PQ}$ is also provided.

Section~\ref{sec:free} explores the free product and principal sums of matroids. These are studied as examples of bipartition complete matroids and evaluated in terms of their minimal decompositions.

The appendix contains a list of basic matroid results and proofs of some of the results of the main paper.

\section{Preliminaries}
\label{sec:Preliminaries}
The preliminary results and the notation used are from \cite{HNarayanan1997}.

The size of a set $X$ is denoted by $\boldsymbol{|X|}.$
We write $\boldsymbol{X\uplus Y}$ in place of $X\cup Y$ when $X$, $Y$ are disjoint,
 $X-Y$ in place of $X \setminus Y.$

We denote the vertex set and edge set of a graph $\mathcal{G}$
 by $V(\mathcal{G}), E(\mathcal{G}),$ respectively.
 We assume the reader is familiar with terms such as circuits, bonds (minimal  set of edges which when deleted increases the number of connected pieces of the graph), trees etc. 
In the case of a disconnected graph, a spanning forest is a union of spanning trees of the connected pieces of the graph.
  The {\bf rank} of a graph ${\cal G},$ denoted $r({\cal G}),$ is the size of a spanning forest of ${\cal G}.$ 

We define \textbf{copies of graphs} as follows.
 A graph ${\cal G}'$ is a copy of the graph ${\cal G}$ iff there are bijections $V({\cal G})\rightarrow V({\cal G}'), $ $E({\cal G})\rightarrow E({\cal G}'), $ with vertex $v$  mapping to vertex $v'$ and edge $e$  mapping to edge $e'$ such that if $e$ lies between $v_1,v_2$ in ${\cal G},$ then $e'$ lies between $v'_1,v'_2$ in ${\cal G}'.$

\subsection{Matroids}
A \textbf{matroid} $\mathcal{M}_S$ on a finite set $S$ is an ordered pair $(S, \mathcal{B})$, where $\mathcal{B}$ is a non-empty collection of subsets of $S$, called \textbf{bases}, satisfying the following base axiom:
\begin{itemize}
    \item If $b_1, b_2 \in \mathcal{B}$ and $e_2 \in b_2 - b_1$, then there exists $e_1 \in b_1 - b_2$ such that $(b_1 - \{e_1\}) \cup \{e_2\} \in \mathcal{B}$.
\end{itemize}

It follows from the base axiom that no base can contain another and that all bases of a matroid have the same size. Subsets of bases of a matroid are called \textbf{independent sets} of the matroid, and supersets of bases are called \textbf{spanning sets} of the matroid. Maximal independent sets contained in a given subset can be shown to have the same size.

The \textbf{rank function} $r: 2^S \rightarrow \mathbb{Z}_{\geq 0}$, corresponding to a matroid $\mathcal{M}_S$, is defined to be the size of a maximal independent set contained in the given subset. The rank of the matroid, denoted $r(\mathcal{M}_S)$, is defined to be $r(S)$. The complement of a base is called a \textbf{cobase}.

A matroid on $S$ with only the empty set as a base is called the \textbf{zero matroid}, denoted by $\mathbf{0}_S$. A matroid on $S$ with only the full set $S$ as a base is called the \textbf{full matroid}, denoted by $\mathcal{F}_S$. (Note that we use $\mathbf{0}_S$ to denote the zero vector space as well as the zero matroid, and $\mathcal{F}_S$ to denote the full vector space as well as the full matroid. This abuse of notation will not cause confusion as context clarifies the intended entity.) The \textbf{uniform matroid}, denoted $U_{x,y},$ is a matroid on an underlying set with $y$ elements, with every subset  of the underlying set of $x$ elements, being a base.  

We denote a matroid on $S \uplus P$ by $\mathcal{M}_{SP}$ or by $\mathcal{M}_{PS}$. Given a matroid $\mathcal{M}_S$ on $S$ and a matroid $\mathcal{M}_P$ on $P$, with $S$ and $P$ disjoint, the \textbf{direct sum} of $\mathcal{M}_S$ and $\mathcal{M}_P$ is a matroid on $S \uplus P$, denoted by $\mathcal{M}_S \oplus \mathcal{M}_P$, whose bases are unions of a base of $\mathcal{M}_S$ with a base of $\mathcal{M}_P$.

\subsection{Dual Matroids and Copies}
\label{subsec:copies}

Given a matroid $\mathcal{M} := (S, \mathcal{B})$, its \textbf{dual} matroid, denoted by $\mathcal{M}^*$, is defined to be $(S, \mathcal{B}^*)$, where $\mathcal{B}^*$ is the collection of complements of the subsets present in $\mathcal{B}$. We can see that the dual of the dual is equal to the original matroid ($\mathcal{M}^{**} = \mathcal{M}$).

We define \textbf{copies of matroids} as follows:
Let $S, S'$ be disjoint copies of each other with $e' \in S'$ corresponding to $e \in S$. By $(\mathcal{M}_S)_{S'}$, we mean the matroid $\mathcal{M}_{S'}$ on $S'$ where $X' \subseteq S'$ is independent iff  its copy $X \subseteq S$ is independent in $\mathcal{M}_S$.

Let $S, S'$ and $P, P'$ be copies with the sets $S, S', P, P'$ being pairwise disjoint. By $(\mathcal{M}_{SP})_{S'P'}$, we mean the matroid on $S' \uplus P'$, where $X' \uplus Y'$ ($X' \subseteq S', Y' \subseteq P'$) is independent iff  its copy $X \uplus Y$ is independent in $\mathcal{M}_{SP}$. In particular, $(\mathcal{M}_{SS'})_{S'S}$ is the matroid where $Y \uplus X'$ ($Y \subseteq S, X' \subseteq S'$) is independent iff  its copy $X \uplus Y'$ is independent in $\mathcal{M}_{SS'}$.

We note that:
\[
[(\mathcal{M}_{SP})_{S'P'}]^* = (\mathcal{M}^*_{SP})_{S'P'}, \quad \text{and} \quad [(\mathcal{M}_{SS'})_{S'S}]^* = (\mathcal{M}^*_{SS'})_{S'S}
\]
\begin{remark}
Note that both $\mathcal{M}_{SS'}$ and $\mathcal{M}_{S'S}$ denote the 
 same matroid. However in $(\mathcal{M}_{SS'})_{S'S},$ the order of the index sets play a role. We start with a matroid of $S\uplus S'$ which could be 
 represented equivalently by $\mathcal{M}_{SS'}$ or by $\mathcal{M}_{S'S}$ and then flip $S,S'.$ The result could be represented equivalently by  $(\mathcal{M}_{S'S})_{SS'}$ or by $(\mathcal{M}_{SS'})_{S'S}.$ 
\end{remark}
\subsection{Circuits, Restrictions, Contractions, and Minors}

A \textbf{circuit} of a matroid $\mathcal{M}$ is a minimal dependent (not independent) set in $\mathcal{M}$. Circuits of $\mathcal{M}^*$ are called \textbf{bonds} of $\mathcal{M}$.

For a matroid $\mathcal{M}_S$, we define the \textbf{restriction} of $\mathcal{M}_S$ to $T \subseteq S$, denoted by $\mathcal{M}_S \circ T$, as a matroid whose independent sets are all subsets of $T$ which are independent in $\mathcal{M}_S$ (equivalently, bases in $\mathcal{M}_S \circ T$ are maximal intersections of bases of $\mathcal{M}_S$ with $T$). 

The \textbf{contraction} of $\mathcal{M}_S$ to $T$, denoted by $\mathcal{M}_S \times T$, is defined to be the matroid whose independent sets are precisely those $X \subseteq T$ which satisfy the property that $X \cup b_{S-T}$ is independent in $\mathcal{M}_S$ whenever $b_{S-T}$ is a base of $\mathcal{M}_S \circ (S-T)$ (equivalently, bases of $\mathcal{M}_S \times T$ are minimal intersections of bases of $\mathcal{M}_S$ with $T$).

(We use $\mathcal{M}_S \circ T$ and $\mathcal{M}_S \times T$ in place of the usual $\mathcal{M}_S \setminus (S-T)$ and $\mathcal{M}_S / (S-T)$ for better readability of an expression like $\mathcal{M}_S \setminus A / B$, which would now read $\mathcal{M}_S \circ (S-A) \times (S-(A\cup B))$.)

It is immediate from the definitions of restriction, contraction, and rank that:
\[
r(\mathcal{M}_S) = r(\mathcal{M}_S \times T) + r(\mathcal{M}_S \circ (S-T))
\]

A \textbf{minor} of $\mathcal{M}_S$ is a matroid of the form $(\mathcal{M}_S \times T_1) \circ T_2$ or $(\mathcal{M}_S \circ T_1) \times T_2$, where $T_2 \subseteq T_1 \subseteq S$. We usually omit the brackets when we speak of minors of matroids (e.g., $\mathcal{M}_S \times T_1 \circ T_2$ in place of $(\mathcal{M}_S \times T_1) \circ T_2$).

A basic result about minors is that when some elements are deleted and some others contracted, what matters is only whether an element is deleted or contracted and not the order in which these operations are carried out. Therefore:
\[
\mathcal{M}_S \times T_1 \circ T_2 = \mathcal{M}_S \circ (S - (T_1 - T_2)) \times T_2, \quad \text{where } T_2 \subseteq T_1 \subseteq S
\]

It is clear from the definitions of restriction and contraction of matroids that:
\[
(\mathcal{M}_S \times T_1 \circ T_2)^* = \mathcal{M}^*_S \circ T_1 \times T_2
\]

\subsection{Vector Space, Graphic Representations, and Matroid Operations}

We associate matroids with vector spaces and graphs as follows:

Let the vector space $\mathcal{V}_S$ be the row space generated by linear combinations of the rows of a matrix $M_S$ over a field $\mathbb{F}$, with $S$ as its set of columns. A column base of $\mathcal{V}_S$ is a maximal linearly independent set of columns of $M_S$, and $\mathcal{M}(\mathcal{V}_S)$ is defined to be the matroid whose bases are the column bases of $\mathcal{V}_S$. Such a matroid is said to be \textbf{representable over} $\mathbb{F}$.

It is clear that $(\mathcal{M}(\mathcal{V}_S))^* = \mathcal{M}(\mathcal{V}_S^\perp)$, since $(I \mid K)$ is a representative matrix for $\mathcal{V}_S$ iff  $(-K^T \mid I)$ is a representative matrix for $\mathcal{V}_S^\perp$. 

If $\mathcal{G}_S$ is a graph with edge set $S$, $\mathcal{M}(\mathcal{G}_S)$ is defined to be the matroid whose bases are spanning forests (maximal circuit-free sets of edges) of $\mathcal{G}_S$. Such a matroid is said to be \textbf{graphic}.

Let $\mathcal{M}_S^1$ and $\mathcal{M}_S^2$ be matroids on $S$. The \textbf{union} of these matroids, denoted by $\mathcal{M}_S^1 \vee \mathcal{M}_S^2$, is defined to be $(S, \mathcal{B}_{\vee})$, where $\mathcal{B}_{\vee}$ is the collection of maximal sets of the form $b^1 \cup b^2$, where $b^1$ is a base of $\mathcal{M}_S^1$ and $b^2$ is a base of $\mathcal{M}_S^2$. It can be shown that $\mathcal{M}_S^1 \vee \mathcal{M}_S^2$ is again a matroid.

The \textbf{intersection} of matroids $\mathcal{M}_S^1$ and $\mathcal{M}_S^2$, denoted by $\mathcal{M}_S^1 \wedge \mathcal{M}_S^2$, is defined to be the matroid $(S, \mathcal{B}_{\wedge})$ whose bases are minimal sets of the form $b^1 \cap b^2$, where $b^1$ is a base of $\mathcal{M}_S^1$ and $b^2$ is a base of $\mathcal{M}_S^2$. Matroid union is related to intersection through dualization:
\[
(\mathcal{M}_S^1 \vee \mathcal{M}_S^2)^* = (\mathcal{M}_S^1)^* \wedge (\mathcal{M}_S^2)^*.
\]

It is immediate from the definitions of matroid union and intersection that for $T \subseteq S$:
\[
(\mathcal{M}_S^1 \vee \mathcal{M}_S^2) \circ T = (\mathcal{M}_S^1 \circ T) \vee (\mathcal{M}_S^2 \circ T)
\]
\[
(\mathcal{M}_S^1 \wedge \mathcal{M}_S^2) \times T = (\mathcal{M}_S^1 \times T) \wedge (\mathcal{M}_S^2 \times T)
\]

Let $S, P, Q$ be pairwise disjoint sets. The union and intersection operations on $\mathcal{M}_{SP}$ and $\mathcal{M}_{PQ}$ are defined respectively by:
\[
\mathcal{M}_{SP} \vee \mathcal{M}_{PQ} := (\mathcal{M}_{SP} \oplus \mathbf{0}_Q) \vee (\mathcal{M}_{PQ} \oplus \mathbf{0}_S)
\]
\[
\mathcal{M}_{SP} \wedge \mathcal{M}_{PQ} := (\mathcal{M}_{SP} \oplus \mathcal{F}_Q) \wedge (\mathcal{M}_{PQ} \oplus \mathcal{F}_S).
\]

Bases $b_{SP}, b_{PQ}$ of $\mathcal{M}_{SP}, \mathcal{M}_{PQ}$, respectively, such that $b_{SP} \cup b_{PQ}$ is a base of $\mathcal{M}_{SP} \vee \mathcal{M}_{PQ}$, are said to be \textbf{maximally distant}.

For the convenience of the reader, we have listed basic results from matroid theory, which are used in  the present paper, in \ref{subsec:matroidresults}.

\section{Bipartition complete matroids}
\label{subsec:complete}

\begin{figure}[htbp]
\begin{center}
 \includegraphics[width=2.5in]{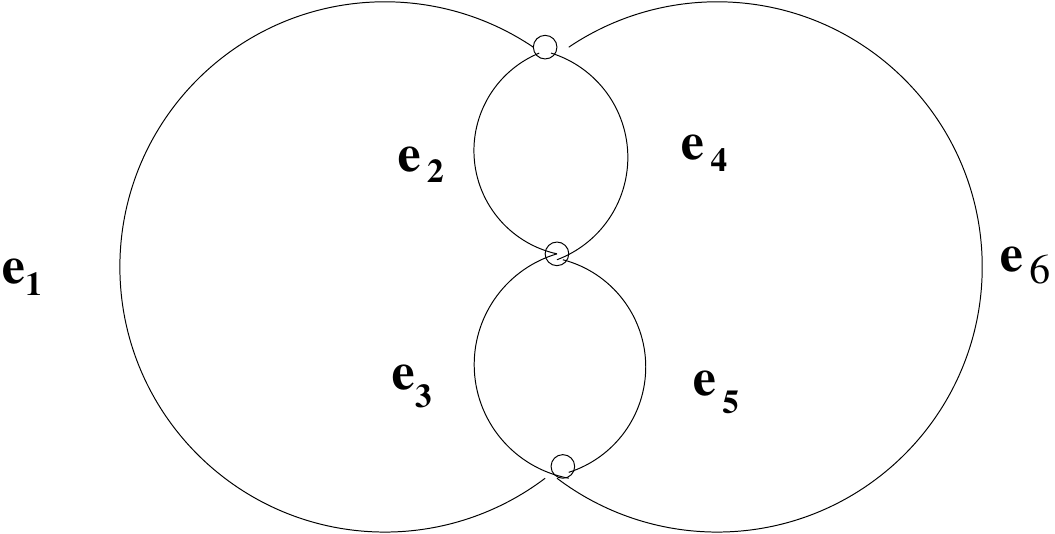}
\caption{Illustration of bipartition completeness}
\label{fig:matroid1}
\end{center}
\end{figure}

In this section, we define bipartition complete matroids.
The idea behind $\{S,Q\}$-completeness is mentioned in passing
as useful in \cite{STHN2014} but not studied further.
Such matroids deserve some attention in their own right,
and instances have occurred in the literature \cite{crapofree},\cite{boninsemidirect}.

\begin{definition}
\label{def:spclosure}
Let $\mathcal{M}_{SQ}$ be a matroid on the ground set $S \uplus Q$.
Let $B_{SQ}$ be the collection of bases of $\mathcal{M}_{SQ}$, and
let $\hat{B}_{SQ}$ be the collection of sets $\hat{b}_{SQ} = \hat{b}_S \uplus \hat{b}_Q$ satisfying the following property:
\begin{quote}
$\hat{b}_{SQ} \in \hat{B}_{SQ}$ iff  there exist bases $b_S \uplus b_Q$, $\hat{b}_S \uplus b_Q$, and $b_S \uplus \hat{b}_Q$ of $\mathcal{M}_{SQ}$, but $\hat{b}_{SQ}$ is not a base of $\mathcal{M}_{SQ}$.
\end{quote}

We define the $\{S,Q\}$-augmentation of $\mathcal{M}_{SQ}$ to be $B_{SQ} \cup \hat{B}_{SQ}$.
If the $\{S,Q\}$-augmentation of $\mathcal{M}_{SQ}$ is exactly $B_{SQ}$, then we say $\mathcal{M}_{SQ}$ is $\{S,Q\}$-complete, or complete with respect to the bipartition $\{S,Q\}$.
\end{definition}
\noindent The following lemma is immediate from the definition of $\{S,Q\}$-completeness of $\mathcal{M}_{SQ}$.

\begin{lemma}
\label{lem:spclosure}
Let $\mathcal{M}_{SQ}$ be a matroid.
\begin{enumerate}
    \item $\mathcal{M}_{SQ}$ is $\{S,Q\}$-complete iff  the dual matroid $\mathcal{M}^*_{SQ}$ is $\{S,Q\}$-complete.
    
    \item Let $\mathcal{M}_{SQ}$ be $\{S,Q\}$-complete. The family of all independent sets of $\mathcal{M}_{SQ} \circ S$ (respectively, $\mathcal{M}_{SQ} \circ Q$) that contain bases of $\mathcal{M}_{SQ} \times S$ (respectively, $\mathcal{M}_{SQ} \times Q$) can be partitioned as $\mathcal{E}_S(\mathcal{M}_{SQ})$ (respectively, $\mathcal{E}_Q(\mathcal{M}_{SQ})$), where:
    \[
    I_S \sim I'_S \quad\Longleftrightarrow\quad \{I_Q : I_S \uplus I_Q \text{ is a base}\} = \{I_Q : I'_S \uplus I_Q \text{ is a base}\}
    \]
    and symmetrically:
    \[
    I_Q \sim I'_Q \quad\Longleftrightarrow\quad \{I_S : I_S \uplus I_Q \text{ is a base}\} = \{I_S : I_S \uplus I'_Q \text{ is a base}\}.
    \]
    Consequently, if $\mathbf{I}_S^j$ is one of the blocks of $\mathcal{E}_S(\mathcal{M}_{SQ})$ and $I_S \in \mathbf{I}_S^j$, then $I_S \uplus I_Q$ is a base of $\mathcal{M}_{SQ}$ iff  $I_Q \in \hat{\mathbf{I}}_Q^j$.
    
    \item The assignment map $\mathbf{I}_S^j \to \hat{\mathbf{I}}_Q^j$ is a bijection.
\end{enumerate}
\end{lemma}

\begin{example}
Consider the matroid $\mathcal{M}_E$ associated with the graph $\mathcal{G}_E$ shown in Figure~\ref{fig:matroid1}, where $E \equiv \{e_1, e_2, e_3, e_4, e_5, e_6\}$. We partition the ground set $E$ in two different ways into $\{S,Q\}$ and $\{A,B\}$ as follows:
\begin{align*}
    S &\equiv \{e_1, e_2, e_4, e_6\}, \quad Q \equiv \{e_3, e_5\}, \\
    A &\equiv \{e_1, e_2, e_3\}, \qquad B \equiv \{e_4, e_5, e_6\}.
\end{align*}
We will show that $\mathcal{M}_E$ is $\{S,Q\}$-complete, but not $\{A,B\}$-complete. 

The bases of $\mathcal{M}_E$ can be explicitly listed as $b^i_S \uplus b^i_Q$:
\begin{align*}
    &\{e_1, e_2\} \uplus \emptyset, \quad \{e_1, e_4\} \uplus \emptyset, \quad \{e_2, e_6\} \uplus \emptyset, \quad \{e_4, e_6\} \uplus \emptyset, \\
    &\{e_1\} \uplus \{e_3\}, \quad \{e_1\} \uplus \{e_5\}, \quad \{e_2\} \uplus \{e_3\}, \quad \{e_2\} \uplus \{e_5\}, \\
    &\{e_4\} \uplus \{e_3\}, \quad \{e_4\} \uplus \{e_5\}, \quad \{e_6\} \uplus \{e_3\}, \quad \{e_6\} \uplus \{e_5\}.
\end{align*}

The resulting equivalence relation blocks for $\mathcal{E}_S(\mathcal{M}_{SQ})$ and $\mathcal{E}_Q(\mathcal{M}_{SQ})$ are:
\begin{align*}
    \mathcal{E}_S(\mathcal{M}_{SQ}) &\equiv \Bigl\{ \bigl\{\{e_1, e_2\}, \{e_1, e_4\}, \{e_2, e_6\}, \{e_4, e_6\}\bigr\}, \, \bigl\{\{e_1\}, \{e_2\}, \{e_4\}, \{e_6\}\bigr\} \Bigr\}, \\
    \mathcal{E}_Q(\mathcal{M}_{SQ}) &\equiv \Bigl\{ \bigl\{\emptyset\bigr\}, \, \bigl\{\{e_3\}, \{e_5\}\bigr\} \Bigr\}.
\end{align*}

To capture $\{S,Q\}$-completeness we may define a bipartite base graph as follows:
Top side vertices are independent sets $b_S$
containing bases of $\mathcal{M}_{SQ} \times S$ and contained in bases of $\mathcal{M}_{SQ} \circ S$.
Similarly
bottom side vertices are independent sets $b_Q$
containing bases of $\mathcal{M}_{SQ} \times Q$ and contained in bases of $\mathcal{M}_{SQ} \circ Q$.
When the matroid $\mathcal{M}_{SQ}$ is bipartition complete, this bipartite graph will appear as a union of disconnected complete bipartite subgraphs (see Figure~\ref{fig:bipartite1}).

Now consider the sets $b^1_A \uplus b^3_B$, $b^1_A \uplus b^4_B$, $b^2_A \uplus b^3_B$, where $b^1_A \equiv \{e_1\}$, $b^2_A \equiv \{e_2\}$, $b^3_B \equiv \{e_5\}$, $b^4_B \equiv \{e_4\}$.
These are all bases of $\mathcal{M}_E$. But $b^2_A \uplus b^4_B = \{e_2,e_4\}$ is not a base of $\mathcal{M}_E$.
\end{example}

A relation theoretic description of bipartition completeness is as follows:  define a bipartite relation between admissible $S$-subsets (containing a base of ${\cal M}_{SQ} \times S$ and contained  in a base of ${\cal M}_{SQ} \circ S$) and admissible $Q$-subsets by
\[  
 I_S\,R\,I_Q :\equiv 
 I_S\uplus I_Q\text{ is a base of }\mbox{$\cal M$}_{SQ}.
\]
The completeness property is exactly the  law
\[
 R\,R^{-1}\,R\subseteq R.
\]
Equivalently, every connected component of the bipartite base graph is complete bipartite.  The equivalence relations $\mbox{$\cal E$}_S$ and $\mbox{$\cal E$}_Q$ are then simply the two shores of these complete bipartite components.

\begin{figure}[htbp] 
    \centering
    \begin{tikzpicture}[
        vertex/.style={circle, draw, minimum size=9mm, inner sep=0.5pt},
        every node/.style={font=\footnotesize}
    ]
        % --- TOP LAYER VERTICES ---
        \node[vertex] (e1) at (0, 0) {$e_1$};
        \node[vertex] (e2) [right=0.45cm of e1] {$e_2$};
        \node[vertex] (e4) [right=0.45cm of e2] {$e_4$};
        \node[vertex] (e6) [right=0.45cm of e4] {$e_6$};
        \node[vertex] (p12) [right=0.45cm of e6] {$\{e_1,e_2\}$};
        \node[vertex] (p14) [right=0.45cm of p12] {$\{e_1,e_4\}$};
        \node[vertex] (p26) [right=0.45cm of p14] {$\{e_2,e_6\}$};
        \node[vertex] (p46) [right=0.45cm of p26] {$\{e_4,e_6\}$};

        % --- BOTTOM LAYER VERTICES ---
        \node[vertex] (r_e3) at (1.2, -3.5) {$e_3$};
        \node[vertex] (r_e5) at (3.8, -3.5) {$e_5$};
        \node[vertex] (r_empty) at (7.8, -3.5) {$\emptyset$};

        % --- EDGES (Top Single Elements to Bottom e3 and e5) ---
        \path[-] (e1) edge (r_e3) edge (r_e5);
        \path[-] (e2) edge (r_e3) edge (r_e5);
        \path[-] (e4) edge (r_e3) edge (r_e5);
        \path[-] (e6) edge (r_e3) edge (r_e5);

        % --- EDGES (Top Pairs to Bottom Empty Set) ---
        \path[-] (p12) edge (r_empty);
        \path[-] (p14) edge (r_empty);
        \path[-] (p26) edge (r_empty);
        \path[-] (p46) edge (r_empty);

    \end{tikzpicture}
    \caption{Bipartite Graph  illustrating $\{S,Q\}-$completeness}
\label{fig:bipartite1}
\end{figure}
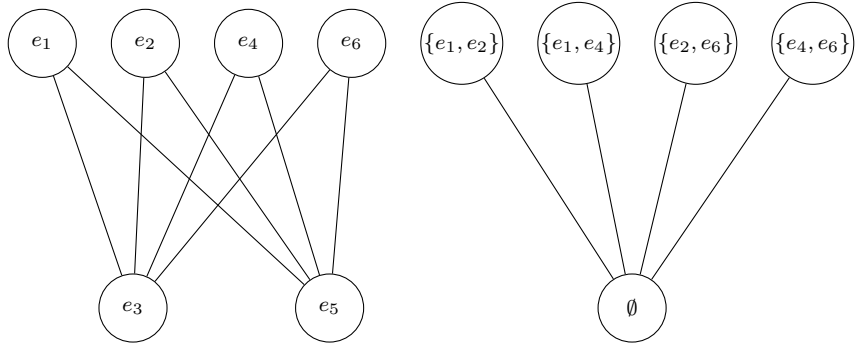
\subsection{Examples of matroids $\mathcal{M}_{SQ}$ which are $\{S,Q\}$-complete}
\label{subsubsec:eg}
We list below, classes of matroids which contain $\{S,Q\}$-complete
matroids.
The first two of the examples would be present in any nontrivial class such as graphic,
representable over a specific field, etc.
\begin{enumerate}
\item The trivial situation where $\mathcal{M}_{SQ} \equiv \mathcal{M}_S \oplus \mathcal{M}_Q$. 
\item The series and parallel combination of graphs $ \mathcal{G}_{SP} , \mathcal{G}_{PQ}$ along an edge $P$ to a graph $\mathcal{G}_{SQ}$ are well known. 
 We can generalize 
this to matroids such that $\mathcal{M}(\mathcal{G}_{SQ}) $ is the series or parallel combination respectively of the matroids $\mathcal{M}(\mathcal{G}_{SP}), \mathcal{M}(\mathcal{G}_{PQ}). $ The matroid $\mathcal{M}(\mathcal{G}_{SQ}) $ 
 would turn out to be $\{S,Q\}$- complete.
\item The free product $\mathcal{M}_S \Box \mathcal{M}_Q$, defined in \cite{crapofree} is $\{S,Q\}$-complete. Here the bases of $\mathcal{M}_{SQ}$ are unions of spanning sets of $\mathcal{M}_Q$ and independent sets of $\mathcal{M}_S$ which together have size $r(\mathcal{M}_S)+r(\mathcal{M}_Q)$. This is discussed in Section \ref{sec:free}.
\item The principal sum $(\mathcal{M}_S,\mathcal{M}_Q;A,B)$, $A\subseteq S$, $B\subseteq Q$, of matroids $\mathcal{M}_S,\mathcal{M}_Q$, defined 
 in \cite{boninsemidirect} is $\{S,Q\}$-complete.
 This is discussed in Section \ref{sec:free}.
\end{enumerate}
The following lemma shows that for a $\{S,Q\}$-complete matroid $\mathcal{M}_{SQ}$, there exist natural partial orders on $\mathcal{E}_S(\mathcal{M}_{SQ})$, $\mathcal{E}_Q(\mathcal{M}_{SQ})$, 
 and there is a correspondence between them.
\begin{lemma}
\label{lem:partialorder}
Let matroid $\mathcal{M}_{SQ}$ be $\{S,Q\}$-complete. 
\begin{enumerate}
\item Let $\mathbf{I}_S^j,\mathbf{I}_S^k$ be blocks of 
$\mathcal{E}_S(\mathcal{M}_{SQ})$. Let $b_X\subseteq b_Y$, for some $b_X\in \mathbf{I}_S^j, b_Y\in \mathbf{I}_S^k$.
 Then we have the following: whenever $b_{X'}\in \mathbf{I}_S^j$, there exists $b_{Y'}\supseteq b_{X'}$, $b_{Y'}\in \mathbf{I}_S^k$, and whenever $b_{Y'}\in \mathbf{I}_S^k$, there exists $b_{X'}\subseteq b_{Y'}$, $b_{X'}\in \mathbf{I}_S^j$. 
\item We have a partial order on $\mathcal{E}_S(\mathcal{M}_{SQ})$ defined by $\mathbf{I}_S^j\leq \mathbf{I}_S^k$ iff $b_X\subseteq b_Y$,
for some $b_X\in \mathbf{I}_S^j, b_Y\in \mathbf{I}_S^k$.
\item Let $\mathbf{I}_S^j\rightarrow \hat{\mathbf{I}}_Q^j$ be the bijection
 defined in Lemma \ref{lem:spclosure}. We then have a partial order on $\mathcal{E}_Q(\mathcal{M}_{SQ})$ defined by $\hat{\mathbf{I}}_Q^j\leq \hat{\mathbf{I}}_Q^k$
iff $\mathbf{I}_S^j\geq \mathbf{I}_S^k$.
\end{enumerate}
\end{lemma}
\begin{proof}
\begin{enumerate}
\item 
 For some $b^1_Q$, we have a base $b_{X}\uplus b^1_Q$ of $\mathcal{M}_{SQ}$.
 Since $b_{Y}\supseteq b_{X}$, there exists a base of $\mathcal{M}_{SQ}$ contained in $b_{Y}\uplus b^1_Q$ and containing $b_{Y}$.
 Thus $b_Y\uplus b^2_Q$ is a base of $\mathcal{M}_{SQ}$ for some $b^2_Q\subseteq b^1_Q$.

 Since $b_{X'}\in \mathbf{I}_S^j$, we must have that $b_{X'}\uplus b^1_Q$ is a base of
 $\mathcal{M}_{SQ}$. Since $b^2_Q\subseteq b^1_Q$, we have $b_{X'}\uplus b^2_Q$ as independent in $\mathcal{M}_{SQ}$. This can be grown to a base $b_{Y'}\uplus b^2_Q$ of $\mathcal{M}_{SQ}$ using additional elements from $b_Y$.
 It is clear that $b_{Y'}\supseteq b_{X'}$ and that $b_{Y'}\in \mathbf{I}_S^k$.

Next let $b_{Y'}\in \mathbf{I}_S^k$.
  The subset $b_{Y'}\uplus b^1_Q$ contains a base of $\mathcal{M}_{SQ}$ that contains $b^1_Q$. Starting from $b^1_Q$, we can grow a base $b_{X'}\uplus b^1_Q$ of $\mathcal{M}_{SQ}$ using only elements from $b_{Y'}$, so that $b_{X'}\subseteq b_{Y'}$ and that $b_{X'}\in \mathbf{I}_S^j$.

\end{enumerate}
Parts 2 and 3 are immediate from part 1.
\end{proof}

\section{Linking of matroids}
\label{subsec:linking}
There is a `linking' operation for matroids, 
 which has been studied 
in detail  
 in \cite{STHN2014}. (We note that an interesting special case of this operation has been studied in \cite{lemos}.) We will use this operation 
to study bipartition complete matroids and, in particular, free products and principal sums. In this section we give a brief description of this operation. Except for Subsection \ref{subsec:amalgamsplicing} and Theorem \ref{thm:matroidminP}, all the results in this section 
have been at least informally stated in \cite{STHN2014}, but their proofs only sketched. However, for the purposes of this paper we need formal statements and proofs which are presented here.
 
The \textit{linking} $\mathcal{M}_{SP}\leftrightarrow_{\mathcal M} \mathcal{M}_{PQ}$ of matroids $\mathcal{M}_{SP}, \mathcal{M}_{PQ}$, on sets $S\uplus P, P\uplus Q$, respectively with $S,Q$ disjoint, is defined as follows:
$\mathcal{M}_{SP}\leftrightarrow_{\mathcal M} \mathcal{M}_{PQ} \equiv (\mathcal{M}_{SP}\vee \mathcal{M}_{PQ})\times (S\uplus Q).$
 We note that $\mathcal{M}_{SP} \setminus P = \mathcal{M}_{SP} \circ S = \mathcal{M}_{SP}\leftrightarrow_{\mathcal M} \mathcal{F}_P$ and $\mathcal{M}_{SP} / P = \mathcal{M}_{SP} \times S = \mathcal{M}_{SP}\leftrightarrow_{\mathcal M} {\mathbf 0}_P.$

The linking operation for matroids 
is in general not associative but,
under an additional condition, it turns out to be so.
Consider for instance the expression
$\mathcal{M}_{A_1B_1}\leftrightarrow \mathcal{M}_{A_2B_2}\leftrightarrow \mathcal{M}_{A_1B_2},$
where all the $A_i,B_j$ are mutually disjoint. It is clear that this expression
has a unique meaning, namely,
$(\mathcal{M}_{A_1B_1}\vee \mathcal{M}_{A_2B_2} \vee \mathcal{M}_{A_1B_2})\times A_2\uplus B_1.$
Essentially,
index subsets survive if they occur only once and they get coupled through the `$\vee$' operation involving terms which occur twice. These latter do not survive in the final expression.
 We call this the \textit{restricted associativity} of linking and state it as a lemma \cite{STHN2014}.
\begin{lemma}
\label{lem:restassoc}
The expression
$(\leftrightarrow_{\mathcal M}) _{i,j,k,\dots}\mathcal{M}_{ A_i B_j C_k\dots}$, where the index sets $A_i, B_j, C_k, \dots$ are all mutually disjoint
and no index set occurs more than twice in it, has the unique meaning:
$(\leftrightarrow_{\mathcal M}) _{i,j,k,\dots}\mathcal{M}_{ A_i B_j C_k\dots}= (\vee _{i,j,k,\dots}\mathcal{M}_{ A_i B_j C_k\dots})\times T,$
 where $T$ is the union of all the index sets which occur only once.
\end{lemma}
We exploit this idea primarily to write $(\mathcal{M}_{SP}\leftrightarrow_{\mathcal M} \mathcal{M}_{PQ})\leftrightarrow_{\mathcal M} \mathcal{F}_Q$ and $(\mathcal{M}_{SP}\leftrightarrow_{\mathcal M} \mathcal{M}_{PQ})\leftrightarrow_{\mathcal M} {\mathbf 0}_Q$, respectively as
$\mathcal{M}_{SP}\leftrightarrow_{\mathcal M} (\mathcal{M}_{PQ}\leftrightarrow_{\mathcal M} \mathcal{F}_Q)$ and $\mathcal{M}_{SP}\leftrightarrow_{\mathcal M} (\mathcal{M}_{PQ}\leftrightarrow_{\mathcal M} {\mathbf 0}_Q).$

The following result is proved in \cite{STHN2014}.
For better readability of the present paper,
 a slightly modified version of the proof and a relevant remark are given in \ref{sec:idtmatroid}.
\begin{theorem}
\label{thm:idtmatroid}
\begin{enumerate}
\item $\mathcal{M}_{SP}\leftrightarrow_{\mathcal M} \mathcal{M}_{PQ}=(\mathcal{M}_{SP}\vee \mathcal{M}_{PQ})\times (S\uplus Q)= (\mathcal{M}_{SP}\wedge \mathcal{M}_{PQ})\circ (S\uplus Q).$
\item $(\mathcal{M}_{SP}\leftrightarrow_{\mathcal M} \mathcal{M}_{PQ})^*= \mathcal{M}^*_{SP}\leftrightarrow_{\mathcal M} \mathcal{M}^*_{PQ}.$
 In particular, $(\mathcal{M}_{SP}\leftrightarrow_{\mathcal M} \mathcal{M}_{P})^*= \mathcal{M}^*_{SP}\leftrightarrow_{\mathcal M} \mathcal{M}^*_{P}.$
\end{enumerate}
\end{theorem}
This result is analogous to the result for vector spaces, where we have $(\mathcal{V}_{SP}\leftrightarrow_{\mathcal V} \mathcal{V}_{PQ})^{\perp}= \mathcal{V}^{\perp}_{SP}\leftrightarrow_{\mathcal V} \mathcal{V}^{\perp}_{(-P)Q},$
 where $\mathcal{V}_{A}^{\perp} \equiv \{g_A : \sum_{e\in A} (g_A(e)f_A(e))=0, f_A\in \mathcal{V}_A\}$ 
 and $\mathcal{V}_{(-P)Q}^{\perp} \equiv \{(-g_P, f_Q) : (g_P, f_Q)\in \mathcal{V}^{\perp}_{PQ}$ (\cite{HNarayanan1986a}, \cite{HNarayanan1997}).

\begin{remark}
We note that there are situations where $\mathcal{M}_{SP}\leftrightarrow_{\mathcal M}\mathcal{M}_{PQ}=\mathcal{M}_S\oplus \mathcal{M}_Q.$
For instance, when $\mathcal{M}_{SP}= \mathcal{M}_S\oplus \mathcal{M}_P, \mathcal{M}_{PQ}= \mathcal{M}_P\oplus \mathcal{M}_Q,$
 we will have $\mathcal{M}_{SP}\leftrightarrow_{\mathcal M}\mathcal{M}_{PQ}=\mathcal{M}_S\oplus\mathcal{M}_Q.$
\end{remark}
We give below a lower bound on the size of $|P|,$ when
 $\mathcal{M}_{SQ}=\mathcal{M}_{SP}\leftrightarrow_{\mathcal M}\mathcal{M}_{PQ}.$
\begin{theorem}
\label{thm:lowerboundP}
\begin{enumerate}
\item Let $\mathcal{M}_{SQ}=\mathcal{M}_{SP}\leftrightarrow_{\mathcal M}\mathcal{M}_{PQ}.$ Then
$|P|\geq r(\mathcal{M}_{SP}\circ S)-r(\mathcal{M}_{SP}\times S)\geq r(\mathcal{M}_{SQ}\circ S)-r(\mathcal{M}_{SQ}\times S)= r(\mathcal{M}_{SQ}\circ Q)-r(\mathcal{M}_{SQ}\times Q).$
\item Let $\mathcal{M}_{SQ}=\mathcal{M}_{SP}\leftrightarrow_{\mathcal M}\mathcal{M}_{PQ}$ and let $|P|= r(\mathcal{M}_{SQ}\circ S)-r(\mathcal{M}_{SQ}\times S)= r(\mathcal{M}_{SQ}\circ Q)-r(\mathcal{M}_{SQ}\times Q).$ Then $P$ is contained in both a base and a cobase of $\mathcal{M}_{SP}$ and $\mathcal{M}_{PQ}.$
\end{enumerate}
\end{theorem}

\begin{proof}
\begin{enumerate}
\item  We have $\mathcal{M}_{SQ}\circ S= (\mathcal{M}_{SP}\leftrightarrow_{\mathcal M}\mathcal{M}_{PQ}) \leftrightarrow_{\mathcal M} \mathcal{F}_Q= \mathcal{M}_{SP}\leftrightarrow_{\mathcal M} (\mathcal{M}_{PQ}\leftrightarrow_{\mathcal M} \mathcal{F}_Q).$
 Next $\mathcal{M}_{SQ}\times S= (\mathcal{M}_{SP}\leftrightarrow_{\mathcal M}\mathcal{M}_{PQ}) \leftrightarrow_{\mathcal M} {\mathbf 0}_Q= \mathcal{M}_{SP}\leftrightarrow_{\mathcal M} (\mathcal{M}_{PQ}\leftrightarrow_{\mathcal M} {\mathbf 0}_Q).$ Therefore, by the definition of the linking operation, $r(\mathcal{M}_{SQ}\circ S) \leq
 r(\mathcal{M}_{SP}\circ S)$ and
$r(\mathcal{M}_{SQ}\times S) \geq r(\mathcal{M}_{SP}\times S).$
 Since $|P|\geq r(\mathcal{M}_{SP}\circ P)-r(\mathcal{M}_{SP}\times P)= r(\mathcal{M}_{SP}\circ S)-r(\mathcal{M}_{SP}\times S),$
 and $r(\mathcal{M}_{SQ}\circ S)-r(\mathcal{M}_{SQ}\times S)= r(\mathcal{M}_{SQ}\circ Q)-r(\mathcal{M}_{SQ}\times Q),$
 the result follows.

\item  By part 1 above, $r(\mathcal{M}_{SP}\circ S)-r(\mathcal{M}_{SP}\times S)\geq r(\mathcal{M}_{SQ}\circ S)-r(\mathcal{M}_{SQ}\times S)$ and $ r(\mathcal{M}_{PQ}\circ Q)-r(\mathcal{M}_{PQ}\times Q)\geq r(\mathcal{M}_{SQ}\circ Q)-r(\mathcal{M}_{SQ}\times Q).$ We have $|P|\geq r(\mathcal{M}_{SP}\circ P)-r(\mathcal{M}_{SP}\times P)$
 and also $|P|\geq r(\mathcal{M}_{PQ}\circ P)-r(\mathcal{M}_{PQ}\times P).$
If $|P|= r(\mathcal{M}_{SQ}\circ S)-r(\mathcal{M}_{SQ}\times S) =r(\mathcal{M}_{SQ}\circ Q)-r(\mathcal{M}_{SQ}\times Q),$ we must therefore have $|P|= r(\mathcal{M}_{SP}\circ P)-r(\mathcal{M}_{SP}\times P)$
 and also $|P|= r(\mathcal{M}_{PQ}\circ P)-r(\mathcal{M}_{PQ}\times P).$
 This implies that $|P|= r(\mathcal{M}_{SP}\circ P)= r(\mathcal{M}_{PQ}\circ P)$ and $r(\mathcal{M}_{SP}\times P)=r(\mathcal{M}_{PQ}\times P)=0$ and therefore that
 $P$ is contained in both a base and a cobase of $\mathcal{M}_{SP}$ and $\mathcal{M}_{PQ}.$
\end{enumerate}
\end{proof}

\begin{definition}
We define the \textbf{$\{S,Q\}$-link width} $\lambda_{SQ}(\mathcal{M}_{SQ})$
 of the matroid $\mathcal{M}_{SQ}$ to be $ r(\mathcal{M}\circ S)-r(\mathcal{M}\times S)= r(\mathcal{M}\circ Q)-r(\mathcal{M}\times Q).$
\end{definition}

We note that the $\{S,Q\}$-link width of $\mathcal{M}_{SQ}$ is called the \textbf{connectivity function at $S$ ($Q$)} of $\mathcal{M}_{SQ}$ in \cite{oxley}.
 From Theorem \ref{thm:lowerboundP}, this is a lower bound for $|P|$ when $\mathcal{M}_{SQ}= \mathcal{M}_{SP}\leftrightarrow_{\mathcal M}\mathcal{M}_{PQ}.$
When this lower bound is achieved, we say $\mathcal{M}_{SQ}$ is \textbf{minimally decomposed} into $\mathcal{M}_{SP}\leftrightarrow_{\mathcal M}\mathcal{M}_{PQ}.$
 Minimal decomposition of a matroid $\mathcal{M}_{SQ}$ achieving the above lower bound, is not always possible. But, when it is, we have a better understanding of the structure of $\mathcal{M}_{SQ},$ and can simplify procedures involving the matroid.

\begin{remark}
\label{rem:dissimilarmatroid}
 Every matroid $\mathcal{M}_{SQ}$ can be trivially decomposed
 as $\mathcal{M}_{SQ}= (\mathcal{M}_{SQ})_{SQ'}\leftrightarrow_{\mathcal M} I_{QQ'}$ where $I_{QQ'}$ has elements
 $e,e'$ in parallel and has $Q$ as a base \cite{STHN2014}.
When possible, minimization of $|P|,$
 keeping $\mathcal{M}_{SP}\leftrightarrow_{\mathcal M} \mathcal{M}_{PQ}$ invariant, with
$|P|= r((\mathcal{M}_{SP}\leftrightarrow_{\mathcal M} \mathcal{M}_{PQ})\circ S) - r((\mathcal{M}_{SP}\leftrightarrow_{\mathcal M} \mathcal{M}_{PQ})\times S),$
  gives advantages both in understanding the structure of $\mathcal{M}_{SQ}$ and in algorithms for building it.
 Examples can be seen in Subsection \ref{subsec:amalgamsplicing} and in Section \ref{subsec:bipartitioncompletion}.
But this minimization
is not always possible.
Figure \ref{fig:matroid1} shows a graph whose (graphic) matroid $\mathcal{M}_{SQ}$ cannot be decomposed
minimally as $\mathcal{M}_{SP}\leftrightarrow_{\mathcal M} \mathcal{M}_{PQ},$ for the specified partition $(S,Q),$
 where $S \equiv \{e_1,e_2,e_3\}, Q \equiv \{e_4,e_5,e_6\}.$
 A proof of the impossibility is given in \ref{sec:mindecompmatroid}.
\end{remark}
\subsection{Amalgam and splicing matroids}
\label{subsec:amalgamsplicing}
 A natural question that arises is whether matroid composition can be used
 to build a matroid $\mathcal{M}_{STQ}$ in which specified matroids
 $\mathcal{M}_{ST}, \mathcal{M}_{TQ}$ occur as minors.
We say $\mathcal{M}_{STQ}$ is an \textbf{amalgam} of
 $\mathcal{M}_{ST}, \mathcal{M}_{TQ}$ iff $\mathcal{M}_{STQ}\circ (S\uplus T)= \mathcal{M}_{ST}, \mathcal{M}_{STQ}\circ (T\uplus Q)= \mathcal{M}_{TQ}$ \cite{oxley}.
Amalgams do not always exist even if the necessary condition
 $\mathcal{M}_{ST}\circ T= \mathcal{M}_{TQ}\circ T$ is satisfied.

We say $\mathcal{M}_{STQ}$ is a \textbf{splicing} of
 $\mathcal{M}_{ST}, \mathcal{M}_{TQ}$ iff $\mathcal{M}_{STQ}\times (S\uplus T)= \mathcal{M}_{ST}, \mathcal{M}_{STQ}\circ (T\uplus Q)= \mathcal{M}_{TQ}$ \cite{boninsplice}. It follows that $\mathcal{M}_{ST}\circ T= \mathcal{M}_{STQ}\times (S\uplus T)\circ T=
 \mathcal{M}_{STQ}\circ (Q\uplus T)\times T=\mathcal{M}_{TQ}\times T.$
 It turns out that the necessary condition $\mathcal{M}_{ST}\circ T= \mathcal{M}_{TQ}\times T,$
  is also sufficient for the splicing of $\mathcal{M}_{ST}, \mathcal{M}_{TQ},$ to exist, as shown in \cite{boninsplice}.

In this subsection we examine how the techniques of this paper can be used 
to simplify procedures for construction of amalgam and splicing matroids 
 when suitable decompositions exist.

Suppose that $\mathcal{M}_{ST}= \mathcal{M}_{SP_1}\leftrightarrow_{\mathcal M} \mathcal{M}_{TP_1},\mathcal{M}_{TQ}= \mathcal{M}_{TP_2}\leftrightarrow_{\mathcal M} \mathcal{M}_{QP_2},$ with the decompositions being minimal. Therefore
 $P_1$ can be included in a base as well as a cobase of $\mathcal{M}_{SP_1}$ and $ \mathcal{M}_{TP_1}$
and $P_2$ can be included in a base as well as a cobase of $\mathcal{M}_{QP_2}$ and $ \mathcal{M}_{TP_2}.$
We then have $\mathcal{M}_{ST}\circ T= (\mathcal{M}_{SP_1}\leftrightarrow_{\mathcal M} \mathcal{M}_{TP_1})\leftrightarrow_{\mathcal M} \mathcal{F}_S=
\mathcal{M}_{TP_1}\leftrightarrow_{\mathcal M} (\mathcal{M}_{SP_1}\leftrightarrow_{\mathcal M} \mathcal{F}_S)= 
\mathcal{M}_{TP_1}\leftrightarrow_{\mathcal M} (\mathcal{M}_{SP_1}\circ P_1)=
\mathcal{M}_{TP_1}\leftrightarrow_{\mathcal M} \mathcal{F}_{P_1}=
\mathcal{M}_{TP_1}\circ T$ and 
$\mathcal{M}_{ST}\times T= (\mathcal{M}_{SP_1}\leftrightarrow_{\mathcal M} \mathcal{M}_{TP_1})\leftrightarrow_{\mathcal M} {\mathbf 0}_S=
\mathcal{M}_{TP_1}\leftrightarrow_{\mathcal M} (\mathcal{M}_{SP_1}\leftrightarrow_{\mathcal M} {\mathbf 0}_S)=
\mathcal{M}_{TP_1}\leftrightarrow_{\mathcal M} (\mathcal{M}_{SP_1}\times P_1)=
\mathcal{M}_{TP_1}\leftrightarrow_{\mathcal M} {\mathbf 0}_{P_1}=
\mathcal{M}_{TP_1}\times T.$ Similarly 
$\mathcal{M}_{TQ}\circ T= \mathcal{M}_{TP_2}\circ T$
 and $\mathcal{M}_{TQ}\times T= \mathcal{M}_{TP_2}\times T.$

Suppose we have $\mathcal{M}_{TP_1P_2}$ as an amalgam of $\mathcal{M}_{TP_1}$
 and $\mathcal{M}_{TP_2}.$
 We again use restricted associativity of linking (Lemma \ref{lem:restassoc}) to show that 
$(\mathcal{M}_{SP_1}\leftrightarrow_{\mathcal M} \mathcal{M}_{TP_1P_2}\leftrightarrow_{\mathcal M} \mathcal{M}_{QP_2})
 \circ (S\uplus T)=\mathcal{M}_{ST}.$ 

 We have 
$(\mathcal{M}_{SP_1}\leftrightarrow_{\mathcal M} \mathcal{M}_{TP_1P_2}\leftrightarrow_{\mathcal M} \mathcal{M}_{QP_2})
 \circ (S\uplus T)=
(\mathcal{M}_{SP_1}\leftrightarrow_{\mathcal M} \mathcal{M}_{TP_1P_2}\leftrightarrow_{\mathcal M} \mathcal{M}_{QP_2})\leftrightarrow_{\mathcal M} \mathcal{F}_Q
= 
\mathcal{M}_{SP_1}\leftrightarrow_{\mathcal M} \mathcal{M}_{TP_1P_2}\leftrightarrow_{\mathcal M} (\mathcal{M}_{QP_2}\leftrightarrow_{\mathcal M} \mathcal{F}_Q)
=
\mathcal{M}_{SP_1}\leftrightarrow_{\mathcal M} \mathcal{M}_{TP_1P_2}\leftrightarrow_{\mathcal M} (\mathcal{M}_{QP_2}\circ P_2)
=
\mathcal{M}_{SP_1}\leftrightarrow_{\mathcal M} (\mathcal{M}_{TP_1P_2}\leftrightarrow_{\mathcal M} \mathcal{F}_{P_2})
=
\mathcal{M}_{SP_1}\leftrightarrow_{\mathcal M} (\mathcal{M}_{TP_1P_2}\circ (T\uplus P_1))
=
\mathcal{M}_{SP_1}\leftrightarrow_{\mathcal M} \mathcal{M}_{TP_1}= 
\mathcal{M}_{ST}.$
 Similarly 
$(\mathcal{M}_{SP_1}\leftrightarrow_{\mathcal M} \mathcal{M}_{TP_1P_2}\leftrightarrow_{\mathcal M} \mathcal{M}_{QP_2})
 \circ (Q\uplus T)= \mathcal{M}_{QT}.$
Therefore $(\mathcal{M}_{SP_1}\leftrightarrow_{\mathcal M} \mathcal{M}_{TP_1P_2}\leftrightarrow_{\mathcal M} \mathcal{M}_{QP_2})$ is an amalgam of $\mathcal{M}_{ST}$ and $\mathcal{M}_{QT}.$

A consequence of the above discussion is that, in order to construct
 counterexamples where amalgams do not exist, we can restrict our attention to
 the case where $P_1$ can be included in a base as well as a cobase of $\mathcal{M}_{SP_1}$ and $\mathcal{M}_{TP_1}$
and $P_2$ can be included in a base as well as a cobase of $\mathcal{M}_{QP_2}$ and $\mathcal{M}_{TP_2}.$
 The counterexample given in \cite{oxley} (Example 7.2.4) is of this kind.

Next let us consider the case where $\mathcal{M}_{TP_1P_2}$ is a splicing
 of $\mathcal{M}_{TP_1}$ and $\mathcal{M}_{TP_2},$
 i.e., $\mathcal{M}_{TP_1P_2}\times (T\uplus P_1)= \mathcal{M}_{TP_1}$
 and
 $\mathcal{M}_{TP_1P_2}\circ (T\uplus P_2)= \mathcal{M}_{TP_2}.$

 We have $(\mathcal{M}_{SP_1}\leftrightarrow_{\mathcal M} \mathcal{M}_{TP_1P_2}\leftrightarrow_{\mathcal M} \mathcal{M}_{QP_2})
 \times (S\uplus T)=
(\mathcal{M}_{SP_1}\leftrightarrow_{\mathcal M} \mathcal{M}_{TP_1P_2}\leftrightarrow_{\mathcal M} \mathcal{M}_{QP_2})\leftrightarrow_{\mathcal M} {\mathbf 0}_Q
=
\mathcal{M}_{SP_1}\leftrightarrow_{\mathcal M} \mathcal{M}_{TP_1P_2}\leftrightarrow_{\mathcal M} (\mathcal{M}_{QP_2}\leftrightarrow_{\mathcal M} {\mathbf 0}_Q)
=
\mathcal{M}_{SP_1}\leftrightarrow_{\mathcal M} \mathcal{M}_{TP_1P_2}\leftrightarrow_{\mathcal M} (\mathcal{M}_{QP_2}\times P_2)
=
\mathcal{M}_{SP_1}\leftrightarrow_{\mathcal M} \mathcal{M}_{TP_1P_2}\leftrightarrow_{\mathcal M} {\mathbf 0}_{P_2}
=
\mathcal{M}_{SP_1}\leftrightarrow_{\mathcal M} (\mathcal{M}_{TP_1P_2}\times (T\uplus P_1))
=
\mathcal{M}_{SP_1}\leftrightarrow_{\mathcal M} \mathcal{M}_{TP_1}=
\mathcal{M}_{ST}.$
 Similarly
 it can be seen that $(\mathcal{M}_{SP_1}\leftrightarrow_{\mathcal M} \mathcal{M}_{TP_1P_2}\leftrightarrow_{\mathcal M} \mathcal{M}_{QP_2})
 \circ (Q\uplus T)= \mathcal{M}_{TQ}.$

\subsection{Minimization of $|P|$ in $\mathcal{M}_{SP}\leftrightarrow_{\mathcal M} \mathcal{M}_{PQ},$ with conditions on $\mathcal{M}_{SP}, \mathcal{M}_{PQ}^*$}
\label{subsec:partmin}

The literature has instances where given two matroids $\mathcal{M}_S, \mathcal{M}_Q,$ a new matroid $\mathcal{M}_{SQ}$ is defined which has the former matroids as restrictions or contractions, for example the free product $\mathcal{M}_S \Box \mathcal{M}_Q$ of \cite{crapofree} or
 the semidirect sum of \cite{boninsemidirect}. Using the following lemmas, we show, if $\mathcal{M}_{SP}\times P = \mathcal{M}_{PQ}^*\times P, $
that more can be done.

\begin{definition}
We say matroid $\mathcal{M}_S^1 \geq \mathcal{M}_S^2$ or matroid $\mathcal{M}_S^2 \leq \mathcal{M}_S^1$ iff every base of $\mathcal{M}_S^1\circ T, T\subseteq S, $ contains a base of $\mathcal{M}_S^2\circ T$ and every base of $\mathcal{M}_S^2\circ T$ is contained in a base of $\mathcal{M}_S^1\circ T.$
\end{definition}

We note that $\mathcal{M}_S^1 \geq \mathcal{M}_S^2$ iff there is a \textit{strong map} from $\mathcal{M}_S^1$ to $\mathcal{M}_S^2,$ i.e., iff
 for every $T_1\subseteq T_2\subseteq S,$ we have
$r_1(T_2)-r_1(T_1)\geq r_2(T_2)-r_2(T_1),$
 where $r_1(\cdot), r_2(\cdot)$ are rank functions of $\mathcal{M}_S^1, \mathcal{M}_S^2$
 respectively \cite{oxley}.

\begin{lemma}
\label{lem:basecontainedcontains}
\begin{enumerate}
\item $\mathcal{M}_{SP}\circ S \geq \mathcal{M}_{SP}\leftrightarrow_{\mathcal M} \mathcal{M}_P \geq \mathcal{M}_{SP}\times S.$
\item $\mathcal{M}_{SP}\leftrightarrow_{\mathcal M} (\mathcal{M}^*_{SP}\circ P) = \mathcal{M}_{SP}\circ S,$ and $\mathcal{M}_{SP}\leftrightarrow_{\mathcal M} (\mathcal{M}^*_{SP}\times P) = \mathcal{M}_{SP}\times S.$
\end{enumerate}
\end{lemma}

\begin{proof}
\begin{enumerate}
\item 
Let $b_S$ be a base of $\mathcal{M}_{SP}\leftrightarrow_{\mathcal M} \mathcal{M}_P.$ Then there exist
 independent sets $b^1_P, b^2_P$ of $\mathcal{M}_{SP}\circ P, \mathcal{M}_P$ respectively
such that $b^1_P\uplus b^2_P$ is a base of $(\mathcal{M}_{SP}\vee\mathcal{M}_P)\circ P,$
 $b_S\uplus (b^1_P\uplus b^2_P)$ is a base of $\mathcal{M}_{SP} \vee \mathcal{M}_P$
and $b_S\uplus b^1_P$ is a base of $\mathcal{M}_{SP}.$
Hence $b_S$ is contained in a base of $\mathcal{M}_{SP}\circ S$
and contains a base of $\mathcal{M}_{SP}\times S.$

We next show that every base of $\mathcal{M}_{SP}\circ S$ contains a
base of $\mathcal{M}_{SP}\leftrightarrow_{\mathcal M} \mathcal{M}_P$ and every base of $\mathcal{M}_{SP}\times S$ is contained in a base of $\mathcal{M}_{SP}\leftrightarrow_{\mathcal M} \mathcal{M}_P.$
Let
 $b_S\uplus b^1_P$ be a base of $\mathcal{M}_{SP},$ $b^2_P$ an independent set of $\mathcal{M}_P,$
 such that $b^1_P\uplus b^2_P$ is a base of $(\mathcal{M}_{SP}\vee \mathcal{M}_P)\circ P.$
If $\hat{b}_S$ be a base of $\mathcal{M}_{SP}\circ S,$ then
$b^1_P\uplus \hat{b}_S $ contains a base $b^1_P\uplus b^1_S $
of $\mathcal{M}_{SP}.$ Now $b^1_P\uplus b^2_P\uplus b^1_S $ is a base of $\mathcal{M}_{SP}\vee \mathcal{M}_P,$ so that $\hat{b}_S$ contains the base
$b^1_S$ of
$(\mathcal{M}_{SP}\vee \mathcal{M}_P)\times S= \mathcal{M}_{SP}\leftrightarrow_{\mathcal M} \mathcal{M}_P.$

If $\tilde{b}_S$ be a base of $\mathcal{M}_{SP}\times S,$ then
 $\tilde{b}_S\uplus b^1_P$ is independent in $\mathcal{M}_{SP},$ so that it can be
 extended to a base of $b^1_P\uplus b^1_S $
of $\mathcal{M}_{SP}.$
Now $b^1_P\uplus b^2_P\uplus b^1_S $ is a base of $\mathcal{M}_{SP}\vee \mathcal{M}_P,$
 so that
$\tilde{b}_S$ is contained in
the base $b^1_S$ of
$(\mathcal{M}_{SP}\vee \mathcal{M}_P)\times S=\mathcal{M}_{SP}\leftrightarrow_{\mathcal M} \mathcal{M}_P.$

Thus $\mathcal{M}_{SP}\circ S\geq \mathcal{M}_{SP}\leftrightarrow_{\mathcal M} \mathcal{M}_P\geq \mathcal{M}_{SP}\times S.$

\item We will show that rank of the concerned matroids are the same.
So from part 1 above, the equality of the matroids follows.

Choose a base $b_S$ of $\mathcal{M}_{SP}\circ S$ and a base $b_P$ of $\mathcal{M}_{SP}\times P.$
 Then $P-b_P$ is a base of $\mathcal{M}_{SP}^*\circ P$ and $b_S\uplus b_P, P-b_P,$ are maximally distant bases of $\mathcal{M}_{SP},\mathcal{M}_{SP}^*\circ P,$ respectively which cover $P$ and $b_S$ is a base of
 $(\mathcal{M}_{SP}\vee (\mathcal{M}_{SP}^*\circ P))\times S=\mathcal{M}_{SP}\leftrightarrow_{\mathcal M} (\mathcal{M}_{SP}^*\circ P).$

Proof of the equality of the matroids $\mathcal{M}_{SP}\times S,\mathcal{M}_{SP}\leftrightarrow_{\mathcal M} (\mathcal{M}_{SP}^*\times P),$ is similar, working with bases $b'_S,b'_P, P-b'_P$ of matroids $\mathcal{M}_{SP}\times S, \mathcal{M}_{SP} \circ P, \mathcal{M}_{SP}^*\times P,$ respectively.
\end{enumerate}
\end{proof}

\begin{lemma}
\label{cor:minorspecial}
Let $\mathcal{M}_{SP}\circ P = \mathcal{M}^*_{PQ}\circ P$ and $\mathcal{M}_{SP}\times P = \mathcal{M}^*_{PQ}\times P.$
We then have the following:
\begin{enumerate}
 \item $(\mathcal{M}_{SP}\leftrightarrow_{\mathcal M} \mathcal{M}_{PQ})\circ S = \mathcal{M}_{SP}\circ S$ and $(\mathcal{M}_{SP}\leftrightarrow_{\mathcal M} \mathcal{M}_{PQ})\times S = \mathcal{M}_{SP}\times S.$
\item $\mathcal{M}^*_{SP}\circ P = \mathcal{M}_{PQ}\circ P,$ $\mathcal{M}^*_{SP}\times P = \mathcal{M}_{PQ}\times P$ and 
$(\mathcal{M}_{SP}\leftrightarrow_{\mathcal M} \mathcal{M}_{PQ})\circ Q = \mathcal{M}_{PQ}\circ Q,$ $(\mathcal{M}_{SP}\leftrightarrow_{\mathcal M} \mathcal{M}_{PQ})\times Q = \mathcal{M}_{PQ}\times Q.$
\end{enumerate}
\end{lemma}

\begin{proof}
\begin{enumerate}
\item  We have:
$(\mathcal{M}_{SP}\leftrightarrow_{\mathcal M} \mathcal{M}_{PQ})\circ S = (\mathcal{M}_{SP}\leftrightarrow_{\mathcal M} \mathcal{M}_{PQ})\leftrightarrow_{\mathcal M} \mathcal{F}_Q = \mathcal{M}_{SP}\leftrightarrow_{\mathcal M} (\mathcal{M}_{PQ}\leftrightarrow_{\mathcal M} \mathcal{F}_Q) = \mathcal{M}_{SP}\leftrightarrow_{\mathcal M} (\mathcal{M}_{PQ}\circ P) = \mathcal{M}_{SP}\leftrightarrow_{\mathcal M} (\mathcal{M}^*_{SP}\circ P) = \mathcal{M}_{SP}\circ S,$ using Lemma~\ref{lem:basecontainedcontains}.

We have:
$(\mathcal{M}_{SP}\leftrightarrow_{\mathcal M} \mathcal{M}_{PQ})\times S = (\mathcal{M}_{SP}\leftrightarrow_{\mathcal M} \mathcal{M}_{PQ})\leftrightarrow_{\mathcal M} {\mathbf 0}_Q = \mathcal{M}_{SP}\leftrightarrow_{\mathcal M} (\mathcal{M}_{PQ}\leftrightarrow_{\mathcal M} {\mathbf 0}_Q) = \mathcal{M}_{SP}\leftrightarrow_{\mathcal M} (\mathcal{M}_{PQ}\times P) = \mathcal{M}_{SP}\leftrightarrow_{\mathcal M} (\mathcal{M}^*_{SP}\times P) = \mathcal{M}_{SP}\times S,$ using Lemma~\ref{lem:basecontainedcontains}.

\item  We have $\mathcal{M}_{SP}^*\circ P = (\mathcal{M}_{SP}\times P)^* = (\mathcal{M}^*_{PQ}\times P)^* = \mathcal{M}_{PQ}\circ P,$ and $\mathcal{M}_{SP}^*\times P = (\mathcal{M}_{SP}\circ P)^* = (\mathcal{M}^*_{PQ}\circ P)^* = \mathcal{M}_{PQ}\times P.$
The proof of $(\mathcal{M}_{SP}\leftrightarrow_{\mathcal M} \mathcal{M}_{PQ})\circ Q = \mathcal{M}_{PQ}\circ Q$ and $(\mathcal{M}_{SP}\leftrightarrow_{\mathcal M} \mathcal{M}_{PQ})\times Q = \mathcal{M}_{PQ}\times Q$ is by replacing $S$ by $Q$ and $\mathcal{M}_{SP}$ by $\mathcal{M}_{PQ}$ in the proof of part 1 above.
\end{enumerate}
\end{proof}
A simple example of a matroid where $\mathcal{M}_{SP}\circ P= \mathcal{M}^*_{PQ}\circ P, \mathcal{M}_{SP}\times P= \mathcal{M}^*_{PQ}\times P,$ is when $P$ is contained in both a base and a cobase of each of the matroids $\mathcal{M}_{SP}, \mathcal{M}_{PQ}.$ In such a case $\mathcal{M}_{SP}\circ P= \mathcal{M}^*_{PQ}\circ P= \mathcal{F}_P, \mathcal{M}_{SP}\times P= \mathcal{M}^*_{PQ}\times P= {\mathbf 0}_P$ and therefore by Lemma \ref{cor:minorspecial}, $\mathcal{M}_{SP}\leftrightarrow_{\mathcal M} \mathcal{M}_{PQ}$ would have the same restriction and contraction on $S$ ($Q$) as $\mathcal{M}_{SP}$ ($\mathcal{M}_{PQ}$).
\begin{definition}
We say matroids ${\mathcal{M}}_{SP},{\mathcal{M}}_{PQ}$ are \textbf{compatible} over $P$ iff $\mathcal{M}_{SP}\circ P= \mathcal{M}_{PQ}^*\circ P$ and $\mathcal{M}_{SP}\times P= \mathcal{M}_{PQ}^*\times P,$ or equivalently using duality, $\mathcal{M}^*_{SP}\circ P= \mathcal{M}_{PQ}\circ P$ and $\mathcal{M}_{SP}^*\times P= \mathcal{M}_{PQ}\times P.$
\end{definition}

The following lemma states some useful properties of compatible matroids.

\begin{lemma}
\label{lem:baseexchange}
Let $\mathcal{M}_{SP}, \mathcal{M}_{PQ}$ be compatible matroids over $P.$
 We have the following.
\begin{enumerate}
\item Let $b_S\uplus b^1_P, b_Q\uplus b^2_P$ be independent sets of $\mathcal{M}_{SP}, \mathcal{M}_{PQ},$ respectively such that their union is a base of $\mathcal{M}_{SP}\vee \mathcal{M}_{PQ},$ covering $P.$
Then $b_S\uplus b^1_P, b_Q\uplus b^2_P$ are bases of $\mathcal{M}_{SP}, \mathcal{M}_{PQ},$ respectively which cover $P$ and contain disjoint bases of $\mathcal{M}_{SP}\times P, \mathcal{M}_{PQ}\times P,$ respectively.
\item If $b_S\uplus b_Q$ is a base of $\mathcal{M}_{SP}\leftrightarrow_{\mathcal M} \mathcal{M}_{PQ},$
 then there exist disjoint bases $b_S\uplus \tilde{b}^1_P, b_Q\uplus \tilde{b}^2_P$ of $\mathcal{M}_{SP},\mathcal{M}_{PQ}$ which cover $P$ and whose union is a base of $\mathcal{M}_{SP}\vee \mathcal{M}_{PQ}.$
\end{enumerate}
\end{lemma}

\begin{proof}
\begin{enumerate}
\item  Choose bases $b'_S,b'_Q$ of $\mathcal{M}_{SP}\circ S,\mathcal{M}_{PQ}\times Q$ and a base $b'_P$ of $\mathcal{M}_{SP}\times P.$
Then $P-b'_P$ is a base of $(\mathcal{M}_{SP}\times P)^*= (\mathcal{M}_{PQ}^*\times P)^*= \mathcal{M}_{PQ}\circ P$ and
 $b'_S\uplus b'_P, b'_Q\uplus (P-b'_P)$ are disjoint bases of $\mathcal{M}_{SP},\mathcal{M}_{PQ},$ respectively, which cover $P$ and are such that their union is a base of $\mathcal{M}_{SP}\vee \mathcal{M}_{PQ}.$

 Since all bases of $\mathcal{M}_{SP}\vee \mathcal{M}_{PQ}$ have the same cardinality, it follows that if $b_S\uplus b^1_P, b_Q\uplus b^2_P$ be independent sets of $\mathcal{M}_{SP}, \mathcal{M}_{PQ},$ respectively such that their union is a base of $\mathcal{M}_{SP}\vee \mathcal{M}_{PQ},$ covering $P,$ then they are bases of $\mathcal{M}_{SP}, \mathcal{M}_{PQ},$ respectively which cover $P.$ Since every base of $\mathcal{M}_{SP}, \mathcal{M}_{PQ}$ contains bases of $\mathcal{M}_{SP}\times P, \mathcal{M}_{PQ}\times P,$ respectively and $b_S\uplus b^1_P, b_Q\uplus b^2_P$ are disjoint, the result follows.

\item  Let $b^1_P$ and $b^2_P \equiv P-b^1_P$ be bases of $\mathcal{M}_{SP}\circ P$ and $\mathcal{M}_{PQ}\times P = \mathcal{M}_{SP}^*\times P,$ respectively.
Let $b^1_S,b^2_Q$ be bases of $\mathcal{M}_{SP}\times S, \mathcal{M}_{PQ} \circ Q,$ respectively.
   Then $b^1_S\uplus b^1_P, b^2_Q\uplus b^2_P$ are bases of $\mathcal{M}_{SP},\mathcal{M}_{PQ},$ respectively and $b^1_S\uplus b^1_P\uplus b^2_Q\uplus b^2_P$ is a base of $\mathcal{M}_{SP}\vee \mathcal{M}_{PQ}$ which is a union of disjoint bases of $\mathcal{M}_{SP},\mathcal{M}_{PQ},$ which also cover $P.$
Therefore, $P$ is a base of $(\mathcal{M}_{SP}\vee \mathcal{M}_{PQ})\circ P.$
  Now $b_S\uplus b_Q$ is a base of $(\mathcal{M}_{SP}\vee \mathcal{M}_{PQ})\times (S\uplus Q).$
 Therefore $b_S\uplus b_Q\uplus P$ is a base of $\mathcal{M}_{SP}\vee \mathcal{M}_{PQ}.$
Therefore there exist disjoint bases $b_S\uplus \tilde{b}^1_P, b_Q\uplus \tilde{b}^2_P$ of   $\mathcal{M}_{SP},\mathcal{M}_{PQ},$ respectively which cover $P.$
\end{enumerate}
\end{proof}

\begin{definition}
 Let $\mathcal{M}_{SP},\mathcal{M}_{PQ}$ be matroids compatible over $P.$ We say they are \textbf{strongly compatible} over $P$ iff the following additional condition is also satisfied:
Let ${b}^3_P, {b}^4_P$ be disjoint bases of $\mathcal{M}_{SP}\times P, \mathcal{M}_{PQ}\times P,$ respectively. Then, whenever $b_S\uplus b_Q$ is a base of $\mathcal{M}_{SP}\leftrightarrow_{\mathcal M} \mathcal{M}_{PQ},$ there exist disjoint bases $b_S\uplus {b}^1_P, b_Q\uplus {b}^2_P$ of $\mathcal{M}_{SP},\mathcal{M}_{PQ},$ respectively, where $b_S\uplus {b}^1_P\supseteq b^3_P$ and $b_Q\uplus {b}^2_P\supseteq b^4_P.$
\end{definition}

In the next section we show that for bipartition complete matroids, under a simple additional condition, compatibility implies strong compatibility (Theorem \ref{thm:compatible}).

The following result
 is analogous to a result on vector spaces (Exercise 8.11 of \cite{HNarayanan1997}) and is important for the subsequent developments in this paper.

\begin{theorem}
\label{thm:matroidminP}
Let $\mathcal{M}_{SP}, \mathcal{M}_{PQ}$ be matroids strongly compatible over $P$ and let $\mathcal{M}_{SQ} \equiv \mathcal{M}_{SP}\leftrightarrow_{\mathcal M} \mathcal{M}_{PQ}.$

Let $b^3_P, b^4_P$ be disjoint bases of $\mathcal{M}_{SP}\times P, \mathcal{M}_{PQ}\times P,$ respectively.
Let $\hat{P} \equiv P - (b^3_P \uplus b^4_P),$ and let $\mathcal{M}_{S\hat{P}} \equiv \mathcal{M}_{SP}\times (S\uplus (P-b^3_P))\circ (S\uplus \hat{P}),$ $\mathcal{M}_{\hat{P}Q} \equiv \mathcal{M}_{PQ}\times (Q\uplus (P-b^4_P))\circ (Q\uplus \hat{P}).$

We have the following.
\begin{enumerate}
\item There exist disjoint bases $b_S\uplus b^1_P, b_Q\uplus b^2_P$  of $\mathcal{M}_{SP}, \mathcal{M}_{PQ},$ respectively which contain $b^3_P, b^4_P,$ respectively  and which cover $P.$
The subsets $b_S\uplus (b^1_P-b^3_P),$ $b_Q\uplus (b^2_P-b^4_P),$ are disjoint bases of $\mathcal{M}_{S\hat{P}} \equiv \mathcal{M}_{SP}\times (S\uplus (P-b^3_P)) \circ (S\uplus ((P-b^3_P)-b^4_P))$ , $\mathcal{M}_{\hat{P}Q} \equiv \mathcal{M}_{PQ}\times (Q\uplus (P-b^4_P)) \circ (Q\uplus ((P-b^3_P)-b^4_P)),$ respectively. Further these bases cover $\hat{P} \equiv (P-b^3_P)-b^4_P.$

\item $\mathcal{M}_{S\hat{P}}\circ \hat{P} = \mathcal{M}_{\hat{P}Q}^*\circ \hat{P}$ and $\mathcal{M}_{S\hat{P}}\times \hat{P} = \mathcal{M}_{\hat{P}Q}^*\times \hat{P}.$
\item $\mathcal{M}_{SQ} = \mathcal{M}_{S\hat{P}}\leftrightarrow_{\mathcal M} \mathcal{M}_{\hat{P}Q},$
\item There exists a base and a cobase of $\mathcal{M}_{S\hat{P}}$ (respectively, $\mathcal{M}_{\hat{P}Q}$) that contains $\hat{P}.$
\item $|\hat{P}|=r(\mathcal{M}_{SP}\circ S)- r(\mathcal{M}_{SP}\times S)=
r(\mathcal{M}_{PQ}\circ Q)- r(\mathcal{M}_{PQ}\times Q)= 
\lambda_{SQ}(\mathcal{M}_{{S}Q}).$
Hence $|\hat{P}|$ is the minimum of $|P|$ when $\mathcal{M}_{SQ} = \mathcal{M}_{SP}\leftrightarrow_{\mathcal M} \mathcal{M}_{PQ}.$

\end{enumerate}
\end{theorem}

\begin{proof}
\begin{enumerate}
\item  
%We note that if we contract a subset $b^1$ of a base $b$ of a matroid $\mathcal{M}_E,$ then $b-b^1$ is a base of $\mathcal{M}_E\times (E-b^1)$ and if we delete a subset $A$ of $E-b$ then $b$ remains a base of $\mathcal{M}_E\circ (E-A).$
 %Since $\mathcal{M}_{SP},\mathcal{M}_{PQ},$ are compatible over $P,$ by Lemma \ref{lem:baseexchange} (2), there exist disjoint bases $(b_S\uplus b^1_P), (b_Q\uplus b^2_P)$ of $\mathcal{M}_{SP},\mathcal{M}_{PQ},$ respectively which cover $P.$
%Now, by strong compatibility  of $\mathcal{M}_{SP},\mathcal{M}_{PQ}$ over $P,$ we may take the subsets $  b^3_P, b^4_P$ to be contained in  $(b_S\uplus b^1_P), (b_Q\uplus b^2_P)$ respectively and $b^4_P, b^3_P$ to be contained in the cobases $(S\uplus P)- (b_S\uplus b^1_P), (P\uplus Q) -(b_Q\uplus b^2_P)$ of $\mathcal{M}_{SP},\mathcal{M}_{PQ},$ respectively.
%
%====================================================

Since $\mathcal{M}_{SP},\mathcal{M}_{PQ},$ are compatible over $P,$ by Lemma \ref{lem:baseexchange} (2), there exist disjoint bases $(b_S\uplus b^1_P), (b_Q\uplus b^2_P)$ of $\mathcal{M}_{SP},\mathcal{M}_{PQ},$ respectively which cover $P,$
 i.e., $(b_S\uplus b^1_P)\uplus (b_Q\uplus b^2_P)$ is a base of $\mathcal{M}_{SP}\vee\mathcal{M}_{PQ},$ and $(\mathcal{M}_{SP}\vee\mathcal{M}_{PQ})\circ P= P.$ 

Therefore, whenever $b_S\uplus b_Q$ is a base of $\mathcal{M}_{SP}\leftrightarrow_{\mathcal M} \mathcal{M}_{PQ}:\equiv (\mathcal{M}_{SP}\vee\mathcal{M}_{PQ})\times (S\uplus Q),$  and  $(b_S\uplus b^1_P), (b_Q\uplus b^2_P)$ are disjoint bases of $\mathcal{M}_{SP},\mathcal{M}_{PQ},$ respectively, we must have  $b^1_P\uplus b^2_P=P.$ 
Now, by strong compatibility  of $\mathcal{M}_{SP},\mathcal{M}_{PQ}$ over $P,$ we may take the subsets $  b^3_P, b^4_P$ to be contained in  $(b_S\uplus b^1_P), (b_Q\uplus b^2_P)$ respectively and $b^4_P, b^3_P$ to be contained in the cobases $(S\uplus P)- (b_S\uplus b^1_P), (P\uplus Q) -(b_Q\uplus b^2_P)$ of $\mathcal{M}_{SP},\mathcal{M}_{PQ},$ respectively.

%
%
%=================================================

We note that if we contract a subset $b^1$ of a base $b$ of a matroid $\mathcal{M}_E,$ then $b-b^1$ is a base of $\mathcal{M}_E\times (E-b^1)$ and if we delete a subset $A$ of $E-b$ then $b$ remains a base of $\mathcal{M}_E\circ (E-A).$

 The result now follows from the definition of $\mathcal{M}_{S\hat{P}}$ and $\mathcal{M}_{\hat{P}Q}.$

\item  We have:
\begin{align*}
\mathcal{M}_{S\hat{P}}\circ \hat{P} &= [\mathcal{M}_{SP}\times (S\uplus (P- b^3_P))\circ (S\uplus ((P- b^3_P)-b^4_P))]\circ ((P- b^3_P)-b^4_P) \\
&= \mathcal{M}_{SP}\times (S\uplus (P- b^3_P))\circ((P- b^3_P)-b^4_P) = \mathcal{M}_{SP}\circ (P-b^4_P)\times((P- b^3_P)-b^4_P) \\
&= \mathcal{M}_{SP}\circ P\circ (P-b^4_P)\times((P- b^3_P)-b^4_P) = \mathcal{M}^*_{PQ}\circ P\circ (P-b^4_P)\times((P- b^3_P)-b^4_P) \\
&= \mathcal{M}^*_{PQ}\circ (P-b^4_P)\times((P- b^3_P)-b^4_P) = [\mathcal{M}_{PQ}\times (P-b^4_P)\circ((P- b^3_P)-b^4_P)]^* \\
&= [\mathcal{M}_{PQ}\circ (Q\uplus(P- b^3_P)) \times ((P- b^3_P)-b^4_P)]^* \\
&= [\mathcal{M}_{PQ}\circ (Q\uplus(P- b^3_P)) \times (Q\uplus ((P- b^3_P)-b^4_P))\times ((P- b^3_P)-b^4_P)]^* \\
&= [\mathcal{M}_{PQ}\times (Q\uplus(P- b^4_P)) \circ (Q\uplus ((P- b^3_P)-b^4_P))\times ((P- b^3_P)-b^4_P)]^* \\
&= [\mathcal{M}_{\hat{P}Q}\times ((P- b^3_P)-b^4_P)]^* = \mathcal{M}^*_{\hat{P}Q}\circ ((P- b^3_P)-b^4_P) = \mathcal{M}^*_{\hat{P}Q}\circ \hat{P}.
\end{align*}
The proof of the equality of $\mathcal{M}_{S\hat{P}}\times \hat{P}$ and $\mathcal{M}^*_{\hat{P}Q}\times \hat{P}$ is similar.

\item  Let $b_S\uplus b_Q$ be a base of $\mathcal{M}_{SP}\leftrightarrow_{\mathcal M} \mathcal{M}_{PQ}.$ 
Then by strong compatibility of $\mathcal{M}_{SP}, \mathcal{M}_{PQ}$ over $P,$
there exist disjoint bases $b_S\uplus b^1_{P}, b_Q\uplus b^2_{P}$ of $\mathcal{M}_{SP}, \mathcal{M}_{PQ}$ respectively, which cover $P,$  such that $b^1_P, b^2_P$ contain the disjoint bases $b^3_P, b^4_P$ of $\mathcal{M}_{SP}\times P, \mathcal{M}_{PQ}\times P,$ respectively.
Clearly $b_S\uplus (b^1_P- b^3_P)$ and $b_Q\uplus (b^2_P- b^4_P)$ are disjoint bases respectively of $\mathcal{M}_{S\hat{P}} \equiv [\mathcal{M}_{SP}\times (S\uplus (P- b^3_P))]\circ (S\uplus (P- b^3_P-b^4_P))$ and $\mathcal{M}_{\hat{P}Q} \equiv [\mathcal{M}_{PQ}\times (Q\uplus (P- b^4_P))]\circ (Q\uplus (P- b^3_P-b^4_P)),$ which cover $\hat{P} \equiv (P-b^3_P-b^4_P).$
Therefore $b_S\uplus b_Q$ is a base of $\mathcal{M}_{S\hat{P}}\leftrightarrow_{\mathcal M} \mathcal{M}_{\hat{P}Q}.$

Next let $b_S\uplus b_Q$ be a base of $\mathcal{M}_{S\hat{P}}\leftrightarrow_{\mathcal M} \mathcal{M}_{\hat{P}Q}.$
 We know that $\mathcal{M}_{S\hat{P}}\circ \hat{P}= \mathcal{M}^*_{\hat{P}Q}\circ \hat{P}$ and $\mathcal{M}_{S\hat{P}}\times \hat{P}= \mathcal{M}^*_{\hat{P}Q}\times \hat{P}.$
 Therefore, arguing as in the case of $\mathcal{M}_{SP}, \mathcal{M}_{PQ},$ there exist disjoint bases $b_S\uplus b^1_{\hat{P}}, b_Q\uplus b^2_{\hat{P}}$ of $\mathcal{M}_{S\hat{P}}, \mathcal{M}_{\hat{P}Q},$ respectively such that $b^1_{\hat{P}}\uplus b^2_{\hat{P}}$ covers $\hat{P}.$ Now $b_S\uplus (b^1_{\hat{P}}\uplus b^3_P)$ and $b_Q\uplus (b^2_{\hat{P}}\uplus b^4_P)$ are disjoint bases of $\mathcal{M}_{SP}, \mathcal{M}_{PQ},$ respectively which cover $P= \hat{P}\uplus b^3_P \uplus b^4_P.$ Therefore, $b_S\uplus b_Q$ is a base of $(\mathcal{M}_{SP}\vee \mathcal{M}_{PQ})\times (S\uplus Q)= \mathcal{M}_{SP}\leftrightarrow_{\mathcal M} \mathcal{M}_{PQ}.$
 Thus $\mathcal{M}_{S\hat{P}}\leftrightarrow_{\mathcal M} \mathcal{M}_{\hat{P}Q}= \mathcal{M}_{SP}\leftrightarrow_{\mathcal M} \mathcal{M}_{PQ}.$

\item  We have $\mathcal{M}_{SP}\circ P= \mathcal{M}^*_{PQ}\circ P, \mathcal{M}_{SP}\times P= \mathcal{M}^*_{PQ}\times P,$ $\mathcal{M}_{S\hat{P}}\circ \hat{P}= \mathcal{M}^*_{\hat{P}Q}\circ \hat{P}, \mathcal{M}_{S\hat{P}}\times \hat{P}= \mathcal{M}^*_{\hat{P}Q}\times \hat{P}$ and by part 3 above, $\mathcal{M}_{S\hat{P}}\leftrightarrow_{\mathcal M} \mathcal{M}_{\hat{P}Q}= \mathcal{M}_{SP}\leftrightarrow_{\mathcal M} \mathcal{M}_{PQ}.$
Therefore by Lemma~\ref{cor:minorspecial}, we have that $\mathcal{M}_{SP}\circ S= \mathcal{M}_{S\hat{P}}\circ S= \mathcal{M}_{SQ}\circ S, \mathcal{M}_{SP}\times S= \mathcal{M}_{S\hat{P}}\times S=\mathcal{M}_{SQ}\times S.$

Let $b_S$ be a base of $\mathcal{M}_{SP}\circ S= \mathcal{M}_{S\hat{P}}\circ S.$ Since $b^3_P$ is a base of $\mathcal{M}_{SP}\times P,$ $b_S\uplus b^3_P$ is a base of $\mathcal{M}_{SP}$ and $(b_S\uplus b^3_P)-b^3_P= b_S$ is a base of $\mathcal{M}_{S\hat{P}}.$ Therefore there is a cobase of $\mathcal{M}_{S\hat{P}}$ that contains $\hat{P}.$

Next $b^4_P$ is a base of $\mathcal{M}_{PQ}\times P,$ i.e., a cobase of $(\mathcal{M}_{PQ}\times P)^*= \mathcal{M}^*_{PQ}\circ P= \mathcal{M}_{SP}\circ P.$ Therefore $b_P \equiv P-b^4_P$ is a base of $\mathcal{M}_{SP}\circ P.$ If $b'_S$ is a base of $\mathcal{M}_{SP}\times S = \mathcal{M}_{S\hat{P}}\times S,$ then $b'_S\uplus b_P$ is a base of $\mathcal{M}_{SP}$ and $(b'_S\uplus b_P)- b^3_P= b'_S\uplus ((P-b^4_P) -b^3_P)= b'_S\uplus \hat{P}$ is a base of $\mathcal{M}_{S\hat{P}}.$ Thus there is a base of $\mathcal{M}_{S\hat{P}}$ that contains $\hat{P}.$

Similarly we can show that there is a cobase of $\mathcal{M}_{\hat{P}Q}$ that contains $\hat{P}$ and a base of $\mathcal{M}_{\hat{P}Q}$ that contains $\hat{P}.$

\item  We have $r(\mathcal{M}_{S\hat{P}}\times \hat{P})=0,$ so that $|\hat{P}|= r(\mathcal{M}_{S\hat{P}}\circ \hat{P}) - r(\mathcal{M}_{S\hat{P}}\times \hat{P})= r(\mathcal{M}_{S\hat{P}}\circ S) - r(\mathcal{M}_{S\hat{P}}\times S)= r(\mathcal{M}_{SQ}\circ S) - r(\mathcal{M}_{SQ}\times S) =r(\mathcal{M}_{SQ}\circ Q) - r(\mathcal{M}_{SQ}\times Q),$ using Lemma~\ref{cor:minorspecial}. By Theorem~\ref{thm:lowerboundP}, it follows that $|\hat{P}|$ is the minimum value of $|P|$ for which $\mathcal{M}_{SQ}= \mathcal{M}_{SP}\leftrightarrow_{\mathcal M} \mathcal{M}_{PQ}.$
\end{enumerate}
\end{proof}

\section{$\{S,Q\}$--augmentation of $\mathcal{M}_{SQ}$}
\label{subsec:bipartitioncompletion}

In this section we show that the $\{S,Q\}$--augmentation of a matroid $\mathcal{M}_{SQ}$
is itself a matroid and show that being $\{S,Q\}$--complete is a sufficient condition for $\mathcal{M}_{SQ}$ to have a minimal
decomposition $\mathcal{M}_{SQ} = \mathcal{M}_{SP} \leftrightarrow_{\mathcal{M}} \mathcal{M}_{PQ},$ with size $|P| = r(\mathcal{M}_{SQ} \circ S) - r(\mathcal{M}_{SQ} \times S)$ (Theorem \ref{thm:sqcloseddecomposition}).
It should be noted, however, that the $\{S,Q\}$--augmentation of a matroid $\mathcal{M}_{SQ}$ may not be $\{S,Q\}$--complete. A counterexample is given in the appendix.

We give below, in Theorem \ref{thm:pseudoid},
a simple expression for the $\{S,Q\}$--augmentation of a matroid $\mathcal{M}_{SQ}$ in terms of
the linking operation with the matroid $\mathcal{M}_{QQ'} :\equiv \mathcal{M}^*_{SQ} \leftrightarrow_{\mathcal{M}} (\mathcal{M}_{SQ})_{SQ'}.$
This is analogous to the vector space $\mathcal{V}_{QQ'} :\equiv \mathcal{V}_{SQ} \leftrightarrow_{\mathcal{V}} (\mathcal{V}_{SQ})_{SQ'}.$ For vector spaces we have the result that
$\mathcal{V}_{SQ} \leftrightarrow_{\mathcal{V}} \mathcal{V}_{QQ'} = (\mathcal{V}_{SQ})_{SQ'}$ \cite{narayanan2016}.
For matroids the corresponding operation $(\mathcal{M}_{SQ} \leftrightarrow_{\mathcal{M}} \mathcal{M}_{QQ'})_{SQ}$ leads to the $\{S,Q\}$--augmentation of $\mathcal{M}_{SQ}.$

\begin{theorem}
\label{thm:pseudoid}
Let $\mathcal{M}_{QQ'} :\equiv \mathcal{M}^*_{SQ} \leftrightarrow_{\mathcal{M}} (\mathcal{M}_{SQ})_{SQ'}.$
We then have the following.
\begin{enumerate}
    \item The $\{S,Q\}$--augmentation of $\mathcal{M}_{SQ}$ is the family of bases of the matroid:
$$(\mathcal{M}_{SQ} \leftrightarrow_{\mathcal{M}} \mathcal{M}_{QQ'})_{SQ}=(\mathcal{M}_{SQ})_{SQ'} \leftrightarrow_{\mathcal{M}} [(\mathcal{M}^*_{SQ})_{S'Q'} \leftrightarrow_{\mathcal{M}} (\mathcal{M}_{SQ})_{S'Q}].
$$
    \item If $\mathcal{M}_{SQ}$ is $\{S,Q\}$--complete, then it is its own $\{S,Q\}$--augmentation.
    \item The dual of the $\{S,Q\}$--augmentation of $\mathcal{M}_{SQ}$ is the $\{S,Q\}$--augmentation of $\mathcal{M}^*_{SQ}.$
    \item If $\mathcal{M}_{SQ}$ is $\{S,Q\}$--complete, $\mathcal{M}_{QQ'}$ is $\{Q,Q'\}$--complete. Further, $\mathcal{E}_{Q'}(\mathcal{M}_{QQ'}) = \mathcal{E}_{Q'}((\mathcal{M}_{SQ})_{SQ'})$ and $\mathcal{E}_Q(\mathcal{M}_{QQ'}) = \mathcal{E}_Q(\mathcal{M}^*_{SQ}).$
\end{enumerate}
\end{theorem}

We need the following lemma for the proof of the theorem.

\begin{lemma}
\label{lem:pseudoid0}
Let $\mathcal{M}_{QQ'} :\equiv \mathcal{M}^*_{SQ} \leftrightarrow_{\mathcal{M}} (\mathcal{M}_{SQ})_{SQ'}.$
We then have the following.
\begin{enumerate}
    \item $\mathcal{M}^*_{QQ'} = (\mathcal{M}_{QQ'})_{Q'Q}.$
    \item $\mathcal{M}_{QQ'} \circ Q = \mathcal{M}^*_{SQ} \circ Q, \quad \mathcal{M}_{QQ'} \times Q = \mathcal{M}^*_{SQ} \times Q.$
    \item $(\mathcal{M}_{SQ} \leftrightarrow_{\mathcal{M}} \mathcal{M}_{QQ'}) \circ S = \mathcal{M}_{SQ} \circ S, \quad (\mathcal{M}_{SQ} \leftrightarrow_{\mathcal{M}} \mathcal{M}_{QQ'}) \circ Q' = (\mathcal{M}_{SQ} \circ Q)_{Q'}.$
    \item $(\mathcal{M}_{SQ} \leftrightarrow_{\mathcal{M}} \mathcal{M}_{QQ'}) \times S = \mathcal{M}_{SQ} \times S, \quad (\mathcal{M}_{SQ} \leftrightarrow_{\mathcal{M}} \mathcal{M}_{QQ'}) \times Q' = (\mathcal{M}_{SQ} \times Q)_{Q'}.$
    \item Every base of $(\mathcal{M}_{SQ})_{SQ'}$ is a base of $(\mathcal{M}_{SQ} \leftrightarrow_{\mathcal{M}} \mathcal{M}_{QQ'}).$
\end{enumerate}
\end{lemma}
The proof of the lemma is relegated to \ref{sec:pseudoid}.

\begin{proof} (of Theorem \ref{thm:pseudoid})
\begin{enumerate}
\item  We first note that:
$\mathcal{M}_{SQ} \leftrightarrow_{\mathcal{M}} \mathcal{M}_{QQ'} :\equiv \mathcal{M}_{SQ} \leftrightarrow_{\mathcal{M}} (\mathcal{M}^*_{SQ} \leftrightarrow_{\mathcal{M}} (\mathcal{M}_{SQ})_{SQ'}) = \mathcal{M}_{SQ} \leftrightarrow_{\mathcal{M}} (\mathcal{M}^*_{S'Q} \leftrightarrow_{\mathcal{M}} (\mathcal{M}_{S'Q})_{S'Q'}),$
where by restricted associativity of the linking operation (Lemma \ref{lem:restassoc}), the last expression has a unique meaning even if we remove the brackets.
We need to show that the set of bases of $\mathcal{M}_{SQ} \leftrightarrow_{\mathcal{M}} \mathcal{M}_{QQ'} = \mathcal{M}_{SQ} \leftrightarrow_{\mathcal{M}} \mathcal{M}^*_{S'Q} \leftrightarrow_{\mathcal{M}} (\mathcal{M}_{S'Q})_{S'Q'}$
contain copies of bases of $\mathcal{M}_{SQ}$ and only those additional bases we require for the augmentation of the set of bases of $\mathcal{M}_{SQ}.$

Let $b_S \uplus b_Q$ be a base of $\mathcal{M}_{SQ}.$ We will show that $b_S \uplus b_{Q'}$ is a base of $\mathcal{M}_{SQ} \leftrightarrow_{\mathcal{M}} \mathcal{M}_{QQ'}.$

We have $b_{S} \uplus {b}_{Q}$ and $(S' - b_{S'}) \uplus (Q - b_Q)$ as bases of $\mathcal{M}_{SQ}$ and $(\mathcal{M}_{SQ}^*)_{S'Q},$ respectively.
Therefore we must have $b_S \uplus (S' - b_{S'})$ as a base of $(\mathcal{M}_{SQ} \vee (\mathcal{M}^*_{SQ})_{S'Q}) \times (S \uplus S') = \mathcal{M}_{SQ} \leftrightarrow_{\mathcal{M}} (\mathcal{M}^*_{SQ})_{S'Q}.$
Since $b_S \uplus (S' - b_{S'})$ and $b_{S'} \uplus b_{Q'}$ are bases of $\mathcal{M}_{SQ} \leftrightarrow_{\mathcal{M}} (\mathcal{M}^*_{SQ})_{S'Q}$ and $(\mathcal{M}_{SQ})_{S'Q'},$ respectively, it follows that
$b_S \uplus b_{Q'}$ is a base of $\mathcal{M}_{SQ} \leftrightarrow_{\mathcal{M}} \mathcal{M}_{QQ'} = (\mathcal{M}_{SQ} \leftrightarrow_{\mathcal{M}} (\mathcal{M}^*_{SQ})_{S'Q}) \leftrightarrow_{\mathcal{M}} (\mathcal{M}_{SQ})_{S'Q'}.$

Let $b_S \uplus b_Q, \, \hat{b}_S \uplus b_Q, \, b_S \uplus \hat{b}_Q$ be bases of $\mathcal{M}_{SQ}.$ We will show that $\hat{b}_S \uplus \hat{b}_{Q'}$ is a base of $\mathcal{M}_{SQ} \leftrightarrow_{\mathcal{M}} \mathcal{M}_{QQ'} = \mathcal{M}_{SQ} \leftrightarrow_{\mathcal{M}} (\mathcal{M}^*_{SQ})_{S'Q} \leftrightarrow_{\mathcal{M}} (\mathcal{M}_{SQ})_{S'Q'},$ where $\hat{b}_{Q'}$ is a copy of $\hat{b}_Q.$

We have bases $(S' - b_{S'}) \uplus (Q - b_Q)$ and $b_{S'} \uplus \hat{b}_{Q'}$ of matroids $(\mathcal{M}_{SQ})^*_{S'Q}$ and $(\mathcal{M}_{SQ})_{S'Q'},$ respectively, $b_{S'}$ and $\hat{b}_{Q'}$ being copies of $b_S$ and $\hat{b}_Q,$ respectively.
Since $\hat{b}_S \uplus b_Q$ is a base of $\mathcal{M}_{SQ}$ and $(S' - b_{S'}) \uplus (Q - b_Q)$ is a base of $(\mathcal{M}^*_{SQ})_{S'Q},$ we have that $\hat{b}_S \uplus (S' - b_{S'})$ is a base of $(\mathcal{M}_{SQ} \vee (\mathcal{M}^*_{SQ})_{S'Q}) \times (S \uplus S') = \mathcal{M}_{SQ} \leftrightarrow_{\mathcal{M}} (\mathcal{M}^*_{SQ})_{S'Q}.$
Since $b_{S'} \uplus \hat{b}_{Q'}$ is a base of $(\mathcal{M}_{SQ})_{S'Q'},$ it follows that $\hat{b}_S \uplus \hat{b}_{Q'}$ is a base of $((\mathcal{M}_{SQ} \leftrightarrow_{\mathcal{M}} (\mathcal{M}^*_{SQ})_{S'Q}) \vee (\mathcal{M}_{SQ})_{S'Q'}) \times (S \uplus Q') = (\mathcal{M}_{SQ} \leftrightarrow_{\mathcal{M}} (\mathcal{M}^*_{SQ})_{S'Q}) \leftrightarrow_{\mathcal{M}} (\mathcal{M}_{SQ})_{S'Q'}.$

Next we will show that if $\hat{b}_S \uplus \hat{b}_{Q'}$ is a base of $\mathcal{M}_{SQ} \leftrightarrow_{\mathcal{M}} \mathcal{M}_{QQ'},$ then there exist bases $\hat{b}_S \uplus b_Q, \, b_{S} \uplus b_Q, \, b_S \uplus \hat{b}_Q$ of $\mathcal{M}_{SQ}.$

We have that $\mathcal{M}_{SQ}$ and $(\mathcal{M}^*_{SQ})_{S'Q}$ are compatible. By Lemma \ref{cor:minorspecial}, we therefore have $(\mathcal{M}_{SQ} \leftrightarrow_{\mathcal{M}} (\mathcal{M}^*_{SQ})_{S'Q}) \circ S' = (\mathcal{M}^*_{SQ})_{S'Q} \circ S'$ and $(\mathcal{M}_{SQ} \leftrightarrow_{\mathcal{M}} (\mathcal{M}^*_{SQ})_{S'Q}) \times S' = (\mathcal{M}^*_{SQ})_{S'Q} \times S'.$ Therefore, $(\mathcal{M}_{SQ} \leftrightarrow_{\mathcal{M}} (\mathcal{M}^*_{SQ})_{S'Q})$ and $(\mathcal{M}_{SQ})_{S'Q'}$ are compatible.

Since $\hat{b}_S \uplus \hat{b}_{Q'}$ is a base of $(\mathcal{M}_{SQ} \leftrightarrow_{\mathcal{M}} (\mathcal{M}^*_{SQ})_{S'Q}) \leftrightarrow_{\mathcal{M}} (\mathcal{M}_{SQ})_{S'Q'},$ by Lemma \ref{lem:baseexchange}, we may assume that there are disjoint bases $B_{SS'}$ and $b_{S'} \uplus \hat{b}_{Q'}$ of the compatible pair $(\mathcal{M}_{SQ} \leftrightarrow_{\mathcal{M}} (\mathcal{M}^*_{SQ})_{S'Q})$ and $(\mathcal{M}_{SQ})_{S'Q'},$ respectively, such that $B_{SS'} \uplus (b_{S'} \uplus \hat{b}_{Q'})$ covers $S'$ and $(B_{SS'} \uplus (b_{S'} \uplus \hat{b}_{Q'})) - S' = \hat{b}_S \uplus \hat{b}_{Q'}.$

Let $B_{SS'} = (S' - b_{S'}) \uplus \hat{b}_S.$ Again, by Lemma \ref{lem:baseexchange}, there must be disjoint bases $\hat{b}_S \uplus b_Q$ and $(S' - b_{S'}) \uplus (Q - b_Q)$ of the compatible pair $\mathcal{M}_{SQ}$ and $(\mathcal{M}^*_{SQ})_{S'Q},$ respectively, so that their union covers $Q$ and $(\hat{b}_S \uplus b_Q) \uplus ((S' - b_{S'}) \uplus (Q - b_Q)) - Q = (S' - b_{S'}) \uplus \hat{b}_S.$
From this it follows that there exist bases $\hat{b}_S \uplus b_Q$ and $(S - b_{S}) \uplus (Q - b_Q)$ of $\mathcal{M}_{SQ}$ and $\mathcal{M}^*_{SQ},$ respectively, and therefore bases $\hat{b}_S \uplus b_Q$ and $b_{S} \uplus b_Q$ of $\mathcal{M}_{SQ}.$
We already have ${b}_{S'} \uplus \hat{b}_{Q'}$ as a base of $(\mathcal{M}_{SQ})_{S'Q'}.$ Therefore there exist bases $\hat{b}_S \uplus b_Q, \, b_{S} \uplus b_Q, \, b_S \uplus \hat{b}_Q$ of $\mathcal{M}_{SQ},$ as needed.

Thus the $\{S,Q\}$--augmentation of $\mathcal{M}_{SQ}$ is the family of bases of the matroid $(\mathcal{M}_{SQ} \leftrightarrow_{\mathcal{M}} \mathcal{M}_{QQ'})_{SQ}.$

\item  This is immediate from the above.

\item  By part 1 of this theorem we have that the $\{S,Q\}$--augmentation of $\mathcal{M}_{SQ}$ is equal to $(\mathcal{M}_{SQ})_{SQ'} \leftrightarrow_{\mathcal{M}} (\mathcal{M}^*_{SQ})_{S'Q'} \leftrightarrow_{\mathcal{M}} (\mathcal{M}_{SQ})_{S'Q}.$
The dual of the $\{S,Q\}$--augmentation of $\mathcal{M}_{SQ}$ is 

$[(\mathcal{M}_{SQ})_{SQ'} \leftrightarrow_{\mathcal{M}} (\mathcal{M}^*_{SQ})_{S'Q'} \leftrightarrow_{\mathcal{M}} (\mathcal{M}_{SQ})_{S'Q}]^*= (\mathcal{M}_{SQ})^*_{SQ'} \leftrightarrow_{\mathcal{M}} (\mathcal{M}^*_{SQ})^*_{S'Q'} \leftrightarrow_{\mathcal{M}} (\mathcal{M}_{SQ})^*_{S'Q},$ using Theorem \ref{thm:idtmatroid}. 
The latter expression is obtained by replacing $\mathcal{M}_{SQ},$ wherever it occurs in the expression for the $\{S,Q\}$--augmentation of $\mathcal{M}_{SQ},$ by $\mathcal{M}^*_{SQ}$ and is therefore the expression for $\{S,Q\}$--augmentation of $\mathcal{M}^*_{SQ}.$

\item  We first note that, from parts 2 and 3 above, since $\mbox{$\cal M$}_{SQ}$ is given to be $\{S,Q\}-$complete, it follows that $\mbox{$\cal M$}^*_{SQ}$ is also  $\{S,Q\}-$complete.
Let $(Q-b^1_Q)\uplus b^2_{Q'},(Q-b^3_Q)\uplus b^4_{Q'},(Q-b^1_Q)\uplus b^4_{Q'}$ be bases in $\mbox{$\cal M$}_{QQ'}.$ We need to show that $(Q-b^3_Q)\uplus b^2_{Q'}$ is also a base of $\mbox{$\cal M$}_{QQ'}.$ Since $\mbox{$\cal M$}_{QQ'}:\equiv  \mbox{$\cal M$}^*_{SQ}\leftrightarrow_{\mathcal M} (\mbox{$\cal M$}_{SQ})_{SQ'}$ 
 and $(Q-b^1_Q)\uplus b^2_{Q'}$ is a base of $\mbox{$\cal M$}_{QQ'},$ we must have, using Lemma \ref{lem:baseexchange}, 
 bases $(Q-b^1_Q)\uplus (S-b^a_S), b^a_S\uplus b^2_{Q'}$ of $\mbox{$\cal M$}^*_{SQ}, (\mbox{$\cal M$}_{SQ})_{SQ'},$ respectively. Similarly we must have  bases $(Q-b^3_Q)\uplus (S-b^b_S),(Q-b^1_Q)\uplus (S-b^c_S)$ of $\mbox{$\cal M$}^*_{SQ}$ and  bases $b^b_S\uplus b^4_{Q'}, b^c_S\uplus b^4_{Q'}$ of $(\mbox{$\cal M$}_{SQ})_{SQ'}$
 respectively. 
 Since $(Q-b^1_Q)\uplus (S-b^a_S),(Q-b^1_Q)\uplus (S-b^c_S)$ are bases of
$ \mbox{$\cal M$}^*_{SQ}, $ it follows that $(S-b^a_S),(S-b^c_S)$ belong to the same block of $\mbox{$\cal E$}_S(\mbox{$\cal M$}^*_{SQ})$ and therefore
 $ b^a_S,b^c_S$ belong to the same block of $\mbox{$\cal E$}_S(\mbox{$\cal M$}_{SQ}).$  
 Since $b^b_S\uplus b^4_{Q'}, b^c_S\uplus b^4_{Q'}$ are bases of $(\mbox{$\cal M$}_{SQ})_{SQ'},$
  it follows that $ b^b_S,b^c_S$ belong to the same block of $\mbox{$\cal E$}_S(\mbox{$\cal M$}_{SQ}).$ 
 Thus $b^a_S,b^b_S$ belong to the same block of $\mbox{$\cal E$}_S(\mbox{$\cal M$}_{SQ}).$
Since $(Q-b^3_Q)\uplus (S-b^b_S)$ is a base of $\mbox{$\cal M$}^*_{SQ}$ it follows that 
 $(Q-b^3_Q)\uplus (S-b^a_S)$ is also a base of $\mbox{$\cal M$}^*_{SQ}.$ 
 Now $b^a_S\uplus b^2_{Q'}$ is a base of $\mbox{$\cal M$}_{SQ}$ so that $(Q-b^3_Q)\uplus b^2_{Q'}$ is a base of $\mbox{$\cal M$}_{QQ'} :\equiv  \mbox{$\cal M$}^*_{SQ}\leftrightarrow_{\mathcal M} (\mbox{$\cal M$}_{SQ})_{SQ'}.$
 
Thus $\mbox{$\cal M$}_{QQ'}$ is $\{Q,Q'\}-$complete.

Now let $b^3_{Q'}, b^4_{Q'}$ belong to the same block of $\mbox{$\cal E$}_{Q'}((\mbox{$\cal M$}_{SQ})_{SQ'}),$
 i.e., $b^3_{Q}, b^4_{Q}$ belong to the same block of $\mbox{$\cal E$}_{Q}(\mbox{$\cal M$}_{SQ}).$
This means that there exists an independent set $b^1_S$ in $\mbox{$\cal M$}_{SQ},$ such that
 $b^1_S\uplus b^3_{Q}, b^1_S\uplus b^4_{Q},$
 are bases of $\mbox{$\cal M$}_{SQ},$ i.e.,  $b^1_S\uplus b^3_{Q'}, b^1_S\uplus b^4_{Q'},$
 are bases of $(\mbox{$\cal M$}_{SQ})_{SQ'}$ and 
$(Q-b^3_{Q})\uplus (S-b^1_S)$ is a  base of
$\mbox{$\cal M$}^*_{SQ}.$ Therefore $ (Q-b^3_{Q})\uplus b^3_{Q'}, (Q-b^3_{Q})\uplus b^4_{Q'}$ are  bases of $\mbox{$\cal M$}_{QQ'}:\equiv  \mbox{$\cal M$}^*_{SQ}\leftrightarrow_{\mathcal M} (\mbox{$\cal M$}_{SQ})_{SQ'}.$
Therefore $b^3_{Q'},b^4_{Q'}$ belong  to the same block of $\mbox{$\cal E$}_{Q'}(\mbox{$\cal M$}_{QQ'}).$  
\\Next let $b^3_{Q'},b^4_{Q'}$ belong  to the same block of $\mbox{$\cal E$}_{Q'}(\mbox{$\cal M$}_{QQ'}),$
 i.e., there exist bases 
$ (Q-b^1_{Q})\uplus b^3_{Q'}, (Q-b^1_{Q})\uplus b^4_{Q'}$ of $\mbox{$\cal M$}_{QQ'}:\equiv  \mbox{$\cal M$}^*_{SQ}\leftrightarrow_{\mathcal M} (\mbox{$\cal M$}_{SQ})_{SQ'}.$
Then there exist subsets $b^a_S, b^b_S$ such that $(S-b^a_S)\uplus (Q-b^1_Q), (S-b^b_S)\uplus (Q-b^1_Q)$ are bases of $\mbox{$\cal M$}^*_{SQ}$ and $b^a_S\uplus b^3_{Q'},
 b^b_S\uplus b^4_{Q'} $ are bases of $(\mbox{$\cal M$}_{SQ})_{SQ'}.$
 This means $(S-b^a_S), (S-b^b_S)$ belong to the same block of $\mbox{$\cal E$}_S(\mbox{$\cal M$}^*_{SQ})$ so that $b^a_S, b^b_S$ belong to the same block of $\mbox{$\cal E$}_S((\mbox{$\cal M$}_{SQ})_{SQ'}).$
 Therefore $  b^3_{Q'},  b^4_{Q'}$ belong to the same block of $\mbox{$\cal E$}_{Q'}((\mbox{$\cal M$}_{SQ})_{SQ'}).$  
\\
Thus $\mbox{$\cal E$}_{Q'}((\mbox{$\cal M$}_{SQ})_{SQ'})= \mbox{$\cal E$}_{Q'}(\mbox{$\cal M$}_{QQ'}).$\\
Now  $b_S \uplus b_Q$ is a base of $\mathcal{M}_{SQ}$ iff $(S - b_S) \uplus (Q - b_Q)$ is a base of $\mathcal{M}^*_{SQ}.$
It follows that $\mbox{$\cal E$}_{Q'}((\mbox{$\cal M$}^*_{SQ})_{SQ'})= \mbox{$\cal E$}_{Q'}(\mbox{$\cal M$}^*_{QQ'}),$ i.e.,
$\mbox{$\cal E$}_{Q}(\mbox{$\cal M$}^*_{SQ})=\mbox{$\cal E$}_{Q}((\mbox{$\cal M$}^*_{QQ'})_{Q'Q})=  \mbox{$\cal E$}_{Q}(\mbox{$\cal M$}_{QQ'}),$ 
 using  Lemma \ref{lem:pseudoid0}  part 1.

\end{enumerate}
\end{proof}

\subsection{Minimization of decomposition of $\{S,Q\}$--complete $\mathcal{M}_{SQ}$}
\label{subsec:minimalbipartition}

Minimal decomposition of a matroid $\mathcal{M}_{SQ}$ into $\mathcal{M}_{SP} \leftrightarrow_{\mathcal{M}} \mathcal{M}_{PQ}$ is not always possible (see \ref{sec:mindecompmatroid}).
But if the matroid $\mathcal{M}_{SQ}$ is $\{S,Q\}$--complete, there is a simple way of minimally decomposing it into $\mathcal{M}_{SP} \leftrightarrow_{\mathcal{M}} \mathcal{M}_{PQ},$ with the additional condition that the matroids $\mathcal{M}_{SP}$ and $\mathcal{M}_{PQ}$ are respectively $\{S,P\}$--complete and $\{P,Q\}$--complete.
To reach this goal we need preliminary results which are useful in themselves.

It is a useful fact that the minors of $\{S,Q\}$--complete matroids are themselves complete with respect to suitable partitions of the underlying subsets. This we prove next.

\begin{theorem}
\label{lem:sqrestrictcontract}
Let matroid $\mathcal{M}_{SQ}$ be $\{S,Q\}$--complete. Let $S^2 \subseteq S^1 \subseteq S$ and $Q^2 \subseteq Q^1 \subseteq Q.$ We have the following.
\begin{enumerate}
    \item Let $S \uplus Q - (S^1 \uplus Q^1)$ contain no bonds of $\mathcal{M}_{SQ}.$ Then $\mathcal{M}_{S^1Q^1} :\equiv \mathcal{M}_{SQ} \circ (S^1 \uplus Q^1)$ is $\{S^1,Q^1\}$--complete. 

Further, if $Q = Q^1,$  and if $I_j$ is a block of $\mathcal{E}_S(\mathcal{M}_{SQ}),$ 
 then $I^1_j:\equiv \{X\in I_j, X\subseteq S^1\}$ is a block 
 of $\mathcal{E}_{S^1}(\mathcal{M}_{S^1Q^1}),$ iff $I^1_j$ is nonempty.

    \item Let $S \uplus Q - (S^1 \uplus Q^1)$ contain no circuits of $\mathcal{M}_{SQ}.$ Then $\mathcal{M}_{S^1Q^1} :\equiv \mathcal{M}_{SQ} \times (S^1 \uplus Q^1)$ is $\{S^1,Q^1\}$--complete. Further, if $Q = Q^1$ 
and if $I_j$ is a block of $\mathcal{E}_S(\mathcal{M}_{SQ}),$ 
then $I^1_j:\equiv \{X, (X \uplus (S-S^1))\in I_j, X\subseteq S^1\} $ is a  block of $\mathcal{E}_{S^1}(\mathcal{M}_{S^1Q^1}),$ 
  iff $I^1_j$ is nonempty.
    
    \item Let $S \uplus Q - (S^1 \uplus Q^1)$ contain no bonds of $\mathcal{M}_{SQ}$ and let $S^1 \uplus Q^1 - (S^2 \uplus Q^2)$ contain no circuits of $\mathcal{M}_{SQ}.$ Then $\mathcal{M}_{S^2Q^2} :\equiv \mathcal{M}_{SQ} \circ (S^1 \uplus Q^1) \times (S^2 \uplus Q^2)$ is $\{S^2,Q^2\}$--complete.  Further, if $Q = Q^2$ and if $I_j$ is a block of $\mathcal{E}_S(\mathcal{M}_{SQ}),$ 
$I^2_j:\equiv \{X, (X \uplus (S^1-S^2))\in I_j, X\subseteq S^2\} $ 
is a  block of $\mathcal{E}_{S^2}(\mathcal{M}_{S^2Q^2}),$ 
iff $I^2_j$ is nonempty.   If for each $I_j$ in $\mathcal{E}_S(\mathcal{M}_{SQ}),$ $I^2_j$ is nonempty, then $\mathcal{E}_Q(\mathcal{M}_{SQ})= \mathcal{E}_{Q^2}(\mathcal{M}_{S^2Q^2}).$

    \item Let $S^2 \subseteq S^1 \subseteq S$ and $Q^2 \subseteq Q^1 \subseteq Q.$ Then $\mathcal{M}_{S^2Q^2} :\equiv \mathcal{M}_{SQ} \circ (S^1 \uplus Q^1) \times (S^2 \uplus Q^2)$ is $\{S^2,Q^2\}$--complete.
% Further, if $S=S^1=S^2,$ then 
%$\mathcal{E}_S(\mathcal{M}_{SQ})= \mathcal{E}_{S^2}(\mathcal{M}_{S^2Q^2}).$
\end{enumerate}
\end{theorem}
\begin{proof}
\begin{enumerate}
    \item Since the subset $(S \uplus Q) - (S^1 \uplus Q^1)$ contains no bonds of $\mathcal{M}_{SQ},$ it can be included in a cobase of $\mathcal{M}_{SQ}.$ Therefore $S^1 \uplus Q^1$ contains a base of $\mathcal{M}_{SQ}.$ Therefore bases of $\mathcal{M}_{S^1Q^1}$ are precisely those bases of $\mathcal{M}_{SQ}$ that are contained in $S^1 \uplus Q^1.$
    
    Now $\mathcal{M}_{SQ}$ is $\{S,Q\}$--complete. So if $b^1_{S^1} \uplus b^1_{Q^1}, \, b^1_{S^1} \uplus b^2_{Q^1}, \, b^2_{S^1} \uplus b^1_{Q^1}$ are bases of $\mathcal{M}_{S^1Q^1},$ they are bases of $\mathcal{M}_{SQ},$ so that $b^2_{S^1} \uplus b^2_{Q^1}$ must be a base of $\mathcal{M}_{SQ}$ that is contained in $S^1 \uplus Q^1.$ Therefore $b^2_{S^1} \uplus b^2_{Q^1}$ must be a base of $\mathcal{M}_{S^1Q^1}.$ Thus $S^1 \uplus Q^1$ is $\{S^1,Q^1\}$--complete.
    
    The statement about blocks of $\mathcal{E}_{S^1}(\mathcal{M}_{S^1Q^1})$ is immediate.

    \item This follows by duality from the previous part. If the subset $S \uplus Q - (S^1 \uplus Q^1)$ contains no circuits of $\mathcal{M}_{SQ},$ then it contains no bonds of $\mathcal{M}^*_{SQ}.$ By part 3 of Theorem \ref{thm:pseudoid}, $\mathcal{M}^*_{AB}$ on $A \uplus B$ is $\{A,B\}$--complete iff $\mathcal{M}_{AB}$ is $\{A,B\}$--complete. Since $(\mathcal{M}^*_{SQ} \circ (S^1 \uplus Q^1))^* = \mathcal{M}_{SQ} \times (S^1 \uplus Q^1),$ the result follows.
    
    The statement about blocks of $\mathcal{E}_{S^1}(\mathcal{M}_{S^1Q^1})$ is immediate.

    \item We note that the subset $S^1 \uplus Q^1 - (S^2 \uplus Q^2)$ contains no circuits of $\mathcal{M}_{SQ}$ iff it contains no circuits of $\mathcal{M}_{SQ} \circ (S^1 \uplus Q^1).$ The result now follows from the previous two parts of the lemma.

The statement about blocks of $\mathcal{E}_{S^2}(\mathcal{M}_{S^2Q^2})$ is immediate.

Let $Q=Q^2$ and let $I_j,\hat{I}_j$ be corresponding blocks of $\mathcal{E}_S(\mathcal{M}_{SQ}),\mathcal{E}_Q(\mathcal{M}_{SQ})$ respectively such that for $b_S\in I_j,
 b_Q\in \hat{I}_j,$ we have $b_S\uplus b_Q$ as a base of $\mathcal{M}_{SQ}.$
 If $I^2_j$ is nonempty, it is clear that for some $b'_{S_2}\in I^2_j,$ 
 $b'_{S_2}\uplus b_Q$ is a base of $ \mathcal{E}_{S^2}(\mathcal{M}_{S^2Q^2}).$
 It follows that $I^2_j,\hat{I}_j$ are corresponding blocks of $\mathcal{E}_{S^2}(\mathcal{M}_{S^2Q^2}),  \mathcal{E}_{Q^2}(\mathcal{M}_{S^2Q^2}).$
 such that for $b_{S_2}\in I^2_j,
 b_Q\in \hat{I}_j,$ we have $b_{S_2}\uplus b_Q$ as a base of $\mathcal{M}_{S^2Q^2}.$
 We are given that for each $I_j$ the block $I^2_j$ is nonempty.
 Therefore blocks of $\mathcal{E}_Q(\mathcal{M}_{SQ})$ remain as blocks of $ \mathcal{E}_{Q^2}(\mathcal{M}_{S^2Q^2}).$

    \item Any minor can be obtained by deleting or contracting elements one at a time. We apply parts 1 and 2 one element at a time.

    Suppose the first element $e$ is to be deleted (contracted). If $e$ is not a self loop or a coloop, we delete (contract) it. If it is a self loop or coloop, contracting or deleting leads to the same matroid. A self loop contains no bonds, so we delete it. A coloop contains no circuits, so we contract it.
    
    The resulting matroid is bipartition complete with respect to the partition obtained from $\{S,Q\}$ by deleting element $e$ from the block that contains it. We repeat this process until all the elements not in the underlying set of the minor are dealt with by deletion or contraction as above. At every stage we have appropriate bipartition complete matroids, so that the final minor $\mathcal{M}_{S^2Q^2}$ is $\{S^2,Q^2\}$--complete.

\end{enumerate}
\end{proof}
The following corollary is immediate. We note that every matroid $\mathcal{M}_Q$ is trivially $\{Q,\emptyset\}$--complete.

\begin{corollary}
\label{cor:partitioncomplete}

Let $\mathcal{M}_{T}$  be $\{S_j,P_j\}-$ complete for every partition $\{S_j,P_j\} , j = 1, \dots, n,$ of $T.$
 We have the following.
\begin{enumerate}
\item Let $\mathcal{M}_Q,$  $Q \subseteq T,$ be a minor of $\mathcal{M}_{T}.$ Then $\mathcal{M}_Q$ is $\{S_j \cap Q, \, P_j \cap Q\}$--complete for $j = 1, \dots, n.$ 
\item Let $\Pi :\equiv \{Q_1, \dots, Q_k\}$ be a partition of $Q$ such that each $Q_m,Q_n, m\ne n$ is separated by some $\{S_j,P_j\}.$ Then every minor on $Q_m\uplus Q_n
, m\ne n$ is $\{Q_m,Q_n\}-$ complete. 
\end{enumerate}

\end{corollary}

\begin{remark}
We note that the union of two $\{S,Q\}$--complete matroids is not necessarily $\{S,Q\}$--complete. 

Let $S :\equiv \{s_1,s_2,s_3\}, \quad Q :\equiv \{q_1,q_2,q_3\}, \quad P :\equiv \{p\}.$ 
Let $U_{x,y}$ denote the uniform matroid on $y$ elements whose bases contain $x$ elements. Define $\mathcal{M}_{SP} :\equiv U_{2,4}, \quad \mathcal{M}_{PQ} :\equiv U_{3,4}, \quad \mathcal{M}^1_{SPQ} :\equiv \mathcal{M}_{SP} \oplus \mathbf{0}_Q,$ and $\mathcal{M}^2_{SPQ} :\equiv \mathbf{0}_{S} \oplus \mathcal{M}_{PQ}.$

Consider the bipartition $\{A,B\} :\equiv \{\{q_1,s_1,p\}, \, \{q_2,q_3,s_2,s_3\}\}.$ It is easy to see that both $\mathcal{M}^1_{SPQ}$ and $\mathcal{M}^2_{SPQ}$ are $\{A,B\}$--complete. 
But $\mathcal{M}^1_{SPQ} \vee \mathcal{M}^2_{SPQ}$ is not $\{A,B\}$--complete.
Consider $a_1 :\equiv \{q_1,p\}, \quad a_2 :\equiv \{s_1,p\}, \quad b_1 :\equiv \{q_2,q_3,s_3\}, \quad b_2 :\equiv \{q_3,s_2,s_3\}.$ It can be verified that $a_1 \uplus b_1, \, a_1 \uplus b_2, \, a_2 \uplus b_1$ are bases of $\mathcal{M}^1_{SPQ} \vee \mathcal{M}^2_{SPQ}$ but $a_2 \uplus b_2$ is not.
\end{remark}

A natural question that arises is whether $\mathcal{M}_{SP} \leftrightarrow_{\mathcal{M}} \mathcal{M}_{PQ}$ is $\{S,Q\}$--complete if $\mathcal{M}_{SP}$ and $\mathcal{M}_{PQ}$ are $\{S,P\}$--complete and $\{P,Q\}$--complete, respectively. We give a simple sufficient condition below. 

For vector spaces there is the ``implicit inversion theorem'' which states that if $\mathcal{V}_{SQ} = \mathcal{V}_{SP} \leftrightarrow \mathcal{V}_{PQ}$ and $\mathcal{V}_{SP} \circ P \supseteq \mathcal{V}_{PQ} \circ P, \quad \mathcal{V}_{SP} \times P \subseteq \mathcal{V}_{PQ} \times P,$ then $\mathcal{V}_{PQ} = \mathcal{V}_{SP} \leftrightarrow \mathcal{V}_{SQ}$ (\cite{HNarayanan1986a}, \cite{narayanan1987topological}, \cite{narayanan2016}). The theorem below contains an analogous version for bipartition complete matroids. 
 It also gives conditions for $\mathcal{M}_{SP}, \mathcal{M}_{PQ}$ to be strongly compatible over $P$ when $\mathcal{M}_{SP}, \mathcal{M}_{PQ}$ are compatible over $P$ and $\{S,P\}$--complete, $\{P,Q\}$--complete, respectively.

We begin with a simple lemma.

\begin{lemma}
\label{lem:blockunion}
If every block of $\mathcal{E}_P(\mathcal{M}_{PQ})$ is a union of blocks of $\mathcal{E}_P(\mathcal{M}^*_{SP}),$ then every base of $\mathcal{M}_{SP} \vee \mathcal{M}_{PQ}$ is a disjoint union of bases of $\mathcal{M}_{SP}$ and $\mathcal{M}_{PQ},$ and $(\mathcal{M}_{SP} \vee \mathcal{M}_{PQ}) \circ P = \mathcal{F}_P.$ 
\end{lemma}

\begin{proof}
If every block of $\mathcal{E}_P(\mathcal{M}_{PQ})$ is a union of blocks of $\mathcal{E}_P(\mathcal{M}^*_{SP}),$ then there exist disjoint bases of the kind $b'_S \uplus b'_P$ and $(P - b'_P) \uplus b'_Q$ of $\mathcal{M}_{SP}$ and \/ $\mathcal{M}_{PQ},$ respectively. Since all bases of $\mathcal{M}_{SP} \vee \mathcal{M}_{PQ}$ have the same size, every base of $\mathcal{M}_{SP} \vee \mathcal{M}_{PQ}$ is a disjoint union of bases of $\mathcal{M}_{SP}$ and $\mathcal{M}_{PQ}.$ Further, $P$ is independent in the union so that $(\mathcal{M}_{SP} \vee \mathcal{M}_{PQ}) \circ P = \mathcal{F}_P.$
\end{proof}

\begin{theorem}
\label{thm:compatible}
Let matroids $\mathcal{M}_{SP}$ , $\mathcal{M}_{PQ}$ be $\{S,P\}$--complete , $\{P,Q\}$--complete, respectively.
\begin{enumerate}
    \item Let $\mathcal{E}_P(\mathcal{M}^*_{SP}) = \mathcal{E}_P(\mathcal{M}_{PQ}).$ Then the matroid $\mathcal{M}_{SQ} :\equiv \mathcal{M}_{SP} \leftrightarrow_{\mathcal{M}} \mathcal{M}_{PQ}$ is $\{S,Q\}$--complete with $\mathcal{E}_S(\mathcal{M}_{SP}) = \mathcal{E}_S(\mathcal{M}_{SQ})$ and $\mathcal{E}_Q(\mathcal{M}_{PQ}) = \mathcal{E}_Q(\mathcal{M}_{SQ}).$
    
    \item Let every block of $\mathcal{E}_P(\mathcal{M}_{PQ})$ be a union of blocks of $\mathcal{E}_P(\mathcal{M}^*_{SP}).$ Then $\mathcal{M}_{PQ} = \mathcal{M}^*_{SP} \leftrightarrow_{\mathcal{M}} \mathcal{M}_{SQ}.$
    
    \item If $Q = \emptyset,$ and the collection of bases of $\mathcal{M}_P$ is a union of blocks of $\mathcal{E}_P(\mathcal{M}^*_{SP}),$ then $\mathcal{M}_{P} = \mathcal{M}^*_{SP} \leftrightarrow_{\mathcal{M}} (\mathcal{M}_{SP} \leftrightarrow_{\mathcal{M}} \mathcal{M}_{P}).$
    
    \item Let $\mathcal{E}_P(\mathcal{M}^*_{SP}) = \mathcal{E}_P(\mathcal{M}_{PQ})$ and let $\mathcal{M}_{SP}, \, \mathcal{M}_{PQ}$ be compatible over $P.$ Then $\mathcal{M}_{SP}$ and $\mathcal{M}_{PQ}$ are strongly compatible over $P.$
\end{enumerate}
\end{theorem}
\begin{proof}
We will use Lemma \ref{lem:blockunion} repeatedly.
\begin{enumerate}
  \item     Let $b_S \uplus b_Q, \, b_S \uplus b'_Q, \, b'_S \uplus b_Q$ be bases of $\mathcal{M}_{SQ} :\equiv \mathcal{M}_{SP} \leftrightarrow_{\mathcal{M}} \mathcal{M}_{PQ}.$ We will show that $b'_S \uplus b'_Q$ is a base of $\mathcal{M}_{SQ}.$
    
    There must exist bases $b_S \uplus b_P, \, b_S \uplus b'_P$ of $\mathcal{M}_{SP}$ and bases $(P - b_P) \uplus b_Q, \, (P - b'_P) \uplus b'_Q$ of $\mathcal{M}_{PQ}.$ Now $b_P, b'_P$ must belong to the same block of $\mathcal{E}_P(\mathcal{M}_{SP}).$ Since $\mathcal{E}_P(\mathcal{M}_{PQ}) = \mathcal{E}_P(\mathcal{M}^*_{SP}),$ it follows that $(P - b_P)$ and $(P - b'_P)$ belong to the same block of $\mathcal{E}_P(\mathcal{M}_{PQ}),$ and therefore $b_Q, b'_Q$ belong to the same block of $\mathcal{E}_Q(\mathcal{M}_{PQ}).$
    
    Since $b'_S \uplus b_Q$ is a base of $\mathcal{M}_{SQ},$ there must exist bases $b'_S \uplus b''_P$ and $(P - b''_P) \uplus b_Q$ of $\mathcal{M}_{SP}$ and \/ $\mathcal{M}_{PQ},$ respectively. Since $(P - b'_P) \uplus b_Q$ and $(P - b''_P) \uplus b_Q$ are bases of $\mathcal{M}_{PQ},$ it follows that $(P - b'_P)$ and $(P - b''_P)$ belong to the same block of $\mathcal{E}_P(\mathcal{M}_{PQ}).$
    
    Since $\mathcal{E}_P(\mathcal{M}^*_{SP}) = \mathcal{E}_P(\mathcal{M}_{PQ}),$ we have $\mathcal{E}_P(\mathcal{M}_{SP}) = \mathcal{E}_P(\mathcal{M}^*_{PQ}).$ Therefore, we must have $b'_P, b''_P$ belonging to the same block of $\mathcal{E}_P(\mathcal{M}_{SP}).$ Since $b'_S \uplus b''_P$ is a base of $\mathcal{M}_{SP},$ it follows that $b'_S \uplus b'_P$ is a base of $\mathcal{M}_{SP}.$ Thus we have $b'_S \uplus b'_P$ and $(P - b'_P) \uplus b'_Q$ as bases of $\mathcal{M}_{SP}$ and \/ $\mathcal{M}_{PQ},$ respectively. Hence $b'_S \uplus b'_Q$ is a base of $\mathcal{M}_{SQ},$ as required.
    
    Next we note that $b_S \uplus b_Q$ is a base of $\mathcal{M}_{SQ}$ iff there exist bases $b_S \uplus b_P$ and $(P - b_P) \uplus b_Q$ of $\mathcal{M}_{SP}$ and \/ $\mathcal{M}_{PQ},$ respectively. This happens iff $b_S, b_P$ belong to corresponding blocks of $\mathcal{E}_S(\mathcal{M}_{SP})$ and \/ $\mathcal{E}_P(\mathcal{M}_{SP}),$ respectively, and $(P - b_P), b_Q$ belong to corresponding blocks of $\mathcal{E}_P(\mathcal{M}_{PQ})$ and \/ $\mathcal{E}_Q(\mathcal{M}_{PQ}),$ respectively. Now $\mathcal{E}_P(\mathcal{M}_{SP}) = \mathcal{E}_P(\mathcal{M}^*_{PQ}).$ Therefore $\mathcal{E}_S(\mathcal{M}_{SP}) = \mathcal{E}_S(\mathcal{M}_{SQ})$ and $\mathcal{E}_Q(\mathcal{M}_{PQ}) = \mathcal{E}_Q(\mathcal{M}_{SQ}).$

    \item   Let $b_P \uplus b_Q$ be a base of $\mathcal{M}_{PQ}.$ Since every block of $\mathcal{E}_P(\mathcal{M}_{PQ})$ is a union of blocks of $\mathcal{E}_P(\mathcal{M}^*_{SP}),$ it follows that there exists a base $b_S \uplus b_P$ of $\mathcal{M}^*_{SP}.$ Since $b_P \uplus b_Q$ and $(S - b_S) \uplus (P - b_P)$ are bases of $\mathcal{M}_{PQ}$ and \/ $\mathcal{M}_{SP},$ respectively, it follows that $(S - b_S) \uplus b_Q$ is a base of $\mathcal{M}_{SQ} = \mathcal{M}_{SP} \leftrightarrow_{\mathcal{M}} \mathcal{M}_{PQ}.$
    
    Now $(S - b_S) \uplus b_Q$ and $b_S \uplus b_P$ are bases of $\mathcal{M}_{SQ}$ and \/ $\mathcal{M}^*_{SP},$ respectively, so that $b_P \uplus b_Q$ is a base of $ \mathcal{M}^*_{SP} \leftrightarrow_{\mathcal{M}} \mathcal{M}_{SQ}.$

Next let $b_P\uplus b_Q$ be a base of $\mathcal{M}^*_{SP} \leftrightarrow_{\mathcal{M}} \mathcal{M}_{SQ}.$
We saw earlier that there are disjoint bases of $\mbox{$\cal M$}_{SQ}, \mbox{$\cal M$}^*_{SP},$
which cover $S$
so that every base of $\mbox{$\cal M$}_{SQ}\vee \mbox{$\cal M$}^*_{SP}$  is a disjoint union of bases of $\mbox{$\cal M$}_{SQ}, \mbox{$\cal M$}^*_{SP}$ and $(\mbox{$\cal M$}_{SQ}\vee \mbox{$\cal M$}^*_{SP})\circ S=\mbox{$\cal F$}_S.$
Therefore there exist bases $b_S\uplus b_Q,(S-b_S)\uplus b_P$
 of $\mbox{$\cal M$}_{SQ}, \mbox{$\cal M$}^*_{SP},$
respectively. Now $\mbox{$\cal M$}_{SQ}= \mbox{$\cal M$}_{SP}\leftrightarrow_{\mathcal M} \mbox{$\cal M$}_{PQ}$ and by the same argument as  above, it follows that every base of $\mbox{$\cal M$}_{SP}\vee \mbox{$\cal M$}_{PQ}$ is a  disjoint union of bases of $\mbox{$\cal M$}_{PQ}, \mbox{$\cal M$}_{SP}$
 and $(\mbox{$\cal M$}_{SP}\vee \mbox{$\cal M$}_{PQ})\circ P=\mbox{$\cal F$}_P.$
 Therefore there exist bases  $b_S\uplus (P-b'_P),b'_P\uplus b_Q$
 of $\mbox{$\cal M$}_{SP}, \mbox{$\cal M$}_{PQ},$ respectively and a base $(S-b_S)\uplus b'_P$ of
 $\mbox{$\cal M$}^*_{SP}.$
Now $(S-b_S)\uplus b_P, (S-b_S)\uplus b'_P$ are bases of $\mbox{$\cal M$}^*_{SP}.$
It follows that $b_P, b'_P,$ belong to the same block of $\mbox{$\cal E$}_P(\mbox{$\cal M$}^*_{SP}).$
Since every block of $\mbox{$\cal E$}_P(\mbox{$\cal M$}_{PQ})$ is a union of blocks of $\mbox{$\cal E$}_P(\mbox{$\cal M$}^*_{SP}),$
we must have that
$b_P, b'_P,$ belong to the same block of $\mbox{$\cal E$}_P(\mbox{$\cal M$}_{PQ}).$
Since $b'_P\uplus b_Q$ is a base of $\mbox{$\cal M$}_{PQ},$ it follows that $b_P\uplus b_Q$
 is a base of $\mbox{$\cal M$}_{PQ}.$

Thus $\mbox{$\cal M$}_{PQ}= \mbox{$\cal M$}^*_{SP}\leftrightarrow_{\mathcal M} \mbox{$\cal M$}_{SQ}.$

\item 
The result follows trivially from part 2 above.

\item  Let $b_P^3$ be a base of ${\cal M}_{SP}\times P.$ Then there exists a  base $\hat{b}_P$ of
 ${\cal M}_{SP}\circ P,$ that contains $b_P^3. $ Now $b^4_P:\equiv  P-\hat{b}_P$ is a base of
 $({\cal M}_{SP}\circ P)^*= {\cal M}_{SP}^*\times P= {\cal M}_{PQ}\times P$ (using compatibility over $P$ of  ${\cal M}_{SP},{\cal M}_{PQ}$).
 Thus there exist disjoint bases $b^3_P,b^4_P$ of ${\cal M}_{SP}\times P, {\cal M}_{PQ}\times P,$
 respectively  and bases $P-b^4_P, P-b^3_P$  of ${\cal M}_{SP}\circ P,{\cal M}_{PQ}\circ  P,$  respectively such that $P-b^4_P\supseteq b^3_P, P-b^3_P\supseteq b^4_P.$ 

Next let $b^3_P,b^4_P$ be disjoint bases of ${\cal M}_{SP}\times P, {\cal M}_{PQ}\times P,$
 respectively and let
$b_S\uplus b'_P$ be a base of ${\cal M}_{SP}.$ 
 Then $b_S\uplus b^3_P$ is independent in ${\cal M}_{SP}$ 
 and since $P-b^4_P$ is a  base of ${\cal M}_{SP}\circ P$  , $b_S\uplus (P-b^4_P)$ is a spanning set of ${\cal M}_{SP}.$
 Therefore $b_S\uplus b^3_P$ can be extended to a base $b_S\uplus b_P$ of ${\cal M}_{SP},$ using only the elements of $(P-b^4_P),$ i.e., $b_S\uplus b^3_P\subseteq  b_S\uplus b_P\subseteq b_S\uplus (P-b^4_P)$ and $b^4_P\subseteq (P-b_P)\subseteq (P-b^3_P).$

 Let $b_P$ belong to a block $I^k_P$ of ${\cal M}_{SP}.$
 Since $\mbox{$\cal E$}_P({\cal M}^*_{SP})= \mbox{$\cal E$}_P({\cal M}_{PQ}),$ 
$(P-b_P)$ belongs to a block $\tilde{I}^k_P$ of $\mbox{$\cal E$}_P({\cal M}_{PQ}).$

Let $b_S\uplus b_Q$ be a base of ${\cal M}_{SQ}:\equiv  {\cal M}_{SP}\leftrightarrow_{\mathcal M} {\cal M}_{PQ}.$
 Now $(P-b_P)\uplus b'_Q$ is a base of ${\cal M}_{PQ}$ for some $b'_Q.$
 Therefore $(b_S\uplus b_P)\uplus (b'_Q\uplus (P-b_P))$
 is a base of ${\cal M}_{SP}\vee {\cal M}_{PQ}$ and $b_S\uplus b'_Q$ is a base of
 ${\cal M}_{SP}\leftrightarrow_{\mathcal M} {\cal M}_{PQ}:\equiv  ({\cal M}_{SP}\vee {\cal M}_{PQ})\times (S\uplus Q).$
It follows that $b'_Q, b_Q$ belong to the same block of $\mbox{$\cal E$}_Q({\cal M}_{SQ}).$
 Now $\mbox{$\cal E$}_Q({\cal M}_{SQ})=\mbox{$\cal E$}_Q({\cal M}_{PQ})$ by part 1 above. Therefore $b'_Q, b_Q$ belong to the same block of $\mbox{$\cal E$}_Q({\cal M}_{PQ}).$
 Therefore $(P-b_P)\uplus b_Q$ is base of ${\cal M}_{PQ}.$
 Thus we have disjoint bases $b_S\uplus b_P, (P-b_P)\uplus b_Q$ of
${\cal M}_{SP},{\cal M}_{PQ},$ respectively such that $b_P\supseteq b^3_P, (P-b_P)\supseteq b_P^4,$ i.e., ${\cal M}_{SP},{\cal M}_{PQ}$ are strongly compatible over $P.$
\end{enumerate}
\end{proof}
\begin{remark}
We note that from part 1 of the above theorem, it follows that $\mathcal{M}^*_{SQ} \leftrightarrow_{\mathcal{M}} (\mathcal{M}_{SQ})_{SQ'}$ is $\{Q,Q'\}$--complete if $\mathcal{M}_{SQ}$ is $\{S,Q\}$--complete (part 4 of Theorem \ref{thm:pseudoid}).
\end{remark}

We are now ready to show that $\{S,Q\}$--complete matroids are minimally decomposable into matroids which are themselves complete with respect to suitable partitions of underlying sets.

\begin{theorem}
\label{thm:sqcloseddecomposition}
Let $\mathcal{M}_{SQ}$ be $\{S,Q\}$--complete. Let $P :\equiv Q'.$ We have the following.
\begin{enumerate}
    \item $\mathcal{M}_{SQ} = (\mathcal{M}_{SQ})_{SQ'} \leftrightarrow_{\mathcal{M}} ((\mathcal{M}_{SQ})_{S'Q} \leftrightarrow_{\mathcal{M}} (\mathcal{M}^*_{SQ})_{S'Q'}).$ Therefore we have $\mathcal{M}_{SQ} = \mathcal{M}_{SP} \leftrightarrow_{\mathcal{M}} \mathcal{M}_{PQ},$ where $\mathcal{M}_{SP} :\equiv (\mathcal{M}_{SQ})_{SQ'}$ and $\mathcal{M}_{PQ} :\equiv ((\mathcal{M}_{SQ})_{S'Q} \leftrightarrow_{\mathcal{M}} (\mathcal{M}^*_{SQ})_{S'Q'}).$ 
    
    \item $\mathcal{M}_{SP} \circ P = \mathcal{M}_{PQ}^* \circ P$ and $\mathcal{M}_{SP} \times P = \mathcal{M}_{PQ}^* \times P.$ Further, $\mathcal{E}_S(\mathcal{M}_{SQ}) = \mathcal{E}_S(\mathcal{M}_{SP}), \quad \mathcal{E}_Q(\mathcal{M}_{SQ}) = \mathcal{E}_Q(\mathcal{M}_{PQ}),$ and $\mathcal{E}_P(\mathcal{M}^*_{SP}) = \mathcal{E}_P(\mathcal{M}_{PQ}).$ Therefore, $\mathcal{M}_{SP}$ and $\mathcal{M}_{PQ}$ are strongly compatible over $P.$

    \item Let $\hat{P}$ be defined starting from $P$ as in Theorem \ref{thm:matroidminP}, so that we have the minimal decomposition $\mathcal{M}_{SQ} = \mathcal{M}_{S\hat{P}} \leftrightarrow_{\mathcal{M}} \mathcal{M}_{\hat{P}Q},$ with 
$|\hat{P}| = \lambda_{SQ}(\mathcal{M}_{SQ}).$
Then $\mathcal{M}_{S\hat{P}}$ and $\mathcal{M}_{\hat{P}Q}$ are respectively $\{S,\hat{P}\}$--complete and $\{\hat{P},Q\}$--complete. Further, $\mathcal{E}_S(\mathcal{M}_{SQ}) = \mathcal{E}_S(\mathcal{M}_{S\hat{P}}), \quad \mathcal{E}_Q(\mathcal{M}_{SQ}) = \mathcal{E}_Q(\mathcal{M}_{\hat{P}Q}),$ and $\mathcal{E}_{\hat P}(\mathcal{M}^*_{S\hat{P}}) = \mathcal{E}_{\hat P}(\mathcal{M}_{\hat{P}Q}).$
\end{enumerate}
\end{theorem}
\begin{proof}
\begin{enumerate}
    \item By part 1 of Theorem \ref{thm:pseudoid}, $(\mathcal{M}_{SQ})_{SQ'} \leftrightarrow_{\mathcal{M}} ((\mathcal{M}_{SQ})_{S'Q} \leftrightarrow_{\mathcal{M}} (\mathcal{M}^*_{SQ})_{S'Q'})$ is the $\{S,Q\}$--augmentation of $\mathcal{M}_{SQ}.$ Therefore, if $\mathcal{M}_{SQ}$ is already $\{S,Q\}$--complete, we must have $\mathcal{M}_{SQ} = (\mathcal{M}_{SQ})_{SQ'} \leftrightarrow_{\mathcal{M}} ((\mathcal{M}_{SQ})_{S'Q} \leftrightarrow_{\mathcal{M}} (\mathcal{M}^*_{SQ})_{S'Q'}).$

    \item We have $\mathcal{M}_{SP} \circ P = (\mathcal{M}_{SQ})_{SQ'} \circ Q'.$ Furthermore, we find:
\begin{align*}
        \mathcal{M}_{PQ}^* \circ P &= ((\mathcal{M}_{SQ})_{S'Q} \leftrightarrow_{\mathcal{M}} (\mathcal{M}^*_{SQ})_{S'Q'})^* \circ Q' \\
        &= ((\mathcal{M}^*_{SQ})_{S'Q} \leftrightarrow_{\mathcal{M}} (\mathcal{M}_{SQ})_{S'Q'}) \circ Q' \quad \text{(by Theorem \ref{thm:idtmatroid})} \\
        &= (\mathcal{M}_{SQ})_{S'Q'} \circ Q' \quad \text{(by Lemma \ref{cor:minorspecial})}\\
&= (\mathcal{M}_{SQ})_{SQ'} \circ Q'.
    \end{align*}
    Therefore $\mathcal{M}_{SP} \circ P = \mathcal{M}_{PQ}^* \circ P.$ The proof that $\mathcal{M}_{SP} \times P = \mathcal{M}_{PQ}^* \times P$ is similar, using the same results.
 Thus $\mathcal{M}_{SP},\mathcal{M}_{PQ}$ are compatible over $P.$

    Next, by part 1 of Theorem \ref{thm:compatible}, $(\mathcal{M}_{QQ'})_{Q'Q} :\equiv [(\mathcal{M}^*_{SQ} \leftrightarrow_{\mathcal{M}} (\mathcal{M}_{SQ})_{SQ'})]_{Q'Q} = (\mathcal{M}_{SQ})_{S'Q} \leftrightarrow_{\mathcal{M}} (\mathcal{M}^*_{SQ})_{S'Q'}$ is $\{Q,Q'\}$--complete with $\mathcal{E}_Q((\mathcal{M}_{QQ'})_{Q'Q}) = \mathcal{E}_Q((\mathcal{M}_{SQ})_{S'Q})$ and $\mathcal{E}_{Q'}((\mathcal{M}_{QQ'})_{Q'Q}) = \mathcal{E}_{Q'}((\mathcal{M}^*_{SQ})_{S'Q'}).$
 
    By part 1 above, $\mathcal{M}_{SQ} = \mathcal{M}_{SP} \leftrightarrow_{\mathcal{M}} \mathcal{M}_{PQ} = (\mathcal{M}_{SQ})_{SQ'} \leftrightarrow_{\mathcal{M}} [(\mathcal{M}_{SQ})_{S'Q} \leftrightarrow_{\mathcal{M}} (\mathcal{M}^*_{SQ})_{S'Q'}].$ Thus:
    \begin{align*}
        \mathcal{E}_S(\mathcal{M}_{SP}) &= \mathcal{E}_S((\mathcal{M}_{SQ})_{SQ'}) = \mathcal{E}_S(\mathcal{M}_{SQ}), \\
        \mathcal{E}_Q(\mathcal{M}_{SQ}) &= \mathcal{E}_Q((\mathcal{M}_{SQ})_{S'Q}) = \mathcal{E}_Q((\mathcal{M}_{QQ'})_{Q'Q}) = \mathcal{E}_Q(\mathcal{M}_{PQ}), \\
        \mathcal{E}_P(\mathcal{M}^*_{SP}) &= \mathcal{E}_{Q'}((\mathcal{M}^*_{SQ})_{S'Q'}) = \mathcal{E}_{Q'}((\mathcal{M}_{QQ'})_{Q'Q}) = \mathcal{E}_P(\mathcal{M}_{PQ}).
    \end{align*}
    The strong compatibility of $\mathcal{M}_{SP}$ and $\mathcal{M}_{PQ}$ over $P$ now follows by part 4 of Theorem \ref{thm:compatible}.

    \item We have $\mathcal{M}_{SQ} = \mathcal{M}_{S\hat{P}} \leftrightarrow_{\mathcal{M}} \mathcal{M}_{\hat{P}Q},$ where $\hat{P} :\equiv P - \tilde{b}^3_P - \tilde{b}^4_P,$ and the components are defined as $\mathcal{M}_{S\hat{P}} :\equiv \mathcal{M}_{SP} \times (S \uplus (P - \tilde{b}^3_P)) \circ (S \uplus \hat{P})$ and $\mathcal{M}_{\hat{P}Q} :\equiv \mathcal{M}_{PQ} \times (Q \uplus (P - \tilde{b}^4_P)) \circ (Q \uplus \hat{P}).$

    Using part 4 of Theorem \ref{lem:sqrestrictcontract}, we see that $\mathcal{M}_{S\hat{P}}$ and $\mathcal{M}_{\hat{P}Q}$ are respectively $\{S,\hat{P}\}$--complete and $\{\hat{P},Q\}$--complete. From part 2 above,  $\mathcal{E}_S(\mathcal{M}_{SQ}) = \mathcal{E}_S(\mathcal{M}_{SP})$
%$ = \mathcal{E}_S(\mathcal{M}_{S\hat{P}})$ and $\mathcal{E}_Q(\mathcal{M}_{SQ}) = 
 and $\mathcal{E}_Q(\mathcal{M}_{PQ}) = \mathcal{E}_Q(\mathcal{M}_{SQ}).$ 
 From  part 3 of Theorem \ref{lem:sqrestrictcontract}, $ \mathcal{E}_S(\mathcal{M}_{SP})= \mathcal{E}_S(\mathcal{M}_{S\hat{P}})$ and $ \mathcal{E}_Q(\mathcal{M}_{PQ})= \mathcal{E}_Q(\mathcal{M}_{\hat{P}Q}).$ 
%We will now show that $\mathcal{E}_{\hat{P}}(\mathcal{M}^*_{S\hat{P}}) = \mathcal{E}_{\hat{P}}(\mathcal{M}_{\hat{P}Q}).$

Let ${\mathbf{I}}^j_{{P}}$ and $({\mathbf{I}}^*)^j_{{P}}$ be corresponding blocks of $\mathcal{E}_{{P}}(\mathcal{M}_{S{P}})$ and $\mathcal{E}_{{P}}(\mathcal{M}^*_{S{P}})$ such that $b_{{P}}$ and ${P} - b_{{P}}$ belong respectively to ${\mathbf{I}}^j_{{P}}$ and $({\mathbf{I}}^*)^j_{{P}}.$ Since 
 $\mathcal{E}_P(\mathcal{M}^*_{SP}) = \mathcal{E}_P(\mathcal{M}_{PQ}),$
 $({\mathbf{I}}^*)^j_{{P}}$ is a block of $\mathcal{E}_P(\mathcal{M}_{PQ}).$
 Therefore there exist bases $b_S\uplus b_P, b_Q\uplus ({P} - b_{{P}})$ of $\mathcal{M}_{S{P}}, \mathcal{M}_{{P}Q},$ respectively.
 By strong compatibility over $P$ of $\mathcal{M}_{S{P}}, \mathcal{M}_{{P}Q},$ 
 we may take $\tilde{b}^3_P, \tilde{b}^4_P$ to be subsets of $ b_P,  ({P} - b_{{P}}),$ respectively.
 Therefore $b_S\uplus (b_P- \tilde{b}^3_P), b_Q\uplus ({P} - b_{{P}})- \tilde{b}^4_P$ 
 are bases of  $\mathcal{M}_{S\hat {P}}, \mathcal{M}_{\hat{P}Q},$ respectively.
  We note that $\hat{P}-(b_P- \tilde{b}^3_P)= P-(\tilde{b}^3_P\uplus \tilde{b}^4_P)- (b_P- \tilde{b}^3_P)= [({P} - b_{{P}})- \tilde{b}^4_P]. $
 Thus there is a unique pair of blocks  ${\mathbf{I}}^j_{\hat{P}}$ and $({\mathbf{I}}^*)^j_{\hat{P}}$ of $\mathcal{M}_{S\hat {P}}, \mathcal{M}_{\hat{P}Q},$ 
 corresponding to ${\mathbf{I}}^j_{{P}}$ and $({\mathbf{I}}^*)^j_{{P}}$
 such that if  $b_{{P}},({P} - b_{{P}})$ belong respectively to ${\mathbf{I}}^j_{{P}},({\mathbf{I}}^*)^j_{{P}},$ with  $b_{{P}}\supseteq \tilde{b}^3_P$ and $({P} - b_{{P}})\supseteq \tilde{b}^4_P,$ then the subsets $b_{{P}}- \tilde{b}^3_P,$  $({P} - b_{{P}})- \tilde{b}^4_P$ belong respectively to  ${\mathbf{I}}^j_{\hat{P}},$  $({\mathbf{I}}^*)^j_{\hat{P}}.$   

On the other hand we know that $\mathcal{E}_S(\mathcal{M}_{SP})= \mathcal{E}_S(\mathcal{M}_{S\hat{P}})$ and $\mathcal{E}_Q(\mathcal{M}_{PQ})= \mathcal{E}_Q(\mathcal{M}_{\hat{P}Q}).$ 
Therefore, starting from $b_S\uplus b_{\hat{P}}, b_Q\uplus (\hat{P} - b_{\hat{P}})$ of $\mathcal{M}_{S\hat{P}}, \mathcal{M}_{\hat{P}Q},$ respectively we can get bases $b_S\uplus (b_{\hat{P}}\uplus \tilde{b}^3_P), b_Q\uplus ((\hat{P} - b_{\hat{P}})\uplus \tilde{b}^4_P)$ of $\mathcal{M}_{S{P}}, \mathcal{M}_{{P}Q},$ respectively. 
 Therefore, there is a unique pair of blocks  ${\mathbf{I}}^j_{{P}}$ and $({\mathbf{I}}^*)^j_{{P}}$ of $\mathcal{M}_{S {P}}, \mathcal{M}_{{P}Q},$  
 corresponding to ${\mathbf{I}}^j_{\hat{P}}$ and $({\mathbf{I}}^*)^j_{\hat{P}}$
 such that whenever $b_{{P}}- \tilde{b}^3_P$ and $({P} - b_{{P}})- \tilde{b}^4_P$ belong respectively to ${\mathbf{I}}^j_{\hat{P}}$ and $({\mathbf{I}}^*)^j_{\hat{P}},$ the subsets $b_P=(b_{{P}}- \tilde{b}^3_P)\uplus \tilde{b}^3_P,$  $P-b_P=(({P} - b_{{P}})- \tilde{b}^4_P) \uplus \tilde{b}^4_P$ belong respectively to  ${\mathbf{I}}^j_{{P}},$  $({\mathbf{I}}^*)^j_{{P}}.$  

Since $\mathcal{E}_P(\mathcal{M}^*_{SP}) = \mathcal{E}_P(\mathcal{M}_{PQ}),$
 we conclude that $\mathcal{E}_{\hat{P}}(\mathcal{M}^*_{S\hat{P}}) = \mathcal{E}_{\hat{P}}(\mathcal{M}_{\hat{P}Q}).$

\end{enumerate}
\end{proof}
\begin{figure}[htbp]
\begin{center}
 \includegraphics[width=4.5in]{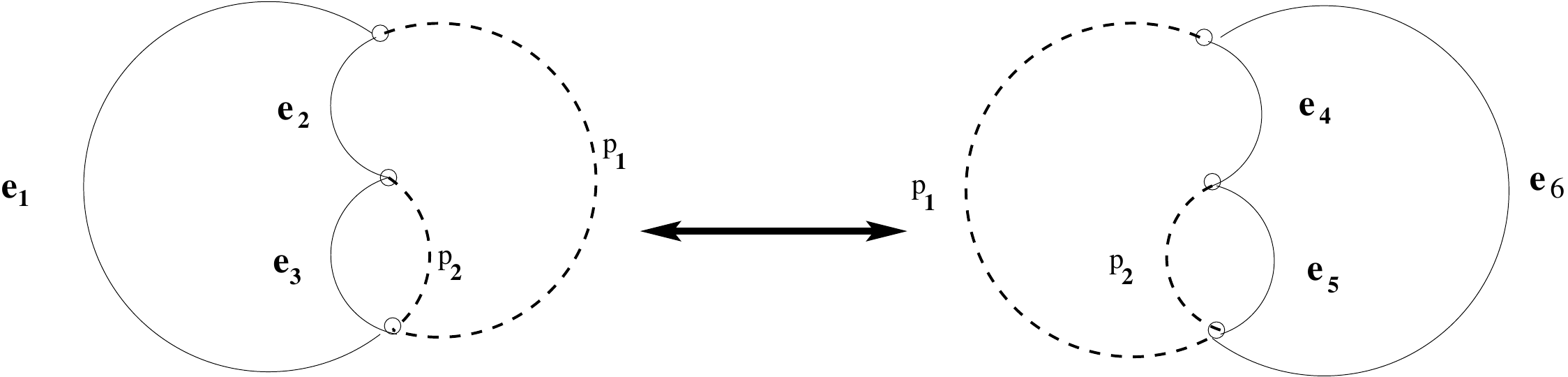}
 \caption{A minimally decomposable matroid with $S=\{e_1,e_2,e_3\}, Q=\{e_4,e_5,e_6\}$
that is not $\{S,Q\}-$complete
}
\label{fig:matroid2}
\end{center}
\end{figure}
\begin{remark}
We note that being $\{S,Q\}$-complete is not a necessary condition for a matroid $\mathcal{M}_{SQ}$ to be decomposable as $\mathcal{M}_{SP} \leftrightarrow_{\mathcal{M}} \mathcal{M}_{PQ}$ with $|P| = r(\mathcal{M}_{SQ} \circ S) - r(\mathcal{M}_{SQ} \times S).$

Let $\mathcal{M}_{SP}$ and $\mathcal{M}_{PQ}$ be the two graphic matroids shown in Figure \ref{fig:matroid2}, and let $\mathcal{M}_{SQ} \equiv \mathcal{M}_{SP} \leftrightarrow_{\mathcal{M}} \mathcal{M}_{PQ}.$ We will show that this is a minimal decomposition for $\mathcal{M}_{SQ}$, but the matroid itself is not $\{S,Q\}$-complete.

First, we observe that $\{e_1, e_2\}$ and $\{p_1, p_2\}$ are bases of $\mathcal{M}_{SP}$ and $\mathcal{M}_{PQ}$ respectively, so that $\{e_1, e_2, p_1, p_2\}$ is a base of $(\mathcal{M}_{SP} \vee \mathcal{M}_{PQ}),$ and $\{e_1, e_2\}$ is a base of:
\[
\mathcal{M}_{SQ} \equiv \mathcal{M}_{SP} \leftrightarrow_{\mathcal{M}} \mathcal{M}_{PQ} = (\mathcal{M}_{SP} \vee \mathcal{M}_{PQ}) \times (S \uplus Q).
\]
Therefore, $r(\mathcal{M}_{SQ} \circ S) = 2.$ Similarly, it can be seen that $\{e_4, e_5\}$ is a base of $\mathcal{M}_{SQ}$, so that $r(\mathcal{M}_{SQ} \times S) = 0.$ It follows that:
\[
r(\mathcal{M}_{SQ} \circ S) - r(\mathcal{M}_{SQ} \times S) = 2,
\]
which implies that $ \mathcal{M}_{SP} \leftrightarrow_{\mathcal{M}} \mathcal{M}_{PQ}$ is a minimal decomposition of $\mathcal{M}_{SQ}.$

Next, we have bases $\{e_2, p_2\}$ and $\{p_1, e_4\}$ for $\mathcal{M}_{SP}$ and $\mathcal{M}_{PQ}$ respectively, and therefore $\{e_2, p_2, p_1, e_4\}$ is a base of $\mathcal{M}_{SP} \vee \mathcal{M}_{PQ}.$ It follows that $\{e_2, e_4\}$ is a base of:
\[
\mathcal{M}_{SQ} \equiv \mathcal{M}_{SP} \leftrightarrow_{\mathcal{M}} \mathcal{M}_{PQ} = (\mathcal{M}_{SP} \vee \mathcal{M}_{PQ}) \times (S \uplus Q).
\]
It can be similarly verified that $\{e_2, e_5\}$ and $\{e_3, e_4\}$ are bases of $\mathcal{M}_{SQ}.$ 

However, $\{e_3, e_5\}$ cannot be a base of $\mathcal{M}_{SQ}$ because $\{e_3, p_2\}$ and $\{p_2, e_5\}$ are dependent in $\mathcal{M}_{SP}$ and $\mathcal{M}_{PQ}$ respectively. Consequently, we cannot have $\{e_3, p_i, p_j, e_5\}$ with $p_i \neq p_j$ as a base of $\mathcal{M}_{SP} \vee \mathcal{M}_{PQ}.$

Thus, $\mathcal{M}_{SQ}$ is not $\{S,Q\}$-complete, even though it is minimally decomposable.
\end{remark}
\section{Free products, principal sums and $\{S,Q\}$-completeness} 
\label{sec:free}
In this section we examine simple ways of building an $\{S,Q\}$-complete matroid $\mathcal{M}_{SQ},$ with specified restriction or contraction to the sets $S$ and $Q.$ Such matroids could be regarded as 'products' or 'sums' of these latter matroids, examples being the previously mentioned free products and principal sums of matroids. Our basic strategy is to use suitable equivalence relations of independent subsets of individual matroids as $\mathcal{E}_S(\mathcal{M}_{SQ}), \mathcal{E}_Q(\mathcal{M}_{SQ}).$ We address this question in the following lemma. We need a preliminary definition.

\begin{definition}
Let $\mathcal{M}_S$ be a matroid on $S.$ We say an equivalence relation $\mathcal{E}(\mathcal{M}_S)$ of specified independent sets of $\mathcal{M}_S$ is a \textbf{rank partition, $0 \leq r \leq r(\mathcal{M}_S),$} iff
  every block $\mathcal{E}^r(\mathcal{M}_S)$ contains all the independent sets of the same rank $r$  and only of rank $r.$
\end{definition}

\begin{lemma}
\label{lem:equivalencecondition}
Let $\mathcal{M}_S, \mathcal{M}_Q$ be matroids on $S, Q$ respectively with rank partitions $\mathcal{E}(\mathcal{M}_S), \mathcal{E}(\mathcal{M}_Q)$ respectively and let $0\leq k\leq r(\mathcal{M}_S)+r(\mathcal{M}_Q), k \in \mathbb{Z}.$ Let $\mathcal{B}$ be the collection of subsets of $S \uplus Q$ which are unions of independent sets of corresponding blocks of $\mathcal{E}^r(\mathcal{M}_S)$ and $\mathcal{E}^{k-r}(\mathcal{M}_Q),$ where:
\[
\max\{0, k - r(\mathcal{M}_Q )\} \leq r \leq \min\{r(\mathcal{M}_S ), k \}.
\]
We then have the following:
\begin{enumerate}
    \item $\mathcal{B}$ is the collection of bases of an $\{S,Q\}$-complete matroid $\mathcal{M}_{SQ}.$
    \item The bases of $\mathcal{M}_{SQ} \times S$ ($\mathcal{M}_{SQ} \times Q$) are the members of $\mathcal{E}^r(\mathcal{M}_S)$ ($\mathcal{E}^r(\mathcal{M}_Q)$), where $r = max \{0,k-r(\mathcal{M}_Q)\}$ ($r = max \{0,k-r(\mathcal{M}_S)\}$).
\item $r(\mathcal{M}_{SQ} \circ S) - r(\mathcal{M}_{SQ} \times S)= r(\mathcal{M}_{SQ} \circ Q) - r(\mathcal{M}_{SQ} \times Q)= min\{ r(\mathcal{M}_S), k\}- max \{0,k-r(\mathcal{M}_Q)\}.$
\end{enumerate}
\end{lemma}

\begin{proof}
\begin{enumerate}
\item  Let $b^1, b^2 \in \mathcal{B}$ and let $b^1 = b^1_{S} \uplus b^1_{Q}, \, b^2 = b^2_{S} \uplus b^2_{Q}.$ We know that $|b^1| = |b^2| = k.$ Let $e \in b^2 - b^1.$ We will show that there exists an element $e' \in b^1 - b^2$ such that $(e \cup (b^1 - e')) \in \mathcal{B}.$

Suppose $e \in S.$ If $e \cup b^1_{S}$ is dependent in $\mathcal{M}_S,$ then it contains a circuit of $\mathcal{M}_S$ which has an element $e'$ that is not in $b^2_{S}.$ Now $(e \cup (b^1_S - e')) \in \mathcal{E}^r(\mathcal{M}_S)$ and $(e \cup (b^1_S - e')) \uplus b^1_{Q} \in \mathcal{B}.$

If $e \cup b^1_{S}$ is independent in $\mathcal{M}_S,$ we consider two cases:
\begin{align*}
&\textbf{Case 1.} \quad |b^2_S| > |b^1_{S}|. \\
&\text{In this case, } |b^2_Q| < |b^1_{Q}|. \text{ Let } e' \in b^1_{Q} - b^2_Q. \text{ Then } e \cup b^1_{S} \uplus (b^1_{Q} - e') \in \mathcal{B}. \\
&\textbf{Case 2.} \quad |b^2_S| \leq |b^1_{S}|. \\
&\text{Since } e \in b^2_{S} - b^1_{S}, \text{ there is an element } e' \in b^1_{S} - b^2_{S}. \text{ Then } e \cup (b^1_{S} - e') \uplus b^1_{Q} \in \mathcal{B}.
\end{align*}
The case where $e \in Q$ is symmetric with the above.
\item  A subset $b_S$ ($b_Q$) is a base of $\mathcal{M}_{SQ} \times S$ ($\mathcal{M}_{SQ} \times Q$) iff  there exists a base $\hat{b}_Q$ of $\mathcal{M}_{SQ} \circ Q$ ($\hat{b}_S$ of $\mathcal{M}_{SQ} \circ S$) such that $b_S \uplus \hat{b}_Q$ ($b_Q \uplus \hat{b}_S$) is a base of $\mathcal{M}_{SQ}.$ The result now follows from the definition of $\mathcal{M}_{SQ}.$

\item 
We have $r(\mathcal{M}_{SQ}) = k = |b|$ for $b \in \mathcal{B}$ 
 and $r(\mathcal{M}_{SQ} \circ S) = min\{ r(\mathcal{M}_S), k\}.$ Since $r(\mathcal{M}_{SQ}) = r(\mathcal{M}_{SQ} \circ S) + r(\mathcal{M}_{SQ} \times Q) = r(\mathcal{M}_{SQ} \circ Q) + r(\mathcal{M}_{SQ} \times S),$ we have  
$ r(\mathcal{M}_{SQ} \circ S) - r(\mathcal{M}_{SQ} \times S) = r(\mathcal{M}_{SQ} \circ Q) - r(\mathcal{M}_{SQ} \times Q).$ Now  $r(\mathcal{M}_{SQ} \times S)=
 max\{0,k-r(\mathcal{M}_Q)\}$ so that $ r(\mathcal{M}_{SQ} \circ S) - r(\mathcal{M}_{SQ} \times S)= min\{ r(\mathcal{M}_S), k\}- max \{0,k-r(\mathcal{M}_Q)\}.$

\end{enumerate}
\end{proof}

\begin{remark}
\label{rem:product}
 When $k\geq max\{r(\mathcal{M}_{S}),r(\mathcal{M}_{Q})\}, $ Lemma \ref{lem:equivalencecondition} gives a method of building a `product' $\mathcal{M}_{SQ}$ of two given matroids $\mathcal{M}_S, \mathcal{M}_Q$ such that these matroids are restrictions of the product to the sets $S, Q,$ respectively. It is clear that to build a product whose contractions are specified matroids, we merely have to dualise the process, i.e., build a matroid $\mathcal{M}^*_{SQ}$ which has $\mathcal{M}^*_S, \mathcal{M}^*_Q$ as restrictions so that $\mathcal{M}_S, \mathcal{M}_Q$ are contractions of $\mathcal{M}_{SQ}.$ The idea of (minimal) composition of matroids gives us a method of building a product for which $\mathcal{M}_S, \mathcal{M}_Q$ are respectively restriction and contraction. This is described in the next theorem.
\end{remark}

We need a preliminary definition.

\begin{definition}
\label{def:RRCC}
Let $\mathcal{M}_S, \mathcal{M}_Q$ be matroids on $S, Q$ respectively with rank partitions $\mathcal{E}(\mathcal{M}_S), \mathcal{E}(\mathcal{M}_Q)$ respectively. 

We refer to the matroid $\mathcal{M}_{SQ}$ which has as bases the collection of subsets of $S \uplus Q$ which are unions of independent sets of corresponding blocks of $\mathcal{E}^r(\mathcal{M}_S)$ and $\mathcal{E}^{k_{\max}-r}(\mathcal{M}_Q),$ where $0 \leq k_{\max} \leq r(\mathcal{M}_S) + r(\mathcal{M}_Q)$ and $\max\{r(\mathcal{M}_S), r(\mathcal{M}_Q)\} \leq k_{\max},$ by $RR(\mathcal{M}_S, \mathcal{M}_Q)_{k_{\max}}.$ 

Similarly, we denote the matroid $(RR(\mathcal{M}^*_S, \mathcal{M}^*_Q)_{k_{\max}})^*$ by $CC(\mathcal{M}_S, \mathcal{M}_Q)_{k^*_{\max}},$ where $0 \leq k_{\max} \leq r(\mathcal{M}^*_S) + r(\mathcal{M}^*_Q), max\{r(\mathcal{M}^*_S),r(\mathcal{M}^*_Q)\} \leq k_{\max}, k^*_{\max} \equiv |S \uplus Q| - k_{\max}.$

When $k_{\max} = \max\{r(\mathcal{M}_S), r(\mathcal{M}_Q)\} = r(\mathcal{M}_S),$ we omit the subscript and write $RR(\mathcal{M}_S, \mathcal{M}_Q)_{r(\mathcal{M}_S)}$ as $RR(\mathcal{M}_S, \mathcal{M}_Q),$ and when $k_{\max} = \max\{r(\mathcal{M}^*_S), r(\mathcal{M}^*_Q)\} = r(\mathcal{M}^*_S),$ we write $(RR(\mathcal{M}_S^*, \mathcal{M}_Q^*)_{r(\mathcal{M}^*_S)})^*$ as $CC(\mathcal{M}_S, \mathcal{M}_Q).$
\end{definition}

\begin{theorem}
\label{thm:RRCC}
Let $\mathcal{M}_{S}, \mathcal{M}_{Q}$ be matroids respectively on disjoint sets $S, Q$ and let $P$ disjoint from $S\cup Q,$ with $0\leq |P|\leq min\{r(\mathcal{M}_S), r(\mathcal{M}^*_Q)\} .$  Let $\hat{\mathcal{M}}_{SP} = RR(\mathcal{M}_{S}, \mathcal{F}_P)$ and let $\hat{\mathcal{M}}_{PQ} = CC(\mathcal{M}_{Q}, \mathbf{0}_P).$ We have the following:
\begin{enumerate}
    \item The subset $P$ is contained in a base as well as in a cobase in the matroids $\hat{\mathcal{M}}_{SP}$ and $\hat{\mathcal{M}}_{PQ}$ so that $\hat{\mathcal{M}}^*_{SP}\circ P= \hat{\mathcal{M}}_{PQ}\circ P, \hat{\mathcal{M}}^*_{SP}\times P= \hat{\mathcal{M}}_{PQ}\times P.$ Further,
 $\mathcal{E}_P(\hat{\mathcal{M}}^*_{SP})= \mathcal{E}_P(\hat{\mathcal{M}}_{PQ}).$ 

    \item $\hat{\mathcal{M}}_{SP} \leftrightarrow_{\mathcal{M}} \mathcal{F}_P = \mathcal{M}_{S}$ and $\hat{\mathcal{M}}_{PQ} \leftrightarrow_{\mathcal{M}} \mathcal{F}_Q = \mathcal{F}_P.$
    \item $\hat{\mathcal{M}}_{SP} \leftrightarrow_{\mathcal{M}} \mathbf{0}_S = \mathbf{0}_P, \, \hat{\mathcal{M}}_{PQ} \leftrightarrow_{\mathcal{M}} \mathbf{0}_P = \mathcal{M}_{Q},$ and $\hat{\mathcal{M}}_{PQ} \leftrightarrow_{\mathcal{M}} \mathbf{0}_Q = \mathbf{0}_P.$
    \item $(\hat{\mathcal{M}}_{SP} \leftrightarrow_{\mathcal{M}} \hat{\mathcal{M}}_{PQ}) \circ S = \mathcal{M}_{S}$ and $(\hat{\mathcal{M}}_{SP} \leftrightarrow_{\mathcal{M}} \hat{\mathcal{M}}_{PQ}) \times Q = \mathcal{M}_{Q}.$
    \item Let $|P| =  \min\{r(\mathcal{M}_{S}), r(\mathcal{M}_{Q}^*)\}$ and let $\mathcal{B}$ be the collection of all subsets of $S \uplus Q$ of size $r(\mathcal{M}_{S}) + r(\mathcal{M}_{Q})$ which are unions of independent sets of $\mathcal{M}_{S}$ and spanning sets of $\mathcal{M}_{Q}.$ Then $\mathcal{B}$ is the collection of bases of $\hat{\mathcal{M}}_{SP} \leftrightarrow_{\mathcal{M}} \hat{\mathcal{M}}_{PQ},$ which is
 an $\{S,Q\}$-complete matroid.
    \item If $|P| = \min\{r(\mathcal{M}_{S}), r(\mathcal{M}_{Q}^*)\},$ then $\hat{\mathcal{M}}_{SP} \leftrightarrow_{\mathcal{M}} \hat{\mathcal{M}}_{PQ}$ is a minimal decomposition of the matroid whose bases are unions of independent sets of $\mathcal{M}_{S}$ and spanning sets of $\mathcal{M}_{Q}$ of size $r(\mathcal{M}_{S}) + r(\mathcal{M}_{Q}).$
\end{enumerate}
\end{theorem}
\begin{proof}
\begin{enumerate}
\item  The definitions of $\hat{\mathcal{M}}_{SP}$ and $\hat{\mathcal{M}}_{PQ}$  imply respectively that $r(\hat{\mathcal{M}}_{SP}) = r(\mathcal{M}_{S}), r(\hat{\mathcal{M}}^*_{PQ}) = r(\mathcal{M}^*_{Q})$ and that $|P|\leq r(\mathcal{M}_{S}),|P|\leq r(\mathcal{M}^*_{Q}).$ From this it follows  that $P$ is contained in a base as well as in a cobase in the matroids $\hat{\mathcal{M}}_{SP}$ and $\hat{\mathcal{M}}_{PQ}$ and that  
 $\mathcal{E}_P(\hat{\mathcal{M}}^*_{SP}), \mathcal{E}_P(\hat{\mathcal{M}}_{PQ})$ are just partitions of the set $2^P$ into subsets of the same cardinality
 and therefore the same.

\item  
Since  $\hat{\mathcal{M}}_{SP} = RR(\mathcal{M}_{S}, \mathcal{F}_P)$ with $k_{max}= r(\mathcal{M}_{S})\geq |P|,$ we have that independent sets of  $\hat{\mathcal{M}}_{SP}\circ S$ are the same as those of $\mathcal{M}_{S}.$ Since  $\hat{\mathcal{M}}^*_{PQ} = RR(\mathcal{M}^*_{Q}, \mathcal{F}_P)$ with $k_{max}= r(\mathcal{M}^*_{Q})\geq |P|,$ we have the null set as the base of   $\hat{\mathcal{M}}^*_{PQ}\times P,$ 
 so that $P$ is the base for $\hat{\mathcal{M}}_{PQ}\circ P.$ 
 We therefore have the following
\[
\hat{\mathcal{M}}_{SP} \leftrightarrow_{\mathcal{M}} \mathcal{F}_P = \hat{\mathcal{M}}_{SP} \circ S = \mathcal{M}_{S},
\quad \hat{\mathcal{M}}_{PQ} \leftrightarrow_{\mathcal{M}} \mathcal{F}_Q = \hat{\mathcal{M}}_{PQ}\circ P=
\mathcal{F}_P.
\]

\item  Since $\hat{\mathcal{M}}_{SP} = RR(\mathcal{M}_{S}, \mathcal{F}_P)$ with $k_{max}= r(\mathcal{M}_{S})\geq |P|,$ we have $\hat{\mathcal{M}}_{SP}\times P= \hat{\mathcal{M}}_{SP} \leftrightarrow_{\mathcal{M}} \mathbf{0}_S =\mathbf{0}_P.$ 
Since  $\hat{\mathcal{M}}^*_{PQ} = RR(\mathcal{M}^*_{Q}, \mathcal{F}_P)$ with $k_{max}= r(\mathcal{M}^*_{Q})\geq |P|,$ we have $\hat{\mathcal{M}}^*_{PQ}\circ P=\mathcal{F}_P, $ so that  $\hat{\mathcal{M}}_{PQ}\times P=
\hat{\mathcal{M}}_{PQ} \leftrightarrow_{\mathcal{M}} \mathbf{0}_Q = \mathbf{0}_P.$ 
We also have:

$
\hat{\mathcal{M}}_{PQ} \leftrightarrow_{\mathcal{M}} \mathbf{0}_P = (RR(\mathcal{M}_{Q}^*, \mathcal{F}_P))^* \times Q = (RR(\mathcal{M}_{Q}^*, \mathcal{F}_P) \circ Q)^* = (\mathcal{M}_Q^*)^* = \mathcal{M}_Q.
$

\item  We use the restricted associativity of linking (Lemma \ref{lem:restassoc}) to prove these results. For the first relation:
\begin{align*}
(\hat{\mathcal{M}}_{SP} \leftrightarrow_{\mathcal{M}} \hat{\mathcal{M}}_{PQ}) \circ S &= (\hat{\mathcal{M}}_{SP} \leftrightarrow_{\mathcal{M}} \hat{\mathcal{M}}_{PQ}) \leftrightarrow_{\mathcal{M}} \mathcal{F}_Q \\
&= \hat{\mathcal{M}}_{SP} \leftrightarrow_{\mathcal{M}} (\hat{\mathcal{M}}_{PQ} \leftrightarrow_{\mathcal{M}} \mathcal{F}_Q) \\
&= \hat{\mathcal{M}}_{SP} \leftrightarrow_{\mathcal{M}} \mathcal{F}_P = \hat{\mathcal{M}}_{SP} \circ S, \quad \text{(using part 2 above)}.
\end{align*}
Next, for the second relation:
\begin{align*}
(\hat{\mathcal{M}}_{SP} \leftrightarrow_{\mathcal{M}} \hat{\mathcal{M}}_{PQ}) \times Q &= (\hat{\mathcal{M}}_{SP} \leftrightarrow_{\mathcal{M}} \hat{\mathcal{M}}_{PQ}) \leftrightarrow_{\mathcal{M}} \mathbf{0}_S \\
&= (\hat{\mathcal{M}}_{SP} \leftrightarrow_{\mathcal{M}} \mathbf{0}_S) \leftrightarrow_{\mathcal{M}} \hat{\mathcal{M}}_{PQ} \\
&= \mathbf{0}_P \leftrightarrow_{\mathcal{M}} \hat{\mathcal{M}}_{PQ} = \mathcal{M}_{Q}, \quad \text{(using part 3 above)}.
\end{align*}

\item  By the fact that $\hat{\mathcal{M}}_{SP}, \hat{\mathcal{M}}_{PQ}$ are compatible over $P$ (part 1 above), it follows that $(\hat{\mathcal{M}}_{SP} \vee \hat{\mathcal{M}}_{PQ}) \circ P = \hat{\mathcal{M}}_{SP} \circ P \vee \hat{\mathcal{M}}_{PQ} \circ P = \mathcal{F}_P.$ Therefore, a subset $b_S \uplus b_Q$ is a base of $\hat{\mathcal{M}}_{SP} \leftrightarrow_{\mathcal{M}} \hat{\mathcal{M}}_{PQ}$ iff  there exists $b_P \subseteq P$ such that $b_S \uplus b_P$ and $(P - b_P) \uplus b_Q$ are bases of $\hat{\mathcal{M}}_{SP}$ and \(\hat{\mathcal{M}}_{PQ}\) respectively.

Now, $b_S \uplus b_P$ is a base of $\hat{\mathcal{M}}_{SP} = RR(\mathcal{M}_{S}, \mathcal{F}_P)$ iff and only if $b_S, b_P$ are independent in $\mathcal{M}_{S}, \mathcal{F}_P$ respectively, and:
\[
|b_S \uplus b_P| = r(\hat{\mathcal{M}}_{SP} \circ S) + r(\hat{\mathcal{M}}_{SP} \times P) = r(\mathcal{M}_{S}).
\]
We note that every subset of $P$ is automatically independent in $\mathcal{F}_P.$

Similarly, $b_Q \uplus (P - b_P)$ is a base of $\hat{\mathcal{M}}_{PQ} = CC(\mathcal{M}_{Q}, \mathbf{0}_P)$ iff  $(Q - b_Q) \uplus b_P$ is a base of $\hat{\mathcal{M}}^*_{PQ} = RR(\mathcal{M}^*_{Q}, \mathcal{F}_P).$ This occurs iff  $(Q - b_Q), b_P$ are independent in $\mathcal{M}^*_{Q}, \mathcal{F}_P$ respectively, and:
\[
|(Q - b_Q) \uplus b_P| = r(\mathcal{M}^*_{Q}) + r(\hat{\mathcal{M}}^*_{PQ} \times P) = r(\mathcal{M}^*_{Q}).
\]
Equivalently, this holds iff  $b_Q \uplus (P - b_P)$ contains spanning sets respectively of $\mathcal{M}_{Q}, \mathbf{0}_P$ and:
\[
|b_Q \uplus (P - b_P)| = r(\mathcal{M}_{Q}) + r(\hat{\mathcal{M}}_{PQ} \circ P) = r(\mathcal{M}_{Q}) + |P|.
\]
We note that every subset of $P$ is a spanning set of $\mathbf{0}_P.$

Thus, we find that $b_S \uplus b_Q$ is a base of $\hat{\mathcal{M}}_{SP} \leftrightarrow_{\mathcal{M}} \hat{\mathcal{M}}_{PQ}$ iff  the following conditions are satisfied:
\begin{itemize}
    \item $b_S$ is independent in $\mathcal{M}_{S}.$
    \item $b_Q$ is a spanning set of $\mathcal{M}_{Q}.$
    \item $|b_S \uplus b_Q| = r(\mathcal{M}_{S}) + r(\mathcal{M}_{Q}).$
    \item $0 \leq r(\mathcal{M}_{S}) - |b_S| = |b_Q| - r(\mathcal{M}_{Q}) = r(\mathcal{M}_{Q}^*) - |Q - b_Q| \leq |P|.$
\end{itemize}

If $|P| = \min\{r(\mathcal{M}_{S}), r(\mathcal{M}_{Q}^*)\},$ the final bounding condition is automatically satisfied. Therefore, $b_S \uplus b_Q$ is a base of $\hat{\mathcal{M}}_{SP} \leftrightarrow_{\mathcal{M}} \hat{\mathcal{M}}_{PQ}$ iff  $b_S$ is independent in $\mathcal{M}_{S},$  $b_Q$ is a spanning set of $\mathcal{M}_{Q}$ and $|b_S|+|b_Q|= r(\mathcal{M}_{S})+ r(\mathcal{M}_{Q}).$  This proves the matroid is $\{S,Q\}$-complete.

\item  Let $\mathcal{M}_{SQ}$ denote the matroid whose bases are unions of independent sets of $\mathcal{M}_{S}$ and spanning sets of $\mathcal{M}_{Q}$ of total size $r(\mathcal{M}_{S}) + r(\mathcal{M}_{Q}).$ If $|P| = \min\{r(\mathcal{M}_{S}), r(\mathcal{M}_{Q}^*)\},$ we will verify that:
\[
|P| = r(\mathcal{M}_{SQ} \circ S) - r(\mathcal{M}_{SQ} \times S) = r(\mathcal{M}_{SQ} \circ Q) - r(\mathcal{M}_{SQ} \times Q).
\]
By Theorem \ref{thm:lowerboundP}, it would then follow that $\hat{\mathcal{M}}_{SP} \leftrightarrow_{\mathcal{M}} \hat{\mathcal{M}}_{PQ}$ is a minimal decomposition of $\mathcal{M}_{SQ}.$

It is clear that $\mathcal{M}_{SQ} \circ S = \mathcal{M}_{S},$ so that $r(\mathcal{M}_{SQ} \circ S) = r(\mathcal{M}_{S}).$ To compute $r(\mathcal{M}_{SQ} \times S),$ we first construct a base of $\mathcal{M}_{SQ} \circ Q,$ i.e., we construct a spanning set $Q'$ of maximum size less than or equal to $r(\mathcal{M}_{SQ}) = r(\mathcal{M}_{S}) + r(\mathcal{M}_{Q}).$ We branch into two cases:

\textbf{Case 1.} If $|Q| \geq r(\mathcal{M}_{SQ}),$ we choose $Q'$ of size equal to $r(\mathcal{M}_{SQ}).$ In this case, $r(\mathcal{M}_{SQ} \times S) = 0$ and:
\[
|P| = r(\mathcal{M}_{SQ} \circ S) - r(\mathcal{M}_{SQ} \times S) = r(\mathcal{M}_{SQ} \circ S) = r(\mathcal{M}_{S}).
\]
Furthermore, we have $|Q| - r(\mathcal{M}_{Q}) = r(\mathcal{M}_{Q}^*) \geq r(\mathcal{M}_{S}).$

\textbf{Case 2.} If $|Q| < r(\mathcal{M}_{SQ}),$ we choose $Q' = Q.$ We then obtain a base of $\mathcal{M}_{SQ}$ that intersects $S$ in a base of $\mathcal{M}_{SQ} \times S.$ Therefore:
\[
r(\mathcal{M}_{SQ} \times S) = r(\mathcal{M}_{S}) + r(\mathcal{M}_{Q}) - |Q| = r(\mathcal{M}_{S}) - r(\mathcal{M}_{Q}^*),
\]
which yields:
\[
|P| = r(\mathcal{M}_{SQ} \circ S) - r(\mathcal{M}_{SQ} \times S) = r(\mathcal{M}_{S}) - (r(\mathcal{M}_{S}) - r(\mathcal{M}_{Q}^*)) = r(\mathcal{M}_{Q}^*).
\]
In this case, $r(\mathcal{M}_{Q}^*) \leq r(\mathcal{M}_{S}).$

Thus, if  $|P| = \min\{r(\mathcal{M}_{S}), r(\mathcal{M}_{Q}^*)\},$ we always  have $|P|=  r(\mathcal{M}_{SQ} \circ S) - r(\mathcal{M}_{SQ} \times S) $ showing that $\hat{\mathcal{M}}_{SP} \leftrightarrow_{\mathcal{M}} \hat{\mathcal{M}}_{PQ}$ forms a minimal decomposition of $\mathcal{M}_{SQ}.$
\end{enumerate}
\end{proof}
\begin{definition}
We refer to the matroid whose bases are unions of independent sets of $\mathcal{M}_S$ and spanning sets of $\mathcal{M}_Q,$ of size $r(\mathcal{M}_S) + r(\mathcal{M}_Q),$ by $RC(\mathcal{M}_S, \mathcal{M}_Q).$
\end{definition}

We note that $(RC(\mathcal{M}_S, \mathcal{M}_Q)) \circ S = \mathcal{M}_S$ and $(RC(\mathcal{M}_S, \mathcal{M}_Q)) \times Q = \mathcal{M}_Q,$ and the definition of $RC(\mathcal{M}_S, \mathcal{M}_Q)$ is identical to that of $\mathcal{M}_S \Box \mathcal{M}_Q,$ the free product of \cite{crapofree}.

We next relate the principal sum of two matroids to $\{S,Q\}$--complete matroids.

\begin{definition}
\label{def:principalsum} \cite{boninsemidirect}.
Let $\mathcal{M}_S$ and $\mathcal{M}_Q$ be matroids on disjoint sets $S$ and $Q,$ respectively. Let $A \subseteq S$ and $B \subseteq Q.$ Let $r_{\mathcal{M}_S}(\cdot)$ be the rank function of $\mathcal{M}_S.$ Let $\mathcal{M}_{SB}$ be the matroid on $S \uplus B,$ with rank function defined as follows:
\begin{align*}
    r(X) &:\equiv r_{\mathcal{M}_S}(X), \quad X \subseteq S; \\
    r(X \uplus B_1) &:\equiv r_{\mathcal{M}_S}(X) + \min\bigl([r_{\mathcal{M}_S}(X \cup A) - r_{\mathcal{M}_S}(X)], \, |B_1|\bigr), \quad X \subseteq S, \; B_1 \subseteq B.
\end{align*}
The principal sum $(\mathcal{M}_S, \mathcal{M}_Q; A, B),$ where $A \subseteq S$ and $B \subseteq Q,$ of matroids $\mathcal{M}_S$ and $\mathcal{M}_Q,$ is defined to be the matroid $\mathcal{M}_{SB} \vee \mathcal{M}_Q.$
\end{definition}

\begin{theorem}
\label{thm:sqcompletepartitionbased}
Let $\mathcal{M}_S$ and $\mathcal{M}_Q$ be matroids on disjoint sets $S$ and $Q,$ respectively. Let $A \subseteq S$ and $B \subseteq Q.$ We have the following.
\begin{enumerate}
    \item The bases of the matroid $\mathcal{M}_{SB}$ are the unions of two disjoint sets, $b_S \uplus B',$ such that $b_S = \hat{b}_S - A',$ where $\hat{b}_S$ is a base of $\mathcal{M}_S,$ and $A' \subseteq A \cap \hat{b}_S, \quad B' \subseteq B, \quad |A'| = |B'|.$
    
    \item The bases of the matroid $\mathcal{M}_{SB} \vee \mathcal{M}_Q$ are the unions of two disjoint sets, $b_S \uplus b_Q,$ such that $b_S = \hat{b}_S - A'$ and $b_Q = \hat{b}_Q \uplus B',$ where $\hat{b}_S$ and $\hat{b}_Q$ are bases of $\mathcal{M}_S$ and $\mathcal{M}_Q,$ respectively, and the subsets satisfy $A' \subseteq A \cap \hat{b}_S, \quad B' \subseteq B \cap (Q - \hat{b}_Q), \quad |A'| = |B'|.$
    
    \item $(\mathcal{M}_S, \mathcal{M}_Q; A, B)$ is $\{S,Q\}$--complete.
    
    \item $(\mathcal{M}_S, \mathcal{M}_Q; A, B)^* = (\mathcal{M}^*_Q, \mathcal{M}^*_S; B, A)$ \cite{boninsemidirect}.
\end{enumerate}
\end{theorem}
\begin{proof}
\begin{enumerate}
    \item It is clear that a base of $\mathcal{M}_S$ is also a base of $\mathcal{M}_{SB}.$
    Let $X \uplus Y,$ with $X \subseteq S$ and $Y \subseteq B,$ be a base of $\mathcal{M}_{SB}.$ Then it is independent and has cardinality $r(\mathcal{M}_S).$ The rank formula implies that $X$ is independent in $\mathcal{M}_S$ and:
    \[|Y| \le r_{\mathcal{M}_S}(X \cup A) - r_{\mathcal{M}_S}(X).\]
    Hence there is a set $A' \subseteq A - X$ of size $|Y|$ such that $X \cup A'$ is independent. Since $|X| + |A'| = r(\mathcal{M}_S),$ the set $\hat{b}_S = X \cup A'$ is a base of $\mathcal{M}_S,$ and $X = \hat{b}_S - A'.$
    
    Conversely, if $\hat{b}_S$ is a base, $A' \subseteq A \cap \hat{b}_S,$ $Y \subseteq B,$ and $|A'| = |Y|,$ then for $X = \hat{b}_S - A'$ the rank formula gives:
    \[r(X \uplus Y) = |X| + |Y| = r(\mathcal{M}_S),\]
    so that $X \uplus Y$ is a base.

    \item Consider the matroid $\mathcal{M}_{SB} \vee \mathcal{M}_Q.$ Let $\mathcal{F}_{SQ}$ be the family of unions of subsets $b_S \uplus b_Q$ where $b_S = \hat{b}_S - A'$ and $b_Q = \hat{b}_Q \uplus B',$ with $\hat{b}_S$ and $\hat{b}_Q$ being bases respectively of $\mathcal{M}_S$ and $\mathcal{M}_Q,$ where $A' \subseteq \hat{b}_S \cap A, \quad B' \subseteq (Q - \hat{b}_Q) \cap B, \quad |A'| = |B'|.$
    
    It is clear that $\mathcal{F}_{SQ}$ is contained in the base family of $\mathcal{M}_{SB} \vee \mathcal{M}_Q.$ We will prove that the base family of $\mathcal{M}_{SB} \vee \mathcal{M}_Q$ is contained in $\mathcal{F}_{SQ}.$
    
    Since any base of $\mathcal{M}_S \vee \mathcal{M}_Q$ is a base of $\mathcal{M}_{SB} \vee \mathcal{M}_Q,$ there exist disjoint bases of $\mathcal{M}_{SB}$ and $\mathcal{M}_Q$ whose union is a base of $\mathcal{M}_{SB} \vee \mathcal{M}_Q.$ Since a base of $\mathcal{M}_S$ is a base of $\mathcal{M}_{SB},$ and all bases of $\mathcal{M}_{SB} \vee \mathcal{M}_Q$ have the same size, it follows that every base of $\mathcal{M}_{SB} \vee \mathcal{M}_Q$ has the form $b_{SB} \uplus b_Q,$ where $b_{SB}$ and $b_Q$ are respectively bases of $\mathcal{M}_{SB}$ and $\mathcal{M}_Q.$
    
    But in part 1 above, we have seen that $b_{SB}$ is of the form $(\hat{b}_S - A') \uplus B',$ where $\hat{b}_S$ is a base of $\mathcal{M}_S,$ $A' \subseteq \hat{b}_S,$ and $B' \subseteq B.$ Since $b_Q$ is a base of $\mathcal{M}_Q$ disjoint from $(\hat{b}_S - A') \uplus B',$ every base of $\mathcal{M}_{SB} \vee \mathcal{M}_Q$ belongs to $\mathcal{F}_{SQ}.$ Thus it follows that $\mathcal{F}_{SQ}$ is the base family of $\mathcal{M}_{SB} \vee \mathcal{M}_Q.$

    \item If $b'_S \uplus b'_Q, \, b'_S \uplus b''_Q, \, b''_S \uplus b'_Q$ are bases of $(\mathcal{M}_S, \mathcal{M}_Q; A, B),$ we have bases $\hat{b}'_S, \hat{b}''_S$ of $\mathcal{M}_S,$ bases $\hat{b}'_Q, \hat{b}''_Q$ of $\mathcal{M}_Q,$ and subsets $A' \subseteq A \cap \hat{b}'_S, \quad A'' \subseteq A \cap \hat{b}''_S, \quad B' \subseteq B \cap (Q - \hat{b}'_Q), \quad B'' \subseteq B \cap (Q - \hat{b}''_Q),$ with $|A'| = |A''| = |B'| = |B''|$ such that:
    \begin{align*}
        b'_S \uplus b'_Q &= (\hat{b}'_S - A') \uplus (\hat{b}'_Q \uplus B'), \\
        b'_S \uplus b''_Q &= (\hat{b}'_S - A') \uplus (\hat{b}''_Q \uplus B''), \\
        b''_S \uplus b'_Q &= (\hat{b}''_S - A'') \uplus (\hat{b}'_Q \uplus B').
    \end{align*}
    It is clear that $b''_S \uplus b''_Q$ is a base of $\mathcal{M}_{SB} \vee \mathcal{M}_Q$ iff $r(\mathcal{M}_S) - |b''_S| = |b''_Q| - r(\mathcal{M}_Q).$ Since $r(\mathcal{M}_S) - |b'_S| = |b'_Q| - r(\mathcal{M}_Q) = |b''_Q| - r(\mathcal{M}_Q)$ and $r(\mathcal{M}_S) - |b''_S| = |b'_Q| - r(\mathcal{M}_Q),$ it follows that $r(\mathcal{M}_S) - |b''_S| = |b''_Q| - r(\mathcal{M}_Q).$ Thus $b''_S \uplus b''_Q$ is a base of $\mathcal{M}_{SB} \vee \mathcal{M}_Q = (\mathcal{M}_S, \mathcal{M}_Q; A, B).$

    \item It is clear that subsets are bases of $(\mathcal{M}_S, \mathcal{M}_Q; A, B)^*$ iff they have the form $((S - \hat{b}'_S) \uplus A') \uplus ((Q - \hat{b}'_Q) - B'),$ where $(S - \hat{b}'_S)$ and $(Q - \hat{b}'_Q)$ are bases respectively of $\mathcal{M}^*_S$ and $\mathcal{M}^*_Q,$ and where $A' \subseteq A \cap \hat{b}'_S = A \cap (S - (S - \hat{b}'_S)), \quad B' \subseteq B \cap (Q - \hat{b}'_Q),$ with $|A'| = |B'|.$ By part 2 above, this is precisely the condition for being bases of $(\mathcal{M}^*_Q, \mathcal{M}^*_S; B, A).$
\end{enumerate}
\end{proof}
\begin{corollary}
\label{cor:decomposeprincipalsum}
Let $\mathcal{M}_{SP}$ be the matroid $(\mathcal{M}_S, \mathbf{0}_{P}; A, P),$ and let $\mathcal{M}_{PQ} :\equiv (\mathcal{M}^*_Q, \mathbf{0}_{P}; B, P)^*,$ with $A\subseteq S, B\subseteq Q, P$ disjoint from $S\cup Q, |P| = \min\{r(\mathcal{M}_S \circ A), \, r(\mathcal{M}^*_Q \circ B)\}.$ Then $(\mathcal{M}_S, \mathcal{M}_Q; A, B) = \mathcal{M}_{SP} \leftrightarrow_{\mathcal{M}} \mathcal{M}_{PQ}.$ This is a minimal decomposition for $(\mathcal{M}_S, \mathcal{M}_Q; A, B).$
 Further,  $\mathcal{E}_P(\mathcal{M}_{SP}^*) = \mathcal{E}_P(\mathcal{M}_{PQ}).$
\end{corollary}

\begin{proof}
The bases of $\mathcal{M}_{SP} :\equiv (\mathcal{M}_S, \mathbf{0}_{P}; A, P)$ are precisely the subsets that have the form $(\hat{b}_S - A') \uplus P',$ where $\hat{b}_S$ is a base of $\mathcal{M}_S, \quad A' \subseteq A \cap \hat{b}_S, \quad P' \subseteq P, \quad |A'| = |P'|.$ Those of $\mathcal{M}^*_{PQ} :\equiv (\mathcal{M}^*_Q, \mathbf{0}_{P}; B, P)$ are precisely the subsets that have the form $((Q - \hat{b}_Q) - B'') \uplus P'',$ where $\hat{b}_Q$ is a base of $\mathcal{M}_Q, \quad B'' \subseteq B \cap (Q - \hat{b}_Q), \quad P'' \subseteq P, \quad |B''| = |P''|.$ Those of $\mathcal{M}_{PQ} :\equiv (\mathcal{M}^*_Q, \mathbf{0}_{P}; B, P)^*$ are precisely the sets that have the form $(\hat{b}_Q \uplus B'') \uplus (P - P''),$ where $B'' \subseteq B \cap (Q - \hat{b}_Q), \quad P'' \subseteq P, \quad |B''| = |P''|.$

Now the bases $(\hat{b}_S - A') \uplus P'$ and $(\hat{b}_Q \uplus B'') \uplus (P - P'')$ of $\mathcal{M}_{SP}$ and $\mathcal{M}_{PQ},$ respectively, are disjoint and cover $P$ iff $P' = P''.$ It follows that the bases of $\mathcal{M}_{SP} \leftrightarrow_{\mathcal{M}} \mathcal{M}_{PQ} = (\mathcal{M}_{SP} \vee \mathcal{M}_{PQ}) \times (S \uplus Q)$ are precisely the subsets with the form $(\hat{b}_S - A') \uplus (\hat{b}_Q \uplus B''),$ where $A' \subseteq A \cap \hat{b}_S, \quad B'' \subseteq B \cap (Q - \hat{b}_Q), \quad |A'| = |B''|.$ In part 2 above, we have seen that bases of $(\mathcal{M}_S, \mathcal{M}_{Q}; A, B)$ are precisely the sets with this form. Thus $(\mathcal{M}_S, \mathcal{M}_{Q}; A, B) = \mathcal{M}_{SP} \leftrightarrow_{\mathcal{M}} \mathcal{M}_{PQ}.$

Next, we show that by the definition of $\mathcal{M}_{SP}$ and $\mathcal{M}_{PQ},$ it follows that $P$ is contained in a base as well as in a cobase of $\mathcal{M}_{SP}$ and $\mathcal{M}_{PQ}.$ From the definition of $\mathcal{M}_{SP},$ it is clear that any base of $\mathcal{M}_S$ is a base of $\mathcal{M}_{SP}.$ Further, since $|P| \le r(\mathcal{M}_S \circ A),$ there exist $A' \subseteq A$ independent in $\mathcal{M}_S$ such that $|A'| = |P|,$ and a base $\hat{b}_S$ of $\mathcal{M}_S$ such that $A' \subseteq \hat{b}_S.$ Therefore there exist cobases of $\mathcal{M}_{SP}$ which contain $P$ and also bases of $\mathcal{M}_{SP}$ which contain $P.$

Similarly, it can be seen that there exist cobases of $\mathcal{M}^*_{PQ} = (\mathcal{M}^*_Q, \mathbf{0}_{P}; B, P)$ which contain $P$ and also bases of $\mathcal{M}^*_{PQ}$ which contain $P.$ Therefore there exist cobases of $\mathcal{M}_{PQ}$ which contain $P$ and also bases of $\mathcal{M}_{PQ}$ which contain $P.$

It follows that $r(\mathcal{M}_{SP} \circ P) = |P|$ and $r(\mathcal{M}_{SP} \times P) = 0.$ Therefore, $|P| = r(\mathcal{M}_{SP} \circ P) - r(\mathcal{M}_{SP} \times P) = r(\mathcal{M}_{SP} \circ S) - r(\mathcal{M}_{SP} \times S).$

Let $\mathcal{M}_{SQ} :\equiv (\mathcal{M}_S, \mathcal{M}_Q; A, B).$ We have:
\begin{align*}
    \mathcal{M}_{SQ} \circ S &= (\mathcal{M}_{SP} \leftrightarrow_{\mathcal{M}} \mathcal{M}_{PQ}) \leftrightarrow_{\mathcal{M}} \mathcal{F}_Q \\
    &= \mathcal{M}_{SP} \leftrightarrow_{\mathcal{M}} (\mathcal{M}_{PQ} \leftrightarrow_{\mathcal{M}} \mathcal{F}_Q) \\
    &= \mathcal{M}_{SP} \leftrightarrow_{\mathcal{M}} (\mathcal{M}_{PQ} \circ P) \\
    &= \mathcal{M}_{SP} \leftrightarrow_{\mathcal{M}} \mathcal{F}_P = \mathcal{M}_{SP} \circ S.
\end{align*}
Similarly, $\mathcal{M}_{SQ} \times S = \mathcal{M}_{SP} \times S.$ Therefore, $|P| = r(\mathcal{M}_{SQ} \circ S) - r(\mathcal{M}_{SQ} \times S).$ By Theorem \ref{thm:lowerboundP}, it follows that $(\mathcal{M}_S, \mathcal{M}_{Q}; A, B)$ is minimally decomposed into $\mathcal{M}_{SP}$ and $\mathcal{M}_{PQ}.$

 When $|P| = \min\{r(\mathcal{M}_S \circ A), \, r(\mathcal{M}^*_Q \circ B)\},$ 
 it can be verified that  $\mathcal{E}_P(\mathcal{M}_{SP}^*) , \mathcal{E}_P(\mathcal{M}_{PQ})$ are both  partitions of the set $2^P$ into subsets of the same cardinality and therefore the same.

\end{proof}
\section{Conclusion}

We have introduced and studied the notion of $\{S,Q\}$--completeness in matroids on $S \uplus Q,$ using the idea of the linking operation $\mathcal{M}_{SP} \leftrightarrow_{\mathcal{M}} \mathcal{M}_{PQ} :\equiv (\mathcal{M}_{SP} \vee \mathcal{M}_{PQ}) \times (S \uplus Q)$ defined in \cite{STHN2014}. A matroid $\mathcal{M}_{SQ}$ is $\{S,Q\}$--complete iff whenever $b_S \uplus b_Q, \, b_S \uplus b'_Q, \, b'_S \uplus b_Q$ are bases of $\mathcal{M}_{SQ},$ so is $b'_S \uplus b'_Q.$ The class of $\{S,Q\}$--complete matroids includes the well-known free products and principal sums of matroids (\cite{crapofree}, \cite{boninsemidirect}). The linking operation is analogous to the matched composition ``$\mathcal{V}_{SP} \leftrightarrow_{\mathcal{V}} \mathcal{V}_{PQ}$'' of vector spaces defined in \cite{HNarayanan1986a}.

The fundamental problem concerning the linking of matroids is, given $\mathcal{M}_{SP} \leftrightarrow_{\mathcal{M}} \mathcal{M}_{PQ},$ to define matroids $\mathcal{M}_{S\hat{P}}$ and $\mathcal{M}_{\hat{P}Q}$ such that $\mathcal{M}_{S\hat{P}} \leftrightarrow_{\mathcal{M}} \mathcal{M}_{\hat{P}Q} = \mathcal{M}_{SP} \leftrightarrow_{\mathcal{M}} \mathcal{M}_{PQ}$ with $|\hat{P}|$ minimized. This is analogous to a corresponding problem for vector spaces which has been solved completely (\cite{HNarayanan1986a}, \cite{narayanan1987topological}). Our approach for matroids was to mimic what worked for vector spaces as far as possible.

We have shown that $|\hat{P}| \ge r((\mathcal{M}_{SP} \leftrightarrow_{\mathcal{M}} \mathcal{M}_{PQ}) \circ S) - r((\mathcal{M}_{SP} \leftrightarrow_{\mathcal{M}} \mathcal{M}_{PQ}) \times S)$ and that this lower bound can be achieved if the matroids $\mathcal{M}_{SP}$ and $\mathcal{M}_{PQ}$ that we start with, satisfy the conditions for being strongly compatible:
\begin{itemize}
    \item $\mathcal{M}_{SP} \circ P = \mathcal{M}^*_{PQ} \circ P$ and $\mathcal{M}_{SP} \times P = \mathcal{M}^*_{PQ} \times P.$
    \item For every specified pair of disjoint bases $b^3_P$ and $b^4_P$ of $\mathcal{M}_{SP} \times P$ and $\mathcal{M}_{PQ} \times P$ respectively and every base $b_S \uplus b_Q$ of $\mathcal{M}_{SP} \leftrightarrow_{\mathcal{M}} \mathcal{M}_{PQ},$ there exist disjoint bases $b_S \uplus b^1_P$ and $b_Q \uplus b^2_P$ of $\mathcal{M}_{SP}$ and $\mathcal{M}_{PQ},$ respectively, such that $b^1_P \supseteq b^3_P$ and $b_Q \uplus b^2_P \supseteq b^4_P.$
\end{itemize}

We have shown that if a matroid $\mathcal{M}_{SQ}:\equiv \mathcal{M}_{SP} \leftrightarrow_{\mathcal{M}} \mathcal{M}_{PQ},$  is $\{S,Q\}$--complete, with $\mathcal{M}_{SP},  \mathcal{M}_{PQ},$ being respectively  $\{S,P\}$--complete, $\{P,Q\}$--complete, then strong compatibility of $\mathcal{M}_{SP},  \mathcal{M}_{PQ},$ reduces to compatibility plus the simple condition $\mathcal{E}_P(\mathcal{M}_{SP}^*) = \mathcal{E}_P(\mathcal{M}_{PQ}).$ 
 We assert that the problem of characterizing a matroid with a given partition of the underlying set, for which minimal decomposition in the above sense is possible, remains open, since $\{S,Q\}$--completeness is only a sufficient condition.

\appendix
\section{Basic matroid results}
\label{subsec:matroidresults}

This section follows \cite{STHN2014}. Proofs of the results might be found in \cite{oxley}.

We have the following set of basic results regarding minors and duals.

\begin{theorem} 
\label{thm:dotcrossidentitym}
Let $S \cap P = \emptyset, \quad T_2 \subseteq T_1 \subseteq S,$ let $\mathcal{M}_S$ be a matroid on $S,$ and let $\mathcal{M}_{SP}$ be a matroid on $S \uplus P.$ We then have the following. 
\begin{enumerate}
    \item Let $b_S, b'_S$ be distinct bases of $\mathcal{M}_S$ and let $e' \in (b'_S - b_S).$ Then there is a circuit $C$ of $\mathcal{M}_S$ contained in $e' \uplus b_S$ and an element $e \in C \cap b_S.$ The subset $(b_S - e) \uplus e'$ is a base of $\mathcal{M}_S.$ 
    \item $(\mathcal{M}_S \times T_1) \circ T_2 = (\mathcal{M}_S \circ (S - (T_1 - T_2))) \times T_2.$
    \item $r(\mathcal{M}_{SP}) = r(\mathcal{M}_{SP} \circ S) + r(\mathcal{M}_{SP} \times P).$
    \item $r(\mathcal{M}_S) + r(\mathcal{M}^{*}_S) = |S|$ and $(\mathcal{M}_S)^{**} = \mathcal{M}_S.$
    \item $T$ is independent in $\mathcal{M}_S$ iff $S - T$ is a spanning set of $\mathcal{M}^{*}_S.$
    \item $\mathcal{M}_{SP}^{*} \circ P = (\mathcal{M}_{SP} \times P)^{*}$ and $\mathcal{M}_{SP}^{*} \times S = (\mathcal{M}_{SP} \circ S)^{*}.$
    \item Every base of $\mathcal{M}_{SP} \times P$ ($\mathcal{M}_{SP} \circ P$) is contained in (contains) a base of $\mathcal{M}_{SP} \circ P$ ($\mathcal{M}_{SP} \times P$). 
    \item A subset $b_P$ is a base of $\mathcal{M}_{SP} \times P$ iff there exists $b_S \subseteq S$ such that $b_S$ is a base of $\mathcal{M}_{SP} \circ S$ and $b_S \uplus b_P$ is a base of $\mathcal{M}_{SP}.$
    \item If $b_S \uplus b_P$ is a base of $\mathcal{M}_{SP},$ then $b_S$ is contained in a base of $\mathcal{M}_{SP} \circ S$ and contains a base of $\mathcal{M}_{SP} \times S.$
    \item A subset $C$ is a bond (circuit) of $\mathcal{M}_{PQ} \times P$ ($\mathcal{M}_{PQ} \circ P$) iff it is a bond (circuit) of $\mathcal{M}_{PQ}$ contained in $P.$ 
\end{enumerate} 
\end{theorem}

The following is a  basic result about matroid unions.

\begin{theorem}
\label{thm:sumintersectionm}
Let $\mathcal{M}^1_A, \, \mathcal{M}^2_B, \, \mathcal{M}_S,$ and $\mathcal{M}'_S$ be matroids. Then:
\begin{enumerate}
    \item $r(\mathcal{M}_S) + r(\mathcal{M}'_S) = r(\mathcal{M}_S \vee \mathcal{M}'_S) + r(\mathcal{M}_S \wedge \mathcal{M}'_S).$
    \item We have the following distributive properties under restriction and contraction over a subset $X \subseteq A \cap B$:
    \begin{align*}
        (\mathcal{M}^1_A \vee \mathcal{M}^2_B) \circ X &= \mathcal{M}^1_A \circ X \vee \mathcal{M}^2_B \circ X, \\
        (\mathcal{M}^1_A \wedge \mathcal{M}^2_B) \times X &= \mathcal{M}^1_A \times X \wedge \mathcal{M}^2_B \times X.
    \end{align*}
\end{enumerate}
\end{theorem}
\section{Proof of Theorem \ref{thm:idtmatroid}, extract from \cite{STHN2014}}
\label{sec:idtmatroid}

\begin{proof}
\begin{enumerate}
    \item Let $b^1_S \uplus b^2_Q$ be a base of LHS. Then there exist bases $b^1_S \uplus b_P^1$ and $b_P^2 \uplus b^2_Q$ of $\mathcal{M}_{SP}$ and $\mathcal{M}_{PQ},$ respectively, such that $b_P^1 \cap b_P^2$ is minimal, and among all such pairs of bases of $\mathcal{M}_{SP}$ and $\mathcal{M}_{PQ},$ the set $b^1_S \uplus b^2_Q$ is minimal.
    
    Since $\mathcal{M}_{SP} \wedge \mathcal{M}_{PQ} :\equiv (\mathcal{M}_{SP} \oplus \mathcal{F}_Q) \wedge (\mathcal{M}_{PQ} \oplus \mathcal{F}_S),$ the set $b^1_S \uplus b^2_Q \uplus (b_P^1 \cap b_P^2)$ is a minimal intersection of bases of the form $b_S \uplus b_P' \uplus Q$ and $S \uplus b_P'' \uplus b_Q$ of $(\mathcal{M}_{SP} \oplus \mathcal{F}_Q)$ and $(\mathcal{M}_{PQ} \oplus \mathcal{F}_S),$ respectively, and therefore is a base of $\mathcal{M}_{SP} \wedge \mathcal{M}_{PQ}.$
    
    Suppose $b^1_S \uplus b^2_Q$ is not a maximal intersection with $S \uplus Q$ among all bases of $\mathcal{M}_{SP} \wedge \mathcal{M}_{PQ}.$ Then there exist bases $\tilde{b}_S \uplus \tilde{b}_P^1$ and $\tilde{b}_P^2 \uplus \hat{b}_Q$ of $\mathcal{M}_{SP}$ and $\mathcal{M}_{PQ},$ respectively, such that $\tilde{b}_S \uplus \tilde{b}_Q \uplus (\tilde{b}_P^1 \cap \tilde{b}_P^2)$ is a base of $\mathcal{M}_{SP} \wedge \mathcal{M}_{PQ}$ and $\tilde{b}_S \uplus \tilde{b}_Q \supset b^1_S \uplus b^2_Q.$ But this implies $|\tilde{b}_P^1 \cap \tilde{b}_P^2| < |b_P^1 \cap b_P^2|,$ which is a contradiction. Hence $b^1_S \uplus b^2_Q$ is a base of RHS.

    Next, let $b^1_S \uplus b^2_Q$ be a base of RHS. Then there exist bases $b^1_S \uplus b_P^1$ and $b_P^2 \uplus b^2_Q$ of $\mathcal{M}_{SP}$ and $\mathcal{M}_{PQ},$ respectively, such that $b^1_S \uplus b^2_Q \uplus (b_P^1 \cap b_P^2)$ is a base of $\mathcal{M}_{SP} \wedge \mathcal{M}_{PQ}$ and further (since $b^1_S \uplus b^2_Q$ is a maximal intersection with $S \uplus Q$ among bases of $\mathcal{M}_{SP} \wedge \mathcal{M}_{PQ}$), for all such pairs of bases, $b_P^1 \cap b_P^2$ is minimal.
    
    Hence $b^1_S \uplus b^2_Q \uplus b_P^1 \uplus b_P^2$ is a base of $\mathcal{M}_{SP} \vee \mathcal{M}_{PQ}.$ Further, $b^1_S \uplus b^2_Q \uplus (b_P^1 \cap b_P^2),$ being a base of $\mathcal{M}_{SP} \wedge \mathcal{M}_{PQ},$ is minimal among all pairs of bases of the form $b_S \uplus b_P' \uplus Q$ and $b_P'' \uplus b_Q \uplus S$ of $\mathcal{M}_{SP} \oplus \mathcal{F}_Q$ and $\mathcal{M}_{PQ} \oplus \mathcal{F}_S,$ respectively. Therefore, for the given intersection $b_P^1 \cap b_P^2,$ the set $b^1_S \uplus b^2_Q$ is a minimal intersection with $S \uplus Q$ among bases of $\mathcal{M}_{SP} \vee \mathcal{M}_{PQ},$ i.e., it is a base of LHS.

    \item We can apply minor duality operators sequentially to obtain the result:
    \begin{align*}
        (\mathcal{M}_{SP} \leftrightarrow \mathcal{M}_{PQ})^* &= [(\mathcal{M}_{SP} \vee \mathcal{M}_{PQ}) \times (S \cup Q)]^* \\
        &= (\mathcal{M}_{SP} \vee \mathcal{M}_{PQ})^* \circ (S \cup Q) \\
        &= (\mathcal{M}_{SP}^* \wedge \mathcal{M}_{PQ}^*) \circ (S \cup Q) \\
        &= (\mathcal{M}_{SP}^* \vee \mathcal{M}_{PQ}^*) \times (S \cup Q) \\
        &= \mathcal{M}_{SP}^* \leftrightarrow \mathcal{M}_{PQ}^*
    \end{align*}
\end{enumerate}
\end{proof}

\begin{remark}
We note that if bases $b^1_S \cup b_P^1$ and $b_P^2 \cup b^2_Q$ of $\mathcal{M}_{SP}$ and $\mathcal{M}_{PQ},$ respectively, are such that $b_P^1 \cap b_P^2$ is minimal, in order that $b^1_S \cup (b_P^1 \cap b_P^2) \cup b^2_Q$ is a base of $\mathcal{M}_{SP} \wedge \mathcal{M}_{PQ},$ we need the additional condition that $b^1_S \cup b^2_Q$ is minimal among all bases of $\mathcal{M}_{SP} \vee \mathcal{M}_{PQ}.$

Similarly, if $b^1_S \cup (b_P^1 \cap b_P^2) \cup b^2_Q$ is a base of $\mathcal{M}_{SP} \wedge \mathcal{M}_{PQ},$ in order that $b^1_S \cup b_P^1 \cup b_P^2 \cup b^2_Q$ is a base of $\mathcal{M}_{SP} \vee \mathcal{M}_{PQ},$ we need the additional condition that $b^1_S \cup b^2_Q$ is maximal among such bases of $\mathcal{M}_{SP} \wedge \mathcal{M}_{PQ}.$
\end{remark}
\section{Proof of impossibility of minimal decomposition of matroid of Figure \ref{fig:matroid1}}
\label{sec:mindecompmatroid}

We will show that the matroid $\mathcal{M}_{SQ}$ on $\{e_1,e_2,e_3,e_4,e_5,e_6\}$ with $S :\equiv \{e_1,e_2,e_3\}$ and $Q :\equiv \{e_4,e_5,e_6\}$ cannot be minimally decomposed into $\mathcal{M}_{SP} \leftrightarrow_{\mathcal{M}} \mathcal{M}_{PQ}.$

If $\mathcal{M}_{SQ}$ is to be minimally decomposed as $\mathcal{M}_{SP} \leftrightarrow_{\mathcal{M}} \mathcal{M}_{PQ},$ we must have $|P| = r(\mathcal{M}_{SQ} \circ S) - r(\mathcal{M}_{SQ} \times S) = 2.$ Let $P :\equiv \{p_1,p_2\}.$

Now, by Theorem \ref{thm:lowerboundP}, we have that $r(\mathcal{M}_{SQ} \circ S) - r(\mathcal{M}_{SQ} \times S) \le r(\mathcal{M}_{SP} \circ S) - r(\mathcal{M}_{SP} \times S)$ and $r(\mathcal{M}_{SQ} \circ Q) - r(\mathcal{M}_{SQ} \times Q) \le r(\mathcal{M}_{PQ} \circ Q) - r(\mathcal{M}_{PQ} \times Q).$ Further, $r(\mathcal{M}_{SP} \circ S) - r(\mathcal{M}_{SP} \times S) = r(\mathcal{M}_{SP} \circ P) - r(\mathcal{M}_{SP} \times P) \le |P|$ and $r(\mathcal{M}_{PQ} \circ Q) - r(\mathcal{M}_{PQ} \times Q) = r(\mathcal{M}_{PQ} \circ P) - r(\mathcal{M}_{PQ} \times P) \le |P|.$

Therefore, in the matroids $\mathcal{M}_{SP}$ and $\mathcal{M}_{PQ},$ for a minimal decomposition, we must have $r(\mathcal{M}_{SP} \circ P) = r(\mathcal{M}_{PQ} \circ P) = |P|$ and $r(\mathcal{M}_{SP} \times P) = r(\mathcal{M}_{PQ} \times P) = 0.$ We infer the following.

\begin{enumerate}
    \item The elements $p_1, p_2$ cannot be parallel in $\mathcal{M}_{SP}$ or $\mathcal{M}_{PQ}$ since $r(\mathcal{M}_{SP} \circ P) = r(\mathcal{M}_{PQ} \circ P) = 2.$
    \item $\mathcal{M}_{SP}$ and $\mathcal{M}_{PQ}$ are compatible over $P.$
    \item An independent set of $\mathcal{M}_{SQ}$ contained in $S$ or $Q,$ respectively, must be contained in a base of $\mathcal{M}_{SP}$ or $\mathcal{M}_{PQ},$ respectively (Lemma \ref{lem:basecontainedcontains}).
    \item Using Lemma \ref{cor:minorspecial}, we have $\mathcal{M}_{SQ} \circ S = \mathcal{M}_{SP} \circ S, \quad \mathcal{M}_{SQ} \times S = \mathcal{M}_{SP} \times S, \quad \mathcal{M}_{SQ} \circ Q = \mathcal{M}_{PQ} \circ Q,$ and $\mathcal{M}_{SQ} \times Q = \mathcal{M}_{PQ} \times Q.$ Therefore, if a subset of $S$ or $Q$ is a circuit in $\mathcal{M}_{SP}$ or $\mathcal{M}_{PQ},$ respectively, it will remain a circuit in $\mathcal{M}_{SQ}.$ Therefore, an element $p_i$ cannot be parallel to two elements of $\mathcal{M}_{SP}$ or $\mathcal{M}_{PQ}$ since two distinct elements in $S$ or $Q$ are not parallel in $\mathcal{M}_{SQ}.$ 
\end{enumerate}

Now we consider various cases.

\begin{enumerate}
    \item In $\mathcal{M}_{SP},$ the elements $p_1, p_2$ are each independent of each of $e_2, e_3.$
    
    Now in $\mathcal{M}_{PQ},$ at most one of $p_1, p_2,$ say $p_1,$ can be parallel to $e_5.$ For $\mathcal{M}_{SP}$ and $\mathcal{M}_{PQ},$ choose bases $\{p_1, e_3\}$ and $\{p_2, e_5\},$ respectively. In $\mathcal{M}_{SP} \leftrightarrow_{\mathcal{M}} \mathcal{M}_{PQ},$ this will result in the base $\{e_3, e_5\}.$ But $e_3$ and $e_5$ are parallel in $\mathcal{M}_{SQ},$ a contradiction.

    \item In $\mathcal{M}_{SP},$ $p_1$ is parallel to $e_3.$ Then it cannot be parallel to $e_1$ or $e_2.$
    \begin{enumerate}
        \item In $\mathcal{M}_{PQ},$ the elements $p_1, p_2$ are each independent of each of $e_4, e_5.$ This is impossible arguing as in case 1 above.
        
        \item In $\mathcal{M}_{PQ},$ $p_1$ is parallel to $e_4.$ Then $p_1$ cannot be parallel to $e_5.$ For $\mathcal{M}_{SP}$ and $\mathcal{M}_{PQ},$ choose bases $\{p_2, e_3\}$ and $\{p_1, e_5\},$ respectively. This will result in the base $\{e_3, e_5\}$ of $\mathcal{M}_{SP} \leftrightarrow_{\mathcal{M}} \mathcal{M}_{PQ}.$ But $e_3$ and $e_5$ are parallel in $\mathcal{M}_{SQ},$ a contradiction.
        
        \item In $\mathcal{M}_{PQ},$ $p_1$ is parallel to $e_5.$ We consider two sub-possibilities.
        
        If $p_2$ is independent of $e_6,$ for $\mathcal{M}_{SP}$ and $\mathcal{M}_{PQ},$ choose bases $\{p_1, e_1\}$ and $\{p_2, e_6\},$ respectively. This will result in the base $\{e_1, e_6\}$ of $\mathcal{M}_{SP} \leftrightarrow_{\mathcal{M}} \mathcal{M}_{PQ}.$ But $e_1$ and $e_6$ are parallel in $\mathcal{M}_{SQ},$ a contradiction.
        
        If $p_2$ is parallel to $e_6,$ it cannot also be parallel to $e_4.$ For $\mathcal{M}_{SP}$ and $\mathcal{M}_{PQ},$ choose bases $\{p_1, e_2\}$ and $\{p_2, e_4\},$ respectively. This will result in the base $\{e_2, e_4\}$ of $\mathcal{M}_{SP} \leftrightarrow_{\mathcal{M}} \mathcal{M}_{PQ}.$ But $e_2$ and $e_4$ are parallel in $\mathcal{M}_{SQ},$ a contradiction.
        
        \item In $\mathcal{M}_{PQ},$ $p_2$ is parallel to $e_4$ or $e_5.$ Choose bases $\{p_1, e_1\}$ and $\{p_2, e_6\},$ respectively. This will result in the base $\{e_1, e_6\}$ of $\mathcal{M}_{SP} \leftrightarrow_{\mathcal{M}} \mathcal{M}_{PQ}.$ But $e_1$ and $e_6$ are parallel in $\mathcal{M}_{SQ},$ a contradiction.
    \end{enumerate}
\end{enumerate}

Therefore the matroid $\mathcal{M}_{SQ}$ on $\{e_1,e_2,e_3,e_4,e_5,e_6\}$ with $S :\equiv \{e_1,e_2,e_3\}$ and $Q :\equiv \{e_4,e_5,e_6\}$ cannot be minimally decomposed into $\mathcal{M}_{SP} \leftrightarrow_{\mathcal{M}} \mathcal{M}_{PQ}.$
\section{A Single Bipartition Augmentation Need Not Be\
Bipartition-Complete}

We present  a matroid $\mathcal{M}_{SQ} $ whose $\{S,Q\}$-augmentation
$Aug_{S,Q}(\mathcal{M}_{SQ})$ is not $\{S,Q\}$-complete.

Let
\[
 S=\{s_1,s_2,s_3\},
 \qquad
 Q=\{q_1,q_2,q_3\},
\]
and let $\mathcal{M}_{SQ}$ be the column matroid, over any field $F$, of the
matrix
$$\begin{array}{c|cccccc}
 &s_1&s_2&s_3&q_1&q_2&q_3\\ \hline
 &1&0&0&1&0&0\\
 &0&1&1&0&1&0\\
 &1&1&0&0&0&1
\end{array}.
$$
The matroid has rank three.  Let
$
 Y_{12}:\equiv \{q_1,q_2\},
 Y_{13}:\equiv \{q_1,q_3\},
 Y_{23}:\equiv \{q_2,q_3\}.
$
The bases having one element in $S$ and two elements in
$Q$ are exactly
$
 \{s_1\}\uplus Y_{23},
 \{s_1\}\uplus Y_{12},
 \{s_2\}\uplus Y_{12},
 \{s_2\}\uplus Y_{13},
 \{s_3\}\uplus Y_{13}.
$

A single augmentation
adds (among other bases) the
three bases
$
 \{s_2\}\uplus Y_{23},
 \{s_1\}\uplus Y_{13},
 \{s_3\}\uplus Y_{12}.
$
For example, $\{s_2\}\uplus Y_{23}$ is added because
$
 \{s_1\}\uplus Y_{12},
 \{s_2\}\uplus Y_{12},
 \{s_1\}\uplus Y_{23}
$
are bases of $\mathcal{M}_{SQ}$.

We note that,
$
 \{s_3\}\uplus Y_{23}=\{s_3,q_2,q_3\}
$
is not added in the first augmentation.  To see this, suppose it were added.
 Then there would be subsets $b_S\subseteq S$
and $b_Q\subseteq Q$ such that
$
 b_S\uplus b_Q,
 \{s_3\}\uplus b_Q,
 b_S\uplus Y_{23}
$
were bases of the original matroid.  Since all bases have size three,
$|b_S|=1$ and $|b_Q|=2$.  The displayed list of bases then forces
$b_Q=Y_{13}$, because $Y_{13}$ is the only two-element $Q$-part paired
with $\{s_3\}$, and it forces $b_S=\{s_1\}$, because $\{s_1\}$ is the
only one-element $S$-part paired with $Y_{23}$.  But
$\{s_1\}\uplus Y_{13}$ is not a base
 of $\mathcal{M}_{SQ}$, a contradiction.
Thus
$
 \{s_3\}\uplus Y_{23}$ is not a base of $Aug_{S,Q}(\mathcal{M}_{SQ}).$

Now the following three sets are bases of
$Aug_{S,Q}(\mathcal{M}_{SQ})$:
$
 \{s_2\}\uplus Y_{13},
 \{s_3\}\uplus Y_{13},
 \{s_2\}\uplus Y_{23}.
$
The first two were already bases of $\mathcal{M}_{SQ}$, and the third was added by
the augmentation.  However, as just proved,
$\{s_3\}\uplus Y_{23}$ is not a base
 of $Aug_{S,Q}(\mathcal{M}_{SQ}).$ 
Therefore $Aug_{S,Q}(\mathcal{M}_{SQ})$ is not
$\{S,Q\}$-complete.

\section{Proof of Lemma \ref{lem:pseudoid0}}
\label{sec:pseudoid}

\begin{proof}
\begin{enumerate}
    \item We have:
    \begin{align*}
        \mathcal{M}_{QQ'}^* &= (\mathcal{M}^*_{SQ} \leftrightarrow_{\mathcal{M}} (\mathcal{M}_{SQ})_{SQ'})^* \\
        &= \mathcal{M}_{SQ} \leftrightarrow_{\mathcal{M}} (\mathcal{M}_{SQ})^*_{SQ'}  \quad \text{(using Theorem \ref{thm:idtmatroid})}\\
        &= (\mathcal{M}_{SQ})^*_{SQ'} \leftrightarrow_{\mathcal{M}} \mathcal{M}_{SQ} \\
        &= (\mathcal{M}^*_{SQ} \leftrightarrow_{\mathcal{M}} (\mathcal{M}_{SQ})_{SQ'})_{Q'Q} = (\mathcal{M}_{QQ'})_{Q'Q}.
    \end{align*}
    (See the discussion on copies of matroids in Subsection \ref{subsec:copies}.)

    \item We find for the restriction minor:
    \begin{align*}
        \mathcal{M}_{QQ'} \circ Q &= (\mathcal{M}^*_{SQ} \leftrightarrow_{\mathcal{M}} (\mathcal{M}_{SQ})_{SQ'}) \circ Q \\
        &= (\mathcal{M}^*_{SQ} \leftrightarrow_{\mathcal{M}} (\mathcal{M}_{SQ})_{SQ'}) \leftrightarrow_{\mathcal{M}} \mathcal{F}_{Q'} \\
        &= \mathcal{M}^*_{SQ} \leftrightarrow_{\mathcal{M}} ((\mathcal{M}_{SQ})_{SQ'} \leftrightarrow_{\mathcal{M}} \mathcal{F}_{Q'}) \\
        &= \mathcal{M}^*_{SQ} \leftrightarrow_{\mathcal{M}} (\mathcal{M}_{SQ} \circ S) = \mathcal{M}^*_{SQ} \circ Q,
    \end{align*}
    using Lemma \ref{lem:basecontainedcontains}.
    
    Similarly, for the contraction minor:
    \begin{align*}
        \mathcal{M}_{QQ'} \times Q &= (\mathcal{M}^*_{SQ} \leftrightarrow_{\mathcal{M}} (\mathcal{M}_{SQ})_{SQ'}) \times Q \\
        &= (\mathcal{M}^*_{SQ} \leftrightarrow_{\mathcal{M}} (\mathcal{M}_{SQ})_{SQ'}) \leftrightarrow_{\mathcal{M}} \mathbf{0}_{Q'} \\
        &= \mathcal{M}^*_{SQ} \leftrightarrow_{\mathcal{M}} (\mathcal{M}_{SQ} \times S) = \mathcal{M}^*_{SQ} \times Q,
    \end{align*}
    using Lemma \ref{lem:basecontainedcontains}.

    \item We use the fact that $(\mathcal{M}_{AB})_{AB'} \circ A = \mathcal{M}_{AB} \circ A$ and $(\mathcal{M}_{AB})_{AB'} \times A = \mathcal{M}_{AB} \times A,$ where $A, B, B'$ are pairwise disjoint and $B'$ is a copy of $B.$
    
    Let $S'$ be a copy of $S$ disjoint from $Q,Q'.$  It is clear that $\mathcal{M}_{QQ'} =  (\mathcal{M}_{SQ})^*_{S'Q} \leftrightarrow_{\mathcal{M}}  (\mathcal{M}_{SQ})_{S'Q'} =(\mathcal{M}_{SQ})_{S'Q'} \leftrightarrow_{\mathcal{M}} (\mathcal{M}_{SQ})^*_{S'Q}.$ We now use the restricted associativity of linking (Lemma \ref{lem:restassoc}) to prove the restrictions.
    \begin{align*}
        (\mathcal{M}_{SQ} \leftrightarrow_{\mathcal{M}} \mathcal{M}_{QQ'}) \circ S &= (\mathcal{M}_{SQ} \leftrightarrow_{\mathcal{M}} (\mathcal{M}_{SQ})_{S'Q'} \leftrightarrow_{\mathcal{M}} (\mathcal{M}_{SQ})^*_{S'Q}) \leftrightarrow_{\mathcal{M}} \mathcal{F}_{Q'} \\
        &= (\mathcal{M}_{SQ} \leftrightarrow_{\mathcal{M}} (\mathcal{M}_{SQ})^*_{S'Q}) \leftrightarrow_{\mathcal{M}} ((\mathcal{M}_{SQ})_{S'Q'} \leftrightarrow_{\mathcal{M}} \mathcal{F}_{Q'}) \\
        &= (\mathcal{M}_{SQ} \leftrightarrow_{\mathcal{M}} (\mathcal{M}_{SQ})^*_{S'Q}) \leftrightarrow_{\mathcal{M}} ((\mathcal{M}_{SQ})_{S'Q'} \circ S') \\
        &= \mathcal{M}_{SQ} \leftrightarrow_{\mathcal{M}} ((\mathcal{M}_{SQ})^*_{S'Q} \leftrightarrow_{\mathcal{M}} (\mathcal{M}_{SQ})_{S'Q'} \circ S') \\
        &= \mathcal{M}_{SQ} \leftrightarrow_{\mathcal{M}} ((\mathcal{M}_{SQ})^*_{S'Q} \circ Q) = \mathcal{M}_{SQ} \circ S,
    \end{align*}
    using Lemma \ref{lem:basecontainedcontains}.
    
    Next, 
    \begin{align*}
        (\mathcal{M}_{SQ} \leftrightarrow_{\mathcal{M}} \mathcal{M}_{QQ'}) \circ Q' &= (\mathcal{M}_{SQ} \leftrightarrow_{\mathcal{M}} (\mathcal{M}_{SQ})_{S'Q'} \leftrightarrow_{\mathcal{M}} (\mathcal{M}_{SQ})^*_{S'Q}) \leftrightarrow_{\mathcal{M}} \mathcal{F}_{S} \\
        &= ((\mathcal{M}_{SQ})_{S'Q'} \leftrightarrow_{\mathcal{M}} (\mathcal{M}_{SQ})^*_{S'Q}) \leftrightarrow_{\mathcal{M}} (\mathcal{M}_{SQ} \leftrightarrow_{\mathcal{M}} \mathcal{F}_{S}) \\
        &= (\mathcal{M}_{SQ})_{S'Q'} \leftrightarrow_{\mathcal{M}} ((\mathcal{M}_{SQ})^*_{S'Q} \leftrightarrow_{\mathcal{M}} (\mathcal{M}_{SQ} \circ Q)) \\
        &= (\mathcal{M}_{SQ})_{S'Q'} \leftrightarrow_{\mathcal{M}} ((\mathcal{M}_{SQ})^*_{S'Q} \leftrightarrow_{\mathcal{M}} ((\mathcal{M}_{SQ})_{S'Q} \circ Q)) \\
        &= (\mathcal{M}_{SQ})_{S'Q'} \leftrightarrow_{\mathcal{M}} ((\mathcal{M}_{SQ})^*_{S'Q} \circ S') \\
        &= (\mathcal{M}_{SQ})_{S'Q'} \leftrightarrow_{\mathcal{M}} ((\mathcal{M}_{SQ})^*_{S'Q'} \circ S') \\
        &= (\mathcal{M}_{SQ})_{S'Q'} \circ Q' = (\mathcal{M}_{SQ} \circ Q)_{Q'},
    \end{align*}
    using Lemma \ref{lem:basecontainedcontains}.
    \item By part 1 and part 3 above we have $(\mathcal{M}^*_{SQ} \leftrightarrow_{\mathcal{M}} \mathcal{M}_{QQ'}^*) \circ S = \mathcal{M}^*_{SQ} \circ S.$ Using Theorem \ref{thm:idtmatroid}, this yields:
    \begin{align*}
        (\mathcal{M}_{SQ} \leftrightarrow_{\mathcal{M}} \mathcal{M}_{QQ'})^* \circ S &= \mathcal{M}^*_{SQ} \circ S \\
        ((\mathcal{M}_{SQ} \leftrightarrow_{\mathcal{M}} \mathcal{M}_{QQ'}) \times S)^* &= (\mathcal{M}_{SQ} \times S)^* \\
        (\mathcal{M}_{SQ} \leftrightarrow_{\mathcal{M}} \mathcal{M}_{QQ'}) \times S &= \mathcal{M}_{SQ} \times S.
    \end{align*}
    Again, by part 1 and part 3 above we have $(\mathcal{M}^*_{SQ} \leftrightarrow_{\mathcal{M}} \mathcal{M}_{QQ'}^*) \circ Q' = (\mathcal{M}^*_{SQ} \circ Q)_{Q'}.$ Using Theorem \ref{thm:idtmatroid}, this yields:
    \begin{align*}
        (\mathcal{M}_{SQ} \leftrightarrow_{\mathcal{M}} \mathcal{M}_{QQ'})^* \circ Q' &= (\mathcal{M}^*_{SQ} \circ Q)_{Q'} \\
        ((\mathcal{M}_{SQ} \leftrightarrow_{\mathcal{M}} \mathcal{M}_{QQ'}) \times Q')^* &= (\mathcal{M}_{SQ} \times Q)_{Q'}^* \\
        (\mathcal{M}_{SQ} \leftrightarrow_{\mathcal{M}} \mathcal{M}_{QQ'}) \times Q' &= (\mathcal{M}_{SQ} \times Q)_{Q'}.
    \end{align*}
    \item Let $b_{S'}$ and $b_{Q'}$ be copies of $b_S,$  $b_Q,$ respectively.
Let $b_{S} \uplus b_{Q'}$ be a base of $(\mathcal{M}_{SQ})_{SQ'}$ so that $b_{S'} \uplus b_Q$ is a base of $(\mathcal{M}_{SQ})_{S'Q}.$  
 We therefore have a base $(S' - b_{S'}) \uplus (Q - b_Q)$
of  $(\mathcal{M}_{SQ})^*_{S'Q}.$ 
    
    Since $b_S \uplus b_Q$ is a base of $\mathcal{M}_{SQ}$ and $(S' - b_{S'}) \uplus (Q - b_Q)$ is a base of $(\mathcal{M}^*_{SQ})_{S'Q},$ we have that $b_S \uplus (S' - b_{S'})$ is a base of $(\mathcal{M}_{SQ} \vee (\mathcal{M}^*_{SQ})_{S'Q}) \times (S \uplus S') = (\mathcal{M}_{SQ} \leftrightarrow_{\mathcal{M}} (\mathcal{M}^*_{SQ})_{S'Q}).$
    
    Now $b_{S'} \uplus b_{Q'}$ is a base of $(\mathcal{M}_{SQ})_{S'Q'}.$ Therefore, $b_S \uplus b_{Q'}$ is a base of $(\mathcal{M}_{SQ})_{S'Q'} \leftrightarrow_{\mathcal{M}}(\mathcal{M}_{SQ} \leftrightarrow_{\mathcal{M}} (\mathcal{M}^*_{SQ})_{S'Q}) = $ $(\mathcal{M}_{SQ} \leftrightarrow_{\mathcal{M}} (\mathcal{M}^*_{SQ})_{S'Q}) \leftrightarrow_{\mathcal{M}} (\mathcal{M}_{SQ})_{S'Q'} = \mathcal{M}_{SQ} \leftrightarrow_{\mathcal{M}} ((\mathcal{M}^*_{SQ})_{S'Q} \leftrightarrow_{\mathcal{M}}(\mathcal{M}_{SQ})_{S'Q'} ) =\mathcal{M}_{SQ} \leftrightarrow_{\mathcal{M}} \mathcal{M}_{QQ'}.$
\end{enumerate}
\end{proof}

\end{document}